\documentclass[reqno,11pt,a4paper]{elsarticle}
\journal{Computer Methods in Applied Mechanics and Engineering}

\usepackage{amssymb,amsmath,graphics,psfrag,graphicx,color,subfigure,arydshln,fancyhdr}

\definecolor{lightblue}{rgb}{0.22,0.45,0.70}
\definecolor{cgray}{rgb}{0.7,0.7,0.7}

\usepackage{hyperref}
\hypersetup{colorlinks=true}

\numberwithin{equation}{section}
\numberwithin{table}{section}
\numberwithin{figure}{section}
\newcommand\bd{\boldsymbol{d}}
\newcommand\bu{\boldsymbol{u}}
\newcommand\bv{\boldsymbol{v}}
\newcommand\bw{\boldsymbol{w}}
\newcommand\bx{\boldsymbol{x}}

\newcommand\bm{\boldsymbol{m}}
\newcommand\nn{\boldsymbol{n}}
\renewcommand\gg{\boldsymbol{g}}
\newcommand\ff{\boldsymbol{f}}
\newcommand\bt{\boldsymbol{t}}
\newcommand\bsigma{\boldsymbol{\sigma}}
\newcommand\beps{\boldsymbol{\epsilon}}

\newcommand\bI{\mathbf{I}}

\newcommand\bH{\mathbf{H}}
\newcommand\bL{\mathbf{L}}
\newcommand\bV{\mathbf{V}}
\newcommand\bW{\mathbf{W}}
\newcommand\rQ{\mathrm{Q}}

\newcommand\rH{\mathrm{H}}
\newcommand\rZ{\mathrm{Z}}
\newcommand\rV{\mathrm{V}}
\newcommand\cT{\mathcal{T}}
\newcommand\cA{\mathcal{A}}
\newcommand\cB{\mathcal{B}}

\newcommand\cF{\mathcal{F}}
\newcommand\cG{\mathcal{G}}
\newcommand\cL{\mathcal{L}}
\newcommand\cH{\mathcal{H}}

\def\esssup{\mathrm{ess\,sup}}

\newtheorem{remark}{Remark}[section]
\newtheorem{lemma}{Lemma}[section]
\newtheorem{theorem}{Theorem}[section]

\newcommand\bdiv{\mathop{\mathbf{div}}\nolimits}
\newcommand\vdiv{\mathop{\mathrm{div}}\nolimits}
\newcommand\vdiva{\mathop{\mathrm{div}_{\mathrm{a}}}\nolimits}

\newcommand\curl{\mathop{\mathrm{curl}}\nolimits}
\newcommand\cero{\boldsymbol{0}}

\newcommand{\dz}{\ensuremath{\, \mathrm{d}z}}

\newcommand{\dr}{\ensuremath{\, \mathrm{d}r}}
\newcommand{\dt}{\ensuremath{\, \mathrm{d}t}}

\newcommand{\uu}{\underline\bu}
\newcommand\cM{\mathcal{M}}
\newcommand\cN{\mathcal{N}}
\newcommand\tfinal{T}

\newcommand{\cred}{}

\newenvironment{proof}{\noindent{\it Proof.}}{\hfill$\square$}

\allowdisplaybreaks

\begin{document}
	\hypersetup{
		linkcolor=lightblue,
		urlcolor=lightblue,
		citecolor=lightblue
		}
	
\begin{frontmatter}	

\title{\textbf{The Biot-Stokes coupling using total pressure: formulation, analysis and application to interfacial flow in the eye}}
\author[monash,sechenov,unach]{Ricardo Ruiz-Baier\corref{cor1}}
		\ead{ricardo.ruizbaier@monash.edu}				
\address[monash]{School of Mathematics, Monash University, 9 Rainforest Walk, Clayton, Victoria 3800, Australia.}		
\address[sechenov]{Institute of Computer Science and Mathematical Modelling, Sechenov University, Moscow, Russian Federation.}
\address[unach]{Universidad Adventista de Chile, Casilla 7-D, Chill\'an, Chile.}	
\cortext[cor1]{Author for correspondence. Email: {\tt ricardo.ruizbaier@monash.edu}. Phone: +61 3 9905 4485.}

\author[bristol]{Matteo Taffetani}
\ead{ho20302@bristol.ac.uk}
\address[bristol]{Department of Engineering Mathematics, University of Bristol, BS8 1TR Bristol, United Kingdom.}

\author[nc]{Hans D. Westermeyer}
\ead{hdwester@ncsu.edu}
\address[nc]{Department of Clinical Sciences, College of Veterinary Medicine, North Carolina State University, 1060 William Moore Drive, Raleigh, NC 27607, USA.}

 \author[pit]{Ivan Yotov}
 \ead{yotov@math.pitt.edu}
 \address[pit]{Department of Mathematics, University of Pittsburgh, Pittsburgh, PA 15260, USA.}

\date{\today}
\begin{abstract}
We consider a multiphysics model for the flow of Newtonian fluid coupled with Biot consolidation equations through an interface, and incorporating total pressure as an unknown in the poroelastic region. A new mixed-primal finite element scheme is proposed solving for the pairs fluid velocity - pressure and displacement - total poroelastic pressure using Stokes-stable elements, and where the formulation does not require Lagrange multipliers to set up the usual transmission conditions on the interface. The stability and well-posedness of the continuous and semi-discrete problems are analysed in detail. Our numerical study is framed \cred{in the context of applicative problems pertaining to heterogeneous  geophysical flows and to eye poromechanics. For the latter, we investigate} different interfacial flow regimes in Cartesian and axisymmetric coordinates that could eventually help describe early morphologic changes associated with glaucoma development in canine species. 
\end{abstract}

\begin{keyword}
Porous media flow \sep Biot consolidation \sep total pressure \sep transmission problem \sep mixed finite element methods \sep eye fluid poromechanics.\\[1ex]

\MSC 65M60 \sep 65M12 \sep 76S05 \sep 74F10 \sep 92C35.
\end{keyword}
\end{frontmatter}


\section{Introduction}
Poroelastic structures are found in many applications of industrial and scientific relevance. Examples include the interaction between soft permeable tissue and blood flow, or the study of {the spatial growth of biofilm in fluids}. When the interaction with a free fluid is considered, {the mechanics of the fluid and poroelastic domains} are coupled through balance of forces and continuity conditions that adopt diverse forms depending on the expected behaviour in the specific application (see, e.g., \cite{dalwadi16,showalter05,taffetani20,tully09} and the references therein). {The particular problem we consider in this paper as motivation for the design of the finite element formulation is the} interfacial flow of aqueous humour between the anterior chamber and the trabecular meshwork (which is a deformable porous structure) in the eye, and how such phenomenon relates to early stages of glaucoma. 

Glaucoma encompasses a group of mechanisms that lead to decreased retinal function, impaired visual fields and blindness. The main risk factor for glaucoma in canines is an abnormal increase in {the} intra-ocular pressure (which under physiologically normal conditions is balanced between aqueous humour production and outflow to the venous drainage system \cite{gum07}). We are interested in modelling the  flow behaviour of aqueous humour within the anterior chamber and its interaction with {the} poroelastic properties of particular compartments in the drainage outlet located between the base of the iris and the limbus, which, in the dog eye and most other non-primate species consists of an array of thin tissue columns (pectinate ligaments) \cite{pearl05} which mark the boundary of the trabecular meshwork with the anterior chamber. Sketches of the regions of interest are depicted in Figure~\ref{fig:second-sketch}. Our focus is on how the physical changes associated with pectinate ligament dysplasia, a change seen in all dogs with primary angle closure glaucoma, affect aqueous humour flow through this boundary. \cred{We stress that the ciliary cleft anatomy of all carnivorous mammals is fairly similar to that of the dog.  In fact, this anatomy is preserved across much of herbivorous mammals as well (see, e.g., \cite{meekins21}). Therefore, even though the dog is probably the most studied due to its status as a companion animal, this work likely applies to most carnivorous and herbivorous mammals.} 

The flow within the anterior chamber will be modelled by Navier-Stokes and Stokes' law for Newtonian fluids, whereas the filtration of aqueous humour through the {deformable} trabecular meshwork and towards the angular aqueous plexus will be described by Darcy's law. Pressure differences are generated by production (from the ciliary muscle) and drainage (to angular aqueous plexus and then linked to the veins at the surface of the sclera through collecting channels) of aqueous humour. 

Other effects that could contribute to modification of the flow patterns and that we do not consider here, are thermal properties (buoyancy mechanisms due to temperature gradients from inner to outer cornea) \cite{bedford86}, cross-link interaction between fibrils in the cornea \cite{gizzi21}, pressure changes due to phacodonesis (vibration of the lens while the head or eye itself moves) and Rapid Eye Movement during sleep \cite{fitt06}, and nonlinear flow conditions in the filtration region (incorporated in \cite{kumar06} through Darcy-Forchheimer models).  

In contrast with \cite{crowder13,ferreira14,martinez19,villamarin12}, here we consider that the coalescing of the pectinate ligaments   results in marked changes in porosity properties of the anterior chamber - trabecular meshwork interface, which  could eventually lead to progressive collapse of the ciliary cleft. We further postulate that these modifications of the tissue's microstructure could be induced by forces exerted by the flow that concentrate at the interface {between} the dysplastic pectinate ligament and the anterior chamber, and which occur over a timescale much larger than that of the ocular pulsating flow. In fact, evidence of the compliance of the trabecular meshwork can be found in, e.g., \cite{johnstone04}. One of the earliest modelling works including a coupling between aqueous humour in the anterior chamber with complying structures {is presented} in \cite{heys01}, where mechanical properties of the bovine iris were employed to set an elastic interface to represent blinking. Other fluid-structure interaction models have been recently developed in \cite{zhang18}, suggesting that flow conditions in the trabecular meshwork and the outlets could be largely affected by the changes of permeability in microstructure{,} and \cite{aletti16}, where poroelastic properties of the choroid and viscoelastic response of the vitreous body are used to set up a more complete 3D model of larger scale that discards a dedicated physiological description of the trabecular meshwork and considers instead a windkessel model.

\begin{figure}[t]
\begin{center}
{\includegraphics[height=0.2\textwidth]{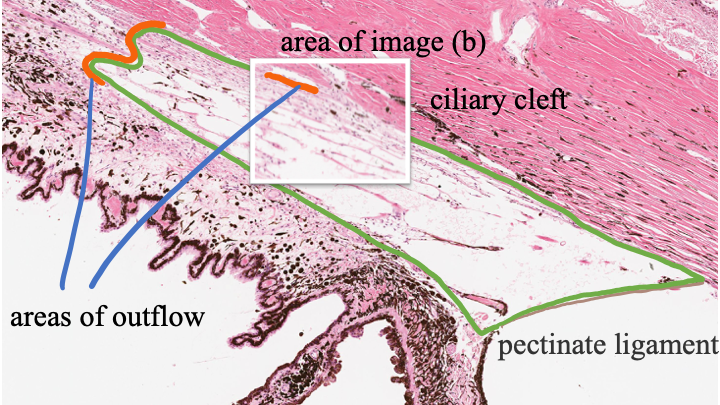}}\ 
{\includegraphics[height=0.2\textwidth]{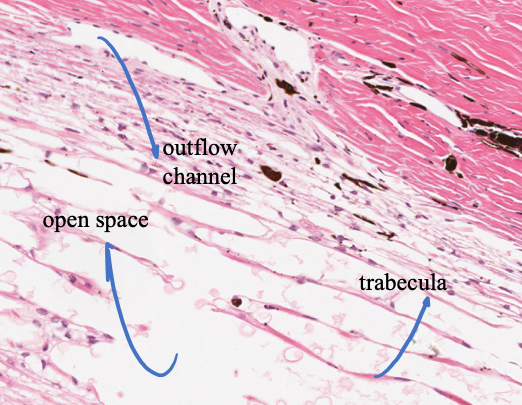}}\  
{\includegraphics[height=0.2\textwidth]{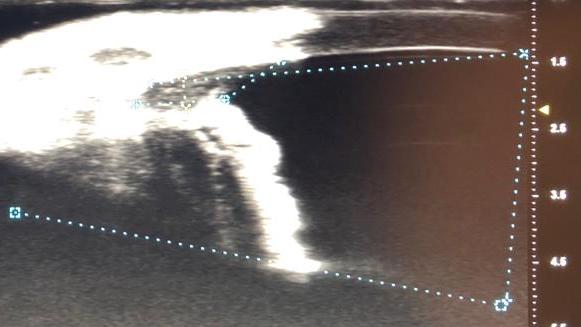}}
\end{center}

\vspace{-3mm}
\caption{Histology sketches of regions of interest (left, centre), and ultrasound image taken from data (right).}\label{fig:second-sketch}
\end{figure}

In the general context of {single phase fluid} / poromechanical coupling, there are already a variety of finite element formulations starting from the work \cite{murad01}, which focuses on the effects of secondary consolidation. 
\cred{More recently, partitioned finite element formulations using domain decomposition and or Nitsche's approach for single and double poroelastic layers in contact with {a single phase fluid}  can be found in \cite{badia09,bukac15,bukac15b,kunwar20}.
Monolithic couplings have been analysed in \cite{ambar18} for a mixed Darcy formulation using a Lagrange multiplier to impose flux continuity (see also \cite{ambar20} for the extension to the case of non-Newtonian fluids), as well in \cite{cesme17,cesme20} for a primal Darcy formulation. Ghost penalty methods have been employed for cut FEM methods valid in the regime of large deformations in \cite{ager19}.} 

Here, and drawing inspiration from the formulation in \cite{ambar18} and \cite{deoliveira20,lee17,oyarzua16}, \cred{we rewrite the poroelasticity equations using three fields (displacement, fluid pressure and total pressure). Compared to previous works \cite{ambar18,ambar20,badia09,bukac15,bukac15b,cesme17,cesme20,kunwar20}, which employ the classical displacement formulation, an advantage of the present approach, inherited from \cite{oyarzua16}, is that the formulation is free of poroelastic locking, meaning that it is robust with respect to the Lam\'e parameters of the poroelastic structure. This is of particular importance when we test variations of the flow response to changes in the material properties of the skeleton and when the solid approaches the incompressibility limit. 
The present work also stands as an extension of the formulation recently employed in \cite{taffetani20} (where only the case of intrinsic incompressible constituent in the poroelastic region were considered) to obtain approximate solutions for heterogeneous poroelasticity coupled with Stokes flow in channels (and using also heterogeneous elastic moduli); while the PDE analysis, numerical aspects, and applicability of the formalism to more realistic scenarios have not yet been addressed. In this work, under adequate assumptions, the analysis of the weak formulation is carried out, using a (time continuous) semi-discrete Galerkin approximation and a weak compactness argument. The well-posedness of the semi-discrete formulation is established using the theory of differential algebraic equations (see, e.g., \cite{brenan95}) and using similar results to those obtained in \cite{ambar18,yi14}. A conforming mixed finite element scheme of general order  is used. Furthermore, a fully discrete scheme based on backward Euler's time discretisation is considered, and the unique solvability  and convergence for the fully discrete scheme are established.}

We have organised the contents of this paper in the following manner. Section~\ref{sec:model} outlines the model problem, motivating each {term in} the balance equations and stating {the}  interfacial and boundary conditions. Section~\ref{sec:weak} states the \cred{weak} form of the governing equations in Cartesian and axisymmetric coordinates. Then{, in Section~\ref{sec:wellp},} we address the construction of the finite element scheme, the well-posedness of the continuous and discrete problems, the stability of the fully discrete system in matrix form. Section~\ref{sec:error} states the fully-discrete scheme {and presents the error estimates}. In Section~\ref{sec:results} we collect   computational results consisting in verification of spatio-temporal convergence and analysis of different cases on simplified and more physiologically  accurate geometries, \cred{including also a typical application in reservoir modelling. One of the examples involves large displacements near the interface, in which case a harmonic extension operator is used to deform the fluid domain}. We close with a summary, some remarks and a discussion on model generalisations in Section~\ref{sec:concl}.

\section{Governing equations}\label{sec:model}
Let us consider a spatial domain $\Omega\subset \mathbb{R}^d$, $d=2,3$ disjointly split into $\Omega_F$ and $\Omega_P$ representing, respectively, the regions where a chamber filled with incompressible fluid and the deformable porous structure are located. We will denote by $\nn$ the unit normal vector on the boundary $\partial\Omega$, and by $\Sigma=\Omega_F\cap \Omega_P$ the interface between the two subdomains. We also define the boundaries $\Gamma_F = \partial\Omega_F\setminus\Sigma$ and {$\Gamma_P = \partial\Omega_P \setminus \Sigma$}, and adopt the convention that on $\Sigma$ the normal vector points from $\Omega_F$ to $\Omega_P$. See a rough sketch in Figure~\ref{fig:sketch}, {that represents} the geometry of the anterior segment in the eye distinguishing between the anterior chamber $\Omega_F$ and the trabecular meshwork $\Omega_P$. The domain is sketched as an axisymmetric region, for which more specific properties will be listed later on.

In presenting the set of governing equations for the coupled fluid - poroelastic system we first focus on the fluid domain, then on the poroelastic domain and, finally, on the initial, boundary and interfacial conditions.

\begin{figure}[t!]
	\begin{center}
		\includegraphics[height=0.33\textwidth]{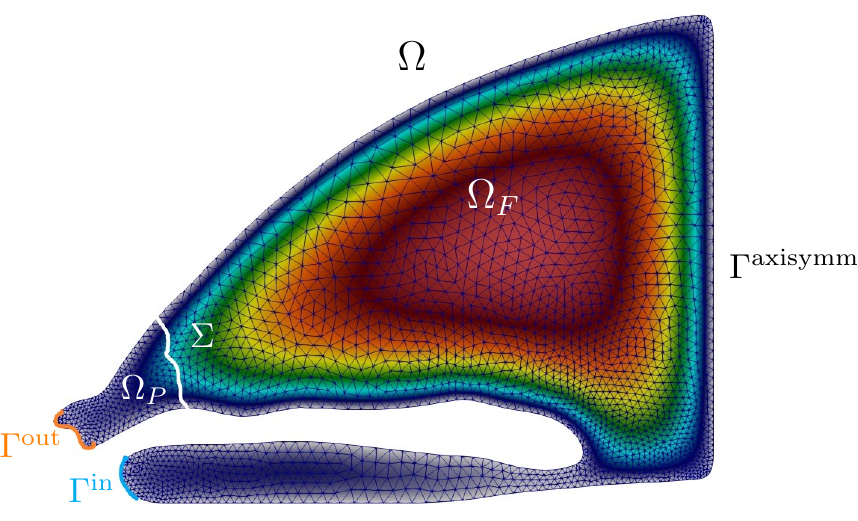}
	\end{center}
	\caption{Schematic diagram of multidomain configuration on a segmented and meshed geometry, including the location of boundaries and interface. The inlet region $\Gamma^{\mathrm{in}}$ and the symmetry axis $\Gamma^{\mathrm{axisymm}}$ are part of the boundary $\Gamma^{\bu}_F$, whereas the outlet region $\Gamma^{\mathrm{out}}$ is part of the boundary $\Gamma^{\bd}_P$.}\label{fig:sketch}
\end{figure}
\subsection{Fluid domain}
In the fluid domain $\Omega_F$, the problem is governed by the momentum and mass conservation equations. Defining the fluid velocity $\bu$ and the fluid pressure $p_F$, the resulting system is written as 
\begin{subequations}
	\begin{align}
	\label{eq:momentumA}
	\rho_f(\partial_t \bu + \bu\cdot\nabla\bu) -\bdiv[2\mu_f \beps(\bu) - p_F\bI] & = \rho_f \gg & \text{in $\Omega_F\times(0,\tfinal]$},\\ 
	\label{eq:massA}
	\vdiv \bu & = 0 & \text{in $\Omega_F\times(0,\tfinal]$},
	\end{align}
\end{subequations}
where $\rho_f,\mu_f$ are  the density and dynamic viscosity of the fluid (e.g. the aqueous humour if we think to the application for the filtration in the eye and thus $\Omega_F$ is the anterior chamber), 
$\gg$ is the gravity acceleration, $\beps(\bu) = \frac{1}{2}(\nabla\bu+\nabla\bu^{\tt t})$ 
is the strain rate tensor and $\partial_t$ indicates derivatives with respect to time.
\cred{Note that, in typical flow conditions of the eye anterior chamber, the Reynolds number is low (approximately 1.25, while the reduced Reynolds number is between  0.06-0.08. See, e.g., \cite{cannizzo17,fitt06,johnstone04,martinez19}). Therefore, 
 for the stability analysis and for some of the numerical tests later on,} we will restrict the fluid model to Stokes' equations. 

\subsection{Poroelastic domain}
The poroelastic domain $\Omega_P$ is a biphasic material constituted by a linear elastic solid phase (potentially intrinsically compressible) and an intrinsically incompressible fluid phase. In the context of the eye poromechanics, the trabecular meshwork region occupying $\Omega_P$ is constituted by three distinctive tissues, the uveal meshwork, the juxtacanalicular meshwork, and the corneoscleral network; {they have different micromechanical properties that, from our modelling perspective, can be regarded as a single poroelastic domain with heterogeneous porosity distribution that, in turn, means possible heterogeneity in the material properties.}  In addition, we anticipate that, although the fluid viscosity is relevant at the scale of the pore, we assume that the fluid can be treated as inviscid at the macroscale. Calling $p_P$ the fluid pressure and $\bd$ the solid displacement, here we introduce the formulation that assumes pressure and displacement as primary variables as presented, for example, in \cite{cowin2007} 
\begin{subequations}\label{eq:massmomentumPD}
	\begin{align}\label{eq:massPD}
	\partial_t  \left(C_0 p_P + \alpha \vdiv \bd\right) - \vdiv \biggl({\frac{\kappa}{\mu_f} (\nabla p_P - \rho_f \gg)}\biggr) & = 0  & \text{in $\Omega_P\times(0,\tfinal]$},\\
	\label{eq:momentumPD}
	-\bdiv[2\mu_s\beps(\bd) + \lambda \left(\vdiv\bd\right) \bI - \alpha p_P \bI ]& = \rho_m\ff & \text{in $\Omega_P\times(0,\tfinal]$}.
	\end{align}
\end{subequations}
The first equation can be derived from the conservation of mass for the fluid phase once employing the Darcy's law  and \cred{the relation} between fluid content - pressure - hydrostatic deformation of the solid phase. The storage capacity $C_0$ is related to the intrinsic compressibility of the solid phase, while $\kappa$ is the permeability (assumed isotropic but heterogeneous).
The second equation {is the conservation of the momentum for the mixture}, where $\ff$ is a (possibly fluid pressure-dependent) body load, $\beps(\bd) = \frac{1}{2}(\nabla\bd+\nabla\bd^{\tt t})$ is the 
infinitesimal strain tensor, $\rho_s$ is the density of the 
porous matrix and $\rho_m$ is the average density of the poroelastic body, $\lambda,\mu_s$ are the Lam\'e constants of the solid; the term in the divergence in the left-hand side of \eqref{eq:momentumPD} is \cred{known} as the effective stress or Terzaghi stress and the parameter $\alpha$, also known as Biot-Willis poroelastic coefficient, depends on the intrinsic compressibility of the solid phase ($\alpha =1$ when the solid phase is intrinsically incompressible). We refer an interested reader to \cite{cowin2007} or \cite{coussy2004} for further details.

For sake of robustness of the formulation with respect to $\lambda$, we introduce the \emph{total pressure} $\varphi: = \alpha p_P - \lambda\vdiv\bd$, as an additional unknown in the system (following \cite{lee17,oyarzua16}), and rewrite the pressure-displacement formulation in \eqref{eq:massmomentumPD} in term of the solid displacement $\bd$, the fluid 
pressure $p_P$, and the total pressure $\varphi$,  as 
\begin{subequations}
	\begin{align}
	\label{eq:momentumM}
	-\bdiv[2\mu_s\beps(\bd) -\varphi \bI ]& = \rho_m\ff & \text{in $\Omega_P\times(0,\tfinal]$},\\
	\label{eq:phi}
	\varphi - \alpha p_P + \lambda \vdiv\bd & = 0 & \text{in $\Omega_P\times(0,\tfinal]$},\\
	\label{eq:massM}
	\bigl(C_0 + \frac{\alpha^2}{\lambda}\bigr)\partial_t p_P - \frac{\alpha}{\lambda} \partial_t \varphi  - \vdiv  \biggl({\frac{\kappa}{\mu_f} (\nabla p_P - \rho_f \gg)}\biggr)
	& = 0 & \text{in $\Omega_P\times(0,\tfinal]$}.
	\end{align}\end{subequations}
In contrast with the formulations in \cite{bukac15,ambar18}, here we do not employ the fluid {velocity} in the porous domain as a separate unknown. Consistently with \cite{macminn2016}, in the framework of linear poroelasticity, the permeability depends on the pressure and the displacement fields only at higher orders; however, it can be heterogeneous, with a spatial distribution dependent, for example, on an initial porosity distribution. We thus simply write 
\begin{equation}\label{eq:generic-kappa}
\kappa = \kappa(\bx). 
\end{equation}
{Likewise, it is also possible to assume heterogeneity of the Lam\'e constants, as in \cite{taffetani20}.} In such a case, we need to assume that there exist constants $\lambda_{\min}$ and
  $\lambda_{\max}$ such that $0 < \lambda_{\min}  \le \lambda \le \lambda_{\max} $. In the analysis we show that the estimates are independent of $\lambda_{\max}$, i.e., the results are uniform in the almost incompressible limit, even for the heterogeneous case. \cred{Heterogeneous permeability and/or heterogeneous Lam\'e parameters will be used in some of the numerical examples of Section~\ref{sec:results}.}

\subsection{Initial, boundary and transmission conditions}
To close the system composed by \eqref{eq:momentumA}, \eqref{eq:massA},  \eqref{eq:momentumPD}, \eqref{eq:massPD}, we need to provide suitable initial data, boundary conditions, and adequate transmission assumptions. 
Without losing generality, we suppose that
\begin{equation}\label{eq:initial}
\bu(0) = \bu_0, \quad p_P(0) = p_{P,0} \quad \text{in $\Omega\times\{0\}$,}
\end{equation}
while for the rest of the variables we will construct compatible initial data. In particular, and for sake of the energy estimates to be addressed in Section~\ref{sec:stability}, we will require initial displacement and an initial total pressure (which in turn is computed from the initial displacement and the initial fluid pressure).  

On the boundary $\Gamma_F$ of the fluid domain we can apply conditions on either the velocity, or the stress tensor; we thus decomposed it between $\Gamma_F^{\bu}$ and $\Gamma_F^{\bsigma}$, 
with $|\Gamma_F^{\bu}|>0$,  
 where we impose, respectively, no slip velocities and zero normal total stresses as
\begin{subequations}
	\begin{align}\label{eq:bcAD}
	\bu & = \cero & \text{on $\Gamma_F^{\bu}\times(0,\tfinal]$},\\
	\label{eq:bcAN}
	[2\mu_f\beps(\bu)-p_F\bI]\nn & = \cero& \text{on $\Gamma_F^{\bsigma}\times(0,\tfinal]$}.
	\end{align}
\end{subequations}
Similarly, on the boundary $\Gamma_P$ of the poroelastic domain we can prescribe conditions on either the displacement or the traction and either the pressure or the fluid flux $\mathbf{q} = - \frac{\kappa}{\mu_f} (\nabla p_P - \rho_f \gg)$. We thus divide the boundary into $\Gamma_P^{p_P}$ and $\Gamma_P^{\bd}$, with $|\Gamma_P^{p_P}|>0$ and $|\Gamma_P^{\bd}|>0$,   where we apply, respectively,
\begin{subequations}
	\begin{align}
	\label{eq:bcMd}
	\bd = \cero\quad \text{and} \quad \frac{\kappa}{\mu_f} (\nabla p_P - \rho_f \gg) \cdot\nn &= 0 &\text{on $\Gamma_P^{\bd}\times(0,\tfinal]$},\\
	\label{eq:bcMu}
	[2\mu_s\beps(\bd) -\varphi \bI]\nn = \cero \quad\text{and}\quad p_P&={0}  &\text{on $\Gamma_P^{p_P}\times(0,\tfinal]$}.
	\end{align}
\end{subequations}
Next, and following \cite{badia09,bukac15,murad01,showalter05}, we consider transmission conditions on $\Sigma$ accounting for the continuity of normal 
fluxes, momentum conservation, balance of fluid normal stresses, and the so-called Beavers-Joseph-Saffman {(BJS)} condition for tangential fluid forces 
\begin{subequations}
	\begin{align}
	\label{eq:inter-U}
	\bu\cdot\nn & = \left(\partial_t \bd  - \frac{\kappa}{\mu_f} (\nabla p_P - \rho_f \gg) \right)\cdot \nn & \text{on $\Sigma\times(0,\tfinal]$},\\
	\label{eq:inter-sigma}
	(2\mu_f \beps(\bu) - p_F\bI)\nn & = (2\mu_s\beps(\bd) -\varphi \bI ) \nn & \text{on $\Sigma\times(0,\tfinal]$},\\
	\label{eq:inter-sigmaf}
	-\nn\cdot (2\mu_f \beps(\bu) - p_F\bI)\nn & =  \cred{p_P} & \text{on $\Sigma\times(0,\tfinal]$},\\
	\label{eq:inter-bjs}
	- {\bt^j}\cdot (2\mu_f \beps(\bu) - p_F\bI)\nn & = \frac{\gamma\mu_f}{\sqrt{\kappa}} (\bu - \partial_t \bd)\cdot {\bt^j,\quad 1\leq j \leq d-1}  & \text{on $\Sigma\times(0,\tfinal]$},
	\end{align}\end{subequations}
where $\gamma>0$ is the slip rate coefficient (or tangential resistance parameter), and we recall that the normal $\nn$ on the interface is understood as pointing from the fluid domain $\Omega_F$ towards the porous structure $\Omega_P$, while {$\bt^1$ stands for the 
tangent vector on $\Sigma$ (for the case of $d=2$, while for 3D $\bt^1,\bt^2$ represent} the two tangent vectors on the interface, normal to $\nn$). 

\section{Weak formulation}\label{sec:weak}
Apart from the nomenclature introduced at the beginning of the section, conventional notation will be adopted throughout the paper. For Lipschitz domains $\Xi$ in $\mathbb{R}^d$, and for $s \in \mathbb{N}$, $k\in \mathbb{N}\cup \infty$ we denote by $W^{k,s}(\Xi)$ the space of all $L^s(\Xi)$ integrable functions with weak derivatives up to order $s$ being also   $L^s(\Xi)$ integrable. As usual, for the special case of $s=2$ we write $H^s(\Xi):=W^{k,2}(\Xi)$ and use boldfaces to refer to vector-valued functions and function spaces, e.g., $\bH^s(\Xi):=[H^{s}(\Xi)]^d$. We will further utilise the Bochner space-time norms, for a separable Banach space $\rV$ and
$f:(0,\tfinal) \to \rV$,
$\|f\|^2_{L^2(0,\tfinal;\rV)} := \int^{\tfinal}_0 \|f(t)\|^2_{\rV} \,\dt$
and
$\|f\|_{L^\infty(0,\tfinal;\rV)} := \mathop{\esssup}\limits_{t\in (0,\tfinal)} \|f(t)\|_{\rV}$.
By $C$ we will denote generic constants that are independent of the mesh size. 

\subsection{Cartesian case}
The initial step in deriving the finite element scheme consists in stating a weak form for \eqref{eq:momentumA}-\eqref{eq:massM}. We proceed to 
test these equations against suitable smooth functions and to integrate over the corresponding subdomain. After applying 
integration by parts wherever adequate, 
 we formally end up with the following remainder on the interface 
\[I_\Sigma = - \langle (2\mu_f \beps(\bu) - p_F\bI) \nn, \bv \rangle_\Sigma +  \langle (2\mu_s\beps(\bd) -\varphi \bI ) \nn, \bw \rangle_\Sigma +\langle  \frac{\kappa}{\mu_f} \nabla p_P \cdot \nn, q_P \rangle_\Sigma,\] 
where $\langle\cdot,\cdot\rangle_\Sigma$ denotes the \cred{duality pairing} between the trace functional space space $H^{1/2}(\Sigma)$ and its dual $H^{-1/2}(\Sigma)$. Then, as in, e.g., \cite{karper09}, we proceed to use each of the transmission conditions \eqref{eq:inter-U}-\eqref{eq:inter-bjs}, yielding the expression 
\[I_\Sigma = \langle \cred{ p_P}, (\bv - \bw)\cdot \nn \rangle_\Sigma + {\sum_{j=1}^{d-1}}\langle \frac{\gamma\mu_f}{\sqrt{\kappa}} (\bu - \partial_t \bd)\cdot {\bt^j},  (\bv-\bw)\cdot {\bt^j}\rangle_\Sigma -\langle (\bu-\partial_t\bd)\cdot \nn, q_P \rangle_\Sigma.\]
This 
interfacial term is well-defined because of the regularity of the entities involved, and this implies that we do not require additional Lagrange multipliers to realise the coupling conditions. Also, in view of the boundary conditions we  
define the Hilbert spaces 
\begin{gather*}
\bH^1_\star(\Omega_F) = \{ \bv\in \bH^1(\Omega_F): \bv|_{\Gamma_F^{\bu}} = \cero\}, \quad 
\bH^1_\star(\Omega_P) = \{ \bw\in \bH^1(\Omega_P): \bw|_{\Gamma_P^{\bd}} = \cero\},\\
 H^1_\star(\Omega_P) = \{ q_P\in H^1(\Omega_P): q_P|_{\Gamma_P^{p_P}} = 0\},
\end{gather*}
associated with the classical norms in $\bH^1(\Omega_F),\bH^1(\Omega_P)$, and $H^1(\Omega_P)$,
respectively. Consequently we have the following 
mixed \cred{weak} form: For $t \in [0,\tfinal]$, find $\bu\in\bH^1_{\star}(\Omega_F)$, $p_F\in L^2(\Omega_F)$,  $\bd\in\bH^1_{\star}(\Omega_P)$, $p_P\in H^1_{\star}(\Omega_P)$, $\varphi\in L^2(\Omega_P)$, such that 
\begin{subequations}\label{eq:weak} 
\begin{align}
 & a_1^F(\partial_t\bu,\bv) + a_2^F(\bu,\bv) + c^F(\bu,\bu;\bv) + b^F_1(\bv,p_F) \nonumber \\ 
& \qquad\quad + \cred{ b^\Sigma_2(\bv,p_P)} + b_3^\Sigma(\bv,\partial_t\bd)  = F^F(\bv) &\forall \bv \in \bH^1_{\star}(\Omega_F), \label{weak-1}\\
& - b^F_1(\bu,q_F) = 0 & \forall q_F\in  L^2(\Omega_F), \label{weak-2}\\
& b_3^\Sigma(\bu,\bw)  + \cred{b_4^\Sigma(\bw,p_P)} + a_1^P(\bd,\bw) \nonumber \\
& \qquad\quad  + a_2^\Sigma(\partial_t\bd,\bw)
+ b_1^P(\bw,\varphi) = F^P(\bw) &\forall \bw \in \bH^1_{\star}(\Omega_P), \label{weak-3}\\
& - b_2^\Sigma(\bu,q_P)  - b_4^\Sigma(\partial_t \bd,q_P)
+ a_3^P(\partial_t p_P,q_P) \nonumber \\
& \qquad\quad + a_4^P(p_P,q_P)
- b_2^P(\partial_t \varphi,q_P)  = G(q_P) &\forall q_P \in H^1_{\star}(\Omega_P), \label{weak-4}\\
& - b_1^P(\bd,\psi) -  b_2^P(\psi,p_P) + a_5^P(\varphi,\psi)  = 0 & \forall \psi\in L^2(\Omega_P),
\label{weak-5}
\end{align}\end{subequations}
 where the \cred{bilinear and trilinear} forms and linear functionals are defined as 
 \begin{gather}
   a_1^F\!(\bu,\bv) = \rho_f\!\int_{\Omega_F}\!\!\! \bu\cdot\bv, \quad 
a_2^F\!(\bu,\bv) = 2\mu_f\!\int_{\Omega_F}\! \!\beps(\bu):\beps(\bv) +
 {\sum_{j=1}^{d-1}} \langle  \frac{\gamma\mu_f}{\sqrt{\kappa}} \bu\cdot{\bt^j},\bv\cdot{\bt^j}\rangle_\Sigma, \nonumber \\  
 c^F\!(\bu,\bw; \bv) =  \rho_f\!\!\int_{\Omega_F}\!\! \!(\bu\cdot\nabla\bw )\cdot \bv, 
 \quad 
  b^F_1(\bv,q_F) = -\int_{\Omega_F} q_F\vdiv\bv,\quad  b_1^P(\bw,\psi) =  -\int_{\Omega_P} \psi\vdiv\bw,
 \nonumber\\
 b_2^\Sigma(\bv,q_P) = \langle q_P, \bv\cdot\nn\rangle_\Sigma, \quad 
 b_3^\Sigma(\bv,\bw) = - {\sum_{j=1}^{d-1}}\langle  \frac{\gamma\mu_f }{\sqrt{\kappa}} \bv \cdot {\bt^j}, \bw\cdot{\bt^j}\rangle_\Sigma,\quad b_4^\Sigma(\bw,q_P) =  -\langle q_P, \bw\cdot\nn\rangle_\Sigma, \nonumber\\
  a_1^P(\bd,\bw) = 2\mu_s\int_{\Omega_P} \beps(\bd):\beps(\bw), \quad
 a_2^\Sigma(\bd,\bw) = {\sum_{j=1}^{d-1}} \langle  \frac{\gamma \mu_f}{\sqrt{\kappa}} \bd \cdot {\bt^j}, \bw\cdot{\bt^j}\rangle_\Sigma, \label{eq:forms}\\
  a_3^P(p_P,q_P) =  \bigl(C_0 + \frac{\alpha^2}{\lambda}\bigr)\int_{\Omega_P} p_P q_P, \quad
  a_4^P(p_P,q_P) =  \int_{\Omega_P} \frac{\kappa}{\mu_f}\nabla p_P \cdot\nabla q_P,\nonumber\\
 b_2^P(\psi,q_P) = \frac{\alpha}{\lambda} \int_{\Omega_P} \psi q_P,\quad 
 a_5^P(\varphi,\psi) =  \frac{1}{\lambda}\int_{\Omega_P} \varphi \psi, \quad 
 F^F(\bv) = \rho_f\int_{\Omega_F} \gg\cdot\bv, \nonumber\\ 
 F^P(\bw) = \rho_s\int_{\Omega_P} \ff\cdot\bw, \quad 
 G(q_P) = \int_{\Omega_P} \rho_f\frac{\kappa}{\mu_f}\gg\cdot\nabla q_P
 - \langle \rho_f\frac{\kappa}{\mu_f}\gg\cdot\nn,q_P\rangle_\Sigma.\nonumber
   \end{gather}

\subsection{Axisymmetric case}
For the specific application of interfacial flow in the eye, the radial symmetry of the domain and of the flow conditions could be better represented using axisymmetric formulations as in \cite{anaya19,crowder13,martinez19}. Then, the domain as well as the expected flow properties are all symmetric with  
	respect to the axis of symmetry $\Gamma^{\mathrm{axisymm}}$. The model equations 
	can be written in the meridional domain $\Omega$ (making abuse of notation, and referring to  Figure~\ref{fig:sketch}).  In such a setting 
	the fluid velocity and solid displacement only possess  radial and vertical 
	components and we recall that the divergence operator of the generic vector field $\bv$ in axisymmetric coordinates (in radial and height variables $r,z$) is 
	\[\vdiva \bv := \partial_z v_z+\frac{1}{r}\partial_r(rv_r),
	\]
while the notation of the gradient coincides with that in Cartesian coordinates. The \cred{weak} formulation \eqref{eq:weak} adopts the following mo\-di\-fications (again making abuse of notation, the unknowns 
are denoted the same as in the Cartesian case): Find $\bu\in\widehat{\bV}$, $p_F\in \widehat{\rQ}^F$,  $\bd\in\widehat{\bW}$, $p_P\in \widehat{\rQ}^P$, $\varphi\in \widehat{\rZ}$, such that 
\begin{subequations}\label{eq:axisym}
\begin{align}
& \rho_f\int_{\Omega_F}\! \partial_t \bu\cdot\bv\, r\dr\dz +2\mu_f\int_{\Omega_F}\! \beps(\bu):\beps(\bv) r\dr\dz +2\mu_f\int_{\Omega_F}\! \frac{1}{r}u_rv_r \dr\dz  -\int_{\Omega_F} p_F\vdiva\bv \, r\dr\dz  \nonumber \\
& \quad +\int_{\Sigma}  \frac{\gamma\mu_f}{\sqrt{\kappa}} (\bu-\partial_t\bd)\!\cdot\bt\bv\cdot\bt  r\dr\dz +  \!\! \cred{\int_{\Sigma} p_F \bv\cdot\nn r \dr\dz} = \rho_f\!\!\int_{\Omega_F}\! \gg\cdot\bv r \dr\dz\quad \forall \bv\in\widehat{\bV},\\
& -\int_{\Omega_F} q_F\vdiva\bu \, r\dr\dz  = 0 \qquad \forall q_F\in\widehat{\rQ}^F, \\
& 2\mu_s\int_{\Omega_P} \beps(\bd):\beps(\bw)\, r\dr\dz  +2\mu_s\int_{\Omega_P}\! \frac{1}{r}d_rw_r \dr\dz  - \int_{\Sigma} \frac{\gamma \mu_f}{\sqrt{\kappa}} (\bu-\partial_t \bd) \cdot \bt \bw\cdot\bt r\dr\dz \nonumber\\ 
& \qquad  -\int_{\Omega_P} \varphi\vdiva\bw \, r\dr\dz -\cred{\int_{\Sigma} p_P \bw\cdot\nn \dr\dz} = \rho_s\int_{\Omega_P} \ff\cdot\bw \, r \dr\dz \qquad \forall \bw\in\widehat{\bW},\\
&  \bigl(C_0 + \frac{\alpha^2}{\lambda}\bigr)\int_{\Omega_P} \partial_t p_P q_P \, r\dr\dz  +  \int_{\Omega_P} \frac{\kappa}{\mu_f}\nabla p_P \cdot\nabla q_P \, r\dr\dz 
- \frac{\alpha}{\lambda} \int_{\Omega_P} \partial_t\varphi q_P \, r\dr\dz \nonumber\\
& \quad + \int_{\Sigma} q_P (\bu - \partial_t\bd) \cdot\nn r\dr\dz  = \int_{\Omega_P}\!\! \rho_f\gg\cdot\nabla q_P  r\dr\dz - \int_\Sigma \!\rho_f\gg\cdot\nn q_P r \dr\dz\quad \forall q_P\in\widehat{\rQ}^P,\\
&  -\int_{\Omega_F} \psi\vdiva\bv \, r\dr\dz + \frac{\alpha}{\lambda} \int_{\Omega_P} p_P\psi \, r\dr\dz 
-\frac{1}{\lambda} \int_{\Omega_P} \varphi\psi \, r\dr\dz = 0\qquad \forall \psi\in \widehat{\rZ}.
 \end{align}
 \end{subequations}
 
Here the functional spaces are now defined as 
\begin{gather*}
\widehat{\bV} := \{ \bv \in V_1^1(\Omega_F)\times \rH_1^1(\Omega_F): \bv|_{\Gamma_F^{\bu}} = \cero \  \text{and}\ \bv\cdot\nn|_{\Gamma^{\text{axisymm}}} = 0\},\quad 
\widehat{\rQ}^F := L^2_1(\Omega_F), \\ 
\widehat{\bW} := \{ \bw \in V_1^1(\Omega_P)\times \rH_1^1(\Omega_P): \bw|_{\Gamma_F^{\bd}} = \cero\},\ 
\widehat{\rQ}^P := \{ q_P \in \rH_1^1(\Omega_P) : q_P|_{\Gamma_P^{p_P}} = 0\},\ 
\widehat{\rZ} := L^2_1(\Omega_P),
\end{gather*}
where, for $m\in \mathbb{R}$, $1\leq p <\infty$ the weighted functional spaces adopt the specification 
\[L_m^p(\Omega_i) = \{ v: \|v\|^p_{m,p,\Omega_i}:=\int_{\Omega_i} |v|^p r^{m} \dr\dz< \infty\},\]
and
\[ \rH_1^1(\Omega_i) := \{ v \in L_1^2(\Omega_i): \nabla v \in \bL_1^2(\Omega_i)\}, \quad 
V_1^1(\Omega_i) := \rH_1^1(\Omega_i) \cap L_{-1}^2(\Omega_i).\]
 
\section{Well-posedness of the weak formulation}\label{sec:wellp}
The following analysis is confined to the Cartesian case. Furthermore, we
focus on the quasi-static Biot-Stokes model, i.e., we neglect the
terms $a_1^F$ and $c^F$ in \eqref{weak-1}, as this is the typical flow
regime for the application of interest. We also restrict our attention
to the case $\tilde \alpha = 1$. The solvability analysis is based on a Galerkin
argument, where one considers the semi-discrete continuous in time formulation
with a discretisation parameter $h$. We establish that it has a unique solution
and derive stability bounds. Then, owing to a weak compactness argument, we pass to
the limit $h \to 0$ and obtain existence and uniqueness of a weak solution. 

\subsection{Semi-discrete mixed finite element formulation}
In addition to the assumptions stated before on the domain geometry, to avoid additional technicalities, we operate under the condition that $\Omega$ is a polytope. 
We denote by $\{\cT_{h}\}_{h>0}$ a shape-regular family of finite element partitions of 
$\bar\Omega$, conformed by tetrahedra (or triangles 
in 2D) $K$ of diameter $h_K$, with mesh size
$h:=\max\{h_K:  K\in\cT_{h}\}$. 
\cred{The} finite-dimensional subspaces for fluid velocity, fluid pressure, porous displacement, 
porous fluid pressure, and total pressure are \cred{denoted 
 $\bV_h,\rQ^F_h,\bW_h,\rQ^P_h,\rZ_h$, respectively. It is sufficient for the following analysis that the pairs $(\bV_h,\rQ^F_h)$ and $(\bW_h,\rZ_h)$ are Stokes inf-sup stable, in the sense that 
  there exist positive constants $\beta_F$ and $\beta_P$ independent of $h$ such that
\begin{subequations}
\begin{align}
  & \forall \, q_{F,h} \in \rQ^F_h, \quad \sup_{\cero \ne \bv_h \in \bV_h}
  \frac{b_1^F(\bv_h,q_{F,h})}{\|\bv_h\|_{\bH^1(\Omega_F)}} \ge \beta_F \|q_{F,h}\|_{L^2(\Omega_F)},
  \label{inf-sup-stokes} \\
  & \forall \, \psi_h \in \rZ_h, \quad \sup_{\cero \ne \bw_h \in \bW_h}
  \frac{b_1^P(\bw_h,\psi_h)}{\|\bw_h\|_{\bH^1(\Omega_P)}} \ge \beta_P \|\psi_h\|_{L^2(\Omega_P)}.
  \label{inf-sup-elast} 
\end{align}\end{subequations}
Feasible choices are   Taylor-Hood,   the MINI element, Crouzeix-Raviart, Scott-Vogelius, Guzm\'an-Neilan, Bernardi-Raugel, equal-order stabilised methods (including total pressure projection stabilisation), divergence-conforming stabilised methods, and many others. 
If taking, for example, generalised Taylor-Hood elements of degree $(k,k-1)$ for $\bV_h \times \rQ_h^F$ and $\bW_h \times \rZ_h$, alongside a piecewise continuous and polynomial space of degree $k$ for $\rQ_h^P$, then their interpolation properties will yield a method of overall order $k$ in space when the displacement and velocity errors are measured in the $\bH^1$-norm, the total pressure and Stokes fluid pressure in the $L^2$-norm, and the Biot fluid pressure in the $H^1$-norm.  A more general case, with different polynomial degrees for the different finite element spaces, is addressed in Section~\ref{sec:error-fully}, below. }

We look for
$(\bu_h, p_{F,h}, \bd_h, p_{P,h}, \varphi_h): [0,\tfinal] \to
\bV_h\times \rQ^F_h\times\bW_h\times\rQ^P_h\times\rZ_h=:\mathbf{H}_h$
such that for a.e. $t \in (0,\tfinal]$,
\begin{subequations}\label{eq:method} 
\begin{align}
  a_2^F(\bu_h,\bv_h) + b^F_1(\bv_h,p_{F,h}) 
  + b^\Sigma_2(\bv_h,p_{P,h}) + b_3^\Sigma(\bv_h,\partial_t\bd_h)
 & = F^F(\bv_h) & \! \forall \bv_h \in \bV_h, \label{method-1}\\
 - b^F_1(\bu_h,q_{F,h}) &= 0 & \! \forall q_{F,h}\in \rQ_h^F, \label{method-2} \\
 b_3^\Sigma(\bu_h,\bw_h)  +  b_4^\Sigma(\bw_h,p_{P,h}) + a_1^P(\bd_h,\bw_h) &\nonumber \\
 + a_2^\Sigma(\partial_t\bd_h,\bw_h)
+ b_1^P(\bw_h,\varphi_h) & = F^P(\bw_h) & \! \forall \bw_h \in \bW_h, \label{method-3} \\
 - b_2^\Sigma(\bu_h,q_{P,h})  - b_4^\Sigma(\partial_t \bd_h,q_{P,h})
+ a_3^P(\partial_t p_{P,h},q_{P,h}) &\nonumber \\
+ a_4^P(p_{P,h},q_{P,h}) 
- b_2^P(\partial_t \varphi_h,q_{P,h})& = G(q_{P,h}) & \! \forall q_{P,h} \in \rQ_h^P,
\label{method-4}\\
 - b_1^P(\bd_h,\psi_h) -  b_2^P(\psi_h,p_{P,h}) + a_5^P(\varphi_h,\psi_h) &= 0 &
\! \forall \psi_h \in \rZ_h. \label{method-5}
\end{align}\end{subequations}
This is a system of differential-algebraic equations (DAE) that can be written in an
operator form as
\begin{equation}\label{eq:semi-discrete}
\partial_t \cN \uu_h(t) + \cM\uu_h(t) = \cF(t), 
\end{equation}
with $\uu_h:=[ \bu_h \ p_{F,h} \ \bd_h \ p_{P,h} \ \varphi_h]^{\tt t}$,
\begin{gather*}
\cN  = \left[\begin{array}{cc:ccc}
\cero & \cero & (\cB_3^\Sigma)' & \cero & \cero \\[1ex]
\cero & \cero & \cero & \cero &\cero \\
\hdashline
\cero & \cero & \cA_2^\Sigma & \cero & \cero \vphantom{\int^{X^X}} \\[1ex]
\cero & \cero & -\cB_4^\Sigma  & \cA_3^P & -(\cB_2^P)' \\[1ex]
\cero & \cero & \cero & \cero & \cero 
  \end{array}\right]\!,  \quad 
\cM = \left[\begin{array}{cc:ccc}
\cA_2^F & (\cB_1^F)' & \cero & (\cB_2^\Sigma)' & \cero \\[1ex]
-\cB_1^F & \cero & \cero & \cero &\cero \\
\hdashline
\cB_3^\Sigma & \cero& \cA_1^P & (\cB_4^\Sigma)' & (\cB_1^P)' \vphantom{\int^{X^X}}\\[1ex]
-\cB_2^\Sigma & \cero & \cero  &\cA_4^P & \cero \\[1ex]
\cero & \cero & -\cB_1^P & -\cB_2^P & \cA_5^P 
  \end{array}\right]\!,  \quad 
\cF = \left[ \begin{array}{c} \cF^{F} \\ \cero \\ \cF^{P} \\ \cG \\ \cero \end{array} \right],
\end{gather*}
where the operators in calligraphic letters are 
induced by the  forms in \eqref{eq:forms}.
The symbol $(\cdot)'$ denotes the dual operator resulting in the transpose block 
matrix of a given elementary block. 

We next discuss the initial condition for \eqref{eq:semi-discrete}.
Recall from \eqref{eq:initial} that we are given initial data $p_{P,0}$ and note that
we do not use the initial data $\bu_0$, since we are considering the quasi-static Stokes
model. We assume that $p_{P,0} \in H^1_\star(\Omega_p)$ and take $p_{P,h}(0) = p_{P,h,0}$,
where $p_{P,h,0} \in \rQ_h^P$ is the $L^2$-projection of $p_{P,0}$. In addition,
we need initial data $\bd_{h,0} \in \bW_h$ and $\varphi_{h,0} \in Z_h$ such
that $\bd_h(0) = \bd_{h,0}$ and $\varphi_h(0) = \varphi_{h,0}$. The construction
of such data will be discussed in Theorem~\ref{thm:well-posed}.

%

\subsection{Well-posedness of the semi-discrete problem}\label{sec:stability}

The well-posedness of the semi-discrete formulation could be analysed by recasting the 
system as a general parabolic operator with degenerate time derivative, as recently 
proposed in \cite{ambar20} for the interaction of non-Newtonian fluids and poroelastic 
media. This would require, however, to include the solid velocity as a new unknown. Instead,
we study it using the theory of DAE, following the approach
from \cite{ambar18}.

In the forthcoming analysis we will appeal to the Poincar\'e inequality
\begin{equation}\label{Poincare}
  \|\nabla q_P\|_{L^2(\Omega_P)} \ge c_P \|q_P\|_{H^1(\Omega_P)},
  \quad \forall \, q_P \in H^1_\star(\Omega_P),
\end{equation}
Korn's inequality
\begin{equation}\label{Korn}
  \|\beps(\bv)\|_{\bL^2(\Omega_F)} \ge c_K^F \|\bv\|_{\bH^1(\Omega_F)},
  \qquad 
  \|\beps(\bw)\|_{\bL^2(\Omega_P)} \ge c_K^P \|\bw\|_{\bH^1(\Omega_P)},
\end{equation}
for all $\bv \in \bH^1_\star(\Omega_F)$, $\bw \in \bH^1_\star(\Omega_P)$,  
and the trace inequality
\begin{equation}\label{trace}
  \|\bv\|_{\bL^2(\Sigma)} \le C_{\Sigma}^F \|\bv\|_{\bH^1(\Omega_F)},
  \qquad 
  \|q_P\|_{L^2(\Sigma)} \le C_{\Sigma}^P \|q_P\|_{H^1(\Omega_P)}, 
\end{equation}
for all $\bv \in \bH^1(\Omega_F)$, $q_P \in H^1(\Omega_P)$. We also assume that there exist constants $0 < \kappa_1 < \kappa_2 < \infty$ such that
\begin{equation}\label{perm}
  \kappa_1 \le \kappa(\bx) \le \kappa_2 \quad \text{for a.e. } \bx \in \Omega_P.
\end{equation}

\begin{lemma}\label{lem:cont}
The bilinear forms that appear in \eqref{eq:method} are continuous in the spaces
$\bV$, $\rQ^F$, $\bW$, $\rQ^P$, and $\rZ$. If $\ff: [0,\tfinal] \to \bL^2(\Omega_P)$,
then the linear functionals on the right-hand sides are also continuous.
\end{lemma}

\begin{proof}
The statement of the lemma follows from the use of the Cauchy-Schwarz inequality and the trace inequality \eqref{trace}.
\end{proof}

\begin{theorem}\label{thm:well-posed}
  For each $\ff \in H^1(0,\tfinal;\bL^2(\Omega_P))$ and $p_{P,0} \in H^1_\star(\Omega_p)$,
  there exist initial data $\bu_{h,0} \in \bV_h$, $p_{F,h,0} \in \rQ^F_h$,
  $\bd_{h,0} \in \bW_h$, and $\varphi_{h,0} \in Z_h$ such that
  the semi-discrete problem \eqref{eq:semi-discrete} with initial conditions
$p_{P,h}(0) = p_{P,h,0}$, $\bd_h(0) = \bd_{h,0}$, and $\varphi_h(0) = \varphi_{h,0}$
has a unique solution satisfying
\begin{subequations}
\begin{align}
& \|\bu_h\|_{L^2(0,\tfinal;\bH^1(\Omega_F))} + \|p_{F,h}\|_{L^2(0,\tfinal;L^2(\Omega_F))}
  + \|\bd_h\|_{L^\infty(0,\tfinal;\bH^1(\Omega_P))} \nonumber \\
& \qquad
  + \sum_{j=1}^{d-1}\|(\bu_h - \partial_t\bd_h)\cdot\bt^j\|_{L^2(0,\tfinal;L^2(\Sigma))}
+ \|p_{P,h}\|_{L^\infty(0,\tfinal;L^2(\Omega_P))}
+ \|p_{P,h}\|_{L^2(0,\tfinal;H^1(\Omega_P))} \nonumber \\
& \qquad
+ \|\varphi_h\|_{L^2(0,\tfinal;L^2(\Omega_P))}
+ \frac{1}{\sqrt{\lambda}}\|\alpha \, p_{P,h} - \varphi_h\|_{L^\infty(0,\tfinal;L^2(\Omega_P))}
\nonumber \\
& \le C \left( \|\gg\|_{L^2(0,\tfinal;\bL^2(\Omega))} + \|\gg\cdot\nn\|_{L^2(0,\tfinal;L^2(\Sigma))}
+ \|\ff\|_{H^1(0,\tfinal;\bL^2(\Omega_P))} + \|p_{P,0}\|_{H^1(\Omega_P)}  \right),
\label{eq:stability}
\end{align}
and
\begin{align}
& \|\bu_h\|_{L^\infty(0,\tfinal;\bH^1(\Omega_F))} + \|p_{F,h}\|_{L^\infty(0,\tfinal;L^2(\Omega_F))}
  + \|\partial_t \bd_h\|_{L^2(0,\tfinal;\bH^1(\Omega_P))}
  \nonumber \\
  & \qquad
  + \sum_{j=1}^{d-1}\|(\bu_h - \partial_t\bd_h)\cdot\bt^j\|_{L^\infty(0,\tfinal;L^2(\Sigma))}  
+ \|\partial_t p_{P,h}\|_{L^2(0,\tfinal;L^2(\Omega_P))}
+ \|p_{P,h}\|_{L^\infty(0,\tfinal;H^1(\Omega_P))} \nonumber \\
  & \qquad
+ \|\varphi_h\|_{L^\infty(0,\tfinal;L^2(\Omega_P))}
+ \frac{1}{\sqrt{\lambda}}\|\alpha \, \partial_t p_{P,h} - \partial_t\varphi_h\|_{L^2(0,\tfinal;L^2(\Omega_P))}
\nonumber \\
& \le C \left( \|\gg\|_{L^2(0,\tfinal;\bL^2(\Omega))} + \|\gg\cdot\nn\|_{L^2(0,\tfinal;L^2(\Sigma))}
+ \|\ff\|_{H^1(0,\tfinal;\bL^2(\Omega_P))} + \|p_{P,0}\|_{H^1(\Omega_P)}  \right),
\label{eq:stability-dt}
\end{align}
\end{subequations}
with a constant $C$ independent of $\lambda_{\max}$ and $h$.
\end{theorem}

\begin{proof}
To establish existence, we employ \cite[Th. 2.3.1]{brenan95}, 
which asserts that \eqref{eq:semi-discrete} has a solution if the
matrix pencil $s\cN + \cM$ is nonsingular for some $s \ne 0$. The
solvability of the associated initial value problem requires initial
data that is consistent with the DAE system. To deal with this issue,
we first consider a related DAE system by including new variables
$\theta_h^j \in \bW_h\cdot\bt^j$, $j = 1,\ldots d-1$ and equations
\begin{equation}\label{eq:theta}
  \langle \theta_h^j,\bw_h\cdot\bt^j \rangle_\Sigma =
  \langle \partial_t \bd_h \cdot \bt^j,\bw_h\cdot\bt^j \rangle_\Sigma \quad \forall \bw_h \in \bW_h, \quad j = 1,\ldots d-1,
\end{equation}
and replacing $\partial_t \bd_h \cdot \bt^j$ with $\theta_h^j$ in
\eqref{method-1} and \eqref{method-3}. Let 
$\tilde\uu_h:=[ \bu_h \ p_{F,h} \ \bd_h \ p_{P,h} \ \varphi_h \ \theta_h^1 \cdots \, \theta_h^{d-1}]^{\tt t}$ and denote
the extended DAE system by
\begin{equation}\label{DAE-ext}
  \partial_t \tilde\cN \tilde\uu_h(t) + \tilde\cM\tilde\uu_h(t) = \tilde\cF(t).
\end{equation}
Clearly any solution of \eqref{DAE-ext} also solves \eqref{eq:semi-discrete}.
We will apply \cite[Th.~2.3.1]{brenan95} to \eqref{DAE-ext}.
We will show that the matrix $\tilde\cN + \tilde\cM$ is nonsingular by proving that the system
$(\tilde\cN + \tilde\cM)\tilde\uu_h = \cero$ has only the zero solution. By eliminating
$\theta_h^j$, this system results in $(\cN + \cM)\uu_h = \cero$. Using that
\begin{equation*}
\cN + \cM = \left[\begin{array}{cc:ccc}
\cA_2^F & (\cB_1^F)' & (\cB_3^\Sigma)' &  (\cB_2^\Sigma)' & \cero \\[1ex]
-\cB_1^F & \cero & \cero & \cero &\cero \\
\hdashline
\cB_3^\Sigma & \cero& \cA_1^P + \cA_2^\Sigma & (\cB_4^\Sigma)' & (\cB_1^P)' \vphantom{\int^{X^X}}\\[1ex]
-\cB_2^\Sigma & \cero & -\cB_4^\Sigma  &\cA_3^P + \cA_4^P & -(\cB_2^P)' \\[1ex]
\cero & \cero & -\cB_1^P & -\cB_2^P & \cA_5^P 
  \end{array}\right],
\end{equation*}
the equation $\uu_h^{\tt t}(\cN + \cM)\uu_h = 0$ gives
\begin{align*}
  &  2\mu_f\!\int_{\Omega_F}\! \!\beps(\bu_h):\beps(\bu_h)
  + \sum_{j=1}^{d-1}\langle  \frac{\gamma\mu_f}{\sqrt{\kappa}}
  (\bu_h - \bd_h) \cdot\bt^j\,,(\bu_h - \bd_h) \cdot\bt^j\rangle_\Sigma
  + 2\mu_s\int_{\Omega_P} \beps(\bd_h):\beps(\bd_h) \\
  & \qquad
  + C_0 \int_{\Omega_P} p_{P,h}^2 + \int_{\Omega_P} \frac{\kappa}{\mu_f}\nabla p_{P,h} \cdot\nabla p_{P,h}
  + \frac{1}{\lambda} \int_{\Omega_P}(\alpha \, p_{P,h} - \varphi_h)^2 = 0,
\end{align*}
which implies that $\bu_h = \cero$, $\bd_h = \cero$, $p_{P,h} = 0$, and $\varphi_h = 0$.
Equations \eqref{eq:theta} imply that $\theta_h^j = 0$. The inf-sup
condition \eqref{inf-sup-stokes} with $q_{F,h} = p_{F,h}$ and 
\eqref{method-1} give $p_{F,h} = 0$.
Therefore \eqref{DAE-ext} has a solution. 

We proceed with the construction of the initial data. We first note
that there exists a solution to \eqref{DAE-ext} satisfying $p_{P,h}(0)
= p_{P,h,0}$, since this initial condition is associated with the
differential equation \eqref{method-4}. We need to find initial values
for the rest of the variables that are consistent with the DAE system.
Let us set $\theta_h^j(0) = \theta^j_{h,0} = 0$ and consider the Stokes problem
\eqref{method-1}--\eqref{method-2} for $\bu_{h,0}$ and $p_{F,h,0}$
with data $p_{P,h,0}$ and $\theta^j_{h,0}$, which is now decoupled and
well-posed from the Stokes finite element theory. Finally, let
$\bd_{h,0}$ and $\varphi_{h,0}$ solve the problem coupling
\eqref{method-3} and \eqref{method-5} with data $p_{P,h,0}$,
$\theta^j_{h,0}$, and $\bu_{h,0}$. The well-posedness of this problem
follows from the theory of saddle-point problems \cite{boffi13}, due to the
inf-sup condition \eqref{inf-sup-elast}, see also \cite{oyarzua16}.
We further note that taking $t \to 0$ in
\eqref{eq:theta} implies that $\partial_t \bd_h(0)\cdot\bt^j|_{\Sigma} =
\theta^j_{h,0} = 0$. Now, taking $t \to 0$ in
\eqref{method-1}--\eqref{method-2}, \eqref{method-3}, and
\eqref{method-5} and using the above construction of the initial data,
we conclude that $\bu_h(0) = \bu_{h,0}$, $p_{F,h}(0) = p_{F,h,0}$,
$\bd_h(0) = \bd_{h,0}$, and $\varphi_h(0) = \varphi_{h,0}$.

We proceed with the stability bound \eqref{eq:stability}.
Differentiating \eqref{method-5} in time and
taking $(\bv_h, q_{F,h}, \linebreak \bw_h,  q_{P,h},   \psi_h) 
= (\bu_h, p_{F,h}, \partial_t \bd_h, p_{P,h}, \varphi_h)$ in \eqref{eq:method} gives
\begin{align}
  & 2\mu_f\!\int_{\Omega_F}\! \!\beps(\bu_h):\beps(\bu_h)
  + \sum_{j=1}^{d-1} \langle  \frac{\gamma\mu_f}{\sqrt{\kappa}}
  (\bu_h - \partial_t\bd_h) \cdot\bt^j\,,(\bu_h - \partial_t\bd_h) \cdot\bt^j\rangle_\Sigma
  + \frac12\frac{\partial}{\partial t} 2\mu_s\int_{\Omega_P} \beps(\bd_h):\beps(\bd_h) 
  \nonumber \\
  & \qquad
  + \frac12\frac{\partial}{\partial t} C_0 \int_{\Omega_P} p_{P,h}^2
  + \int_{\Omega_P} \frac{\kappa}{\mu_f}\nabla p_{P,h} \cdot\nabla p_{P,h}
  + \frac12\frac{\partial}{\partial t}
  \frac{1}{\lambda} \int_{\Omega_P}(\alpha \, p_{P,h} - \varphi_h)^2 \nonumber \\
  & \qquad\quad = \rho_f\int_{\Omega_F} \gg\cdot\bu_h
  + \rho_s\int_{\Omega_P} \ff\cdot\partial_t\bd_h
  + \rho_f \int_{\Omega_P} \frac{\kappa}{\mu_f}\gg\cdot\nabla p_{P,h}
  - \rho_f\langle \frac{\kappa}{\mu_f}\gg\cdot\nn,p_{P,h}\rangle_\Sigma.
  \label{eq:stab-1}
\end{align}
Integrating from $0$ to $t \in (0,\tfinal]$, we obtain
\begin{align}
  & 
  2\mu_f \int_0^t \|\beps(\bu_h)\|_{\bL^2(\Omega_F)}^2
  + \sum_{j=1}^{d-1} \gamma\mu_f \int_0^t\|\kappa^{-1/4}(\bu_h - \partial_t\bd_h)\cdot\bt^j\|_{\bL^2(\Sigma)}^2
  + \mu_s \|\beps(\bd_h)(t)\|_{\bL^2(\Omega_P)}^2 \nonumber \\
  & \quad\qquad\qquad
  + \frac{C_0}{2}\|p_{P,h}(t)\|_{L^2(\Omega_P)}^2
  + \frac{1}{\mu_f} \int_0^t \|\kappa^{1/2}\nabla p_{P,h}\|_{L^2(\Omega_P)}^2
  + \frac{1}{2\lambda}\|(\alpha \, p_{P,h} - \varphi_h)(t)\|^2_{L^2(\Omega_P)} \nonumber \\
  & \qquad\quad = \mu_s \|\beps(\bd_h)(0)\|_{\bL^2(\Omega_P)}^2
  + \frac{C_0}{2}\|p_{P,h}(0)\|_{L^2(\Omega_P)}^2
  + \frac{1}{2\lambda}\|(\alpha \, p_{P,h} - \varphi_h)(0)\|^2_{L^2(\Omega_P)} \nonumber \\
& \qquad\quad\qquad
  + \rho_f \int_0^t \int_{\Omega_F} \gg\cdot\bu_h
  - \rho_s \int_0^t \int_{\Omega_P} \partial_t\ff\cdot\bd_h
  + \rho_s \int_{\Omega_P} \ff(t)\cdot\bd_h(t)
  - \rho_s \int_{\Omega_P} \ff(0)\cdot\bd_h(0)
  \nonumber \\
& \qquad\quad\qquad
  + \rho_f \int_0^t \int_{\Omega_P} \frac{\kappa}{\mu_f}\gg\cdot\nabla p_{P,h}
  - \rho_f \int_0^t \langle \frac{\kappa}{\mu_f}\gg\cdot\nn,p_{P,h}\rangle_\Sigma, \label{eq:stab-2}
\end{align}
where we have integrated by parts in time the second term on the right-hand side in
\eqref{eq:stab-1}.  Then, on the left-hand side we use Korn's inequality
\eqref{Korn}, the Poincar\'e inequality \eqref{Poincare}, and the
permeability bound \eqref{perm}, whereas on   \cred{the right-hand} side we use
the Cauchy-Schwarz inequality, the trace inequality \eqref{trace}, and
Young's inequality, obtaining
\begin{align}
&  
  2\mu_f(c_K^F)^2 \int_0^t \|\bu_h\|_{\bH^1(\Omega_F)}^2
  + \sum_{j=1}^{d-1}\frac{\gamma\mu_f}{\sqrt\kappa_2} \int_0^t\|(\bu_h - \partial_t\bd_h)\cdot\bt^j\|_{\bL^2(\Sigma)}^2
  + \mu_s (c_K^P)^2 \|\bd_h(t)\|_{\bH^1(\Omega_P)}^2 \nonumber \\
& \qquad\qquad  
  + \frac{C_0}{2}\|p_{P,h}(t)\|_{L^2(\Omega_P)}^2
  + \frac{\kappa_1 c_P^2}{\mu_f} \int_0^t \|p_{P,h}\|_{H^1(\Omega_P)}^2
  + \frac{1}{2\lambda}\|(\alpha \, p_{P,h} - \varphi_h)(t)\|^2_{L^2(\Omega_P)}
  \nonumber \\
  & \qquad \le \mu_s \|\beps(\bd_h)(0)\|_{\bL^2(\Omega_P)}^2
  + \frac{C_0}{2}\|p_{P,h}(0)\|_{L^2(\Omega_P)}^2
  + \frac{1}{2\lambda}\|(\alpha \, p_{P,h} - \varphi_h)(0)\|^2_{L^2(\Omega_P)}
  \nonumber \\
& \qquad\quad
  + \frac{\epsilon}{2} \left(\rho_f \int_0^t \|\bu_h\|_{\bL^2(\Omega_F)}^2
  + \rho_s \|\bd_h(t)\|_{\bL^2(\Omega_P)}^2
  + \frac{\rho_f \kappa_2}{\mu_f}((C_\Sigma^P)^2+1)\int_0^t\|p_{P,h}\|^2_{H^1(\Omega_P)}  \right)
  \nonumber \\
 & \qquad\quad
  + \frac{1}{2\epsilon}\left( \rho_f \int_0^t \|\gg\|_{\bL^2(\Omega_F)}^2
  + \rho_s \|\ff(t)\|_{\bL^2(\Omega_P)}^2
  + \frac{\rho_f \kappa_2}{\mu_f}\int_0^t(\|\gg\|_{\bL^2(\Omega_P)}^2
  + \|\gg\cdot\nn\|_{L^2(\Sigma)}^2)  \right) \nonumber \\
  & \qquad\quad
  + \frac{\rho_s}{2} \int_0^t \|\bd_h\|_{\bL^2(\Omega_P)}^2
  + \frac{\rho_s}{2} \int_0^t \|\partial_t \ff\| _{\bL^2(\Omega_P)}^2
  + \frac{\rho_s}{2} \|\bd_h(0)\|_{\bL^2(\Omega_P)}^2
  + \frac{\rho_s}{2} \|\ff(0)\|_{\bL^2(\Omega_P)}^2. \label{eq:stab-3}
\end{align}
Taking $\epsilon$ sufficiently small and employing Gronwall's inequality for the term
$\displaystyle \frac{\rho_s}{2} \int_0^t \|\bd_h\|_{\bL^2(\Omega_P)}^2$, we obtain
\begin{align}
 &\quad  \int_0^t \|\bu_h\|_{\bH^1(\Omega_F)}^2
  + \sum_{j=1}^{d-1}\int_0^t\|(\bu_h - \partial_t\bd_h)\cdot\bt^j\|_{\bL^2(\Sigma)}^2
  + \|\bd_h(t)\|_{\bH^1(\Omega_P)}^2 \nonumber \\
& \qquad\qquad\qquad
  + \|p_{P,h}(t)\|_{L^2(\Omega_P)}^2
  + \int_0^t \|p_{P,h}\|_{H^1(\Omega_P)}^2
  + \frac{1}{\lambda}\|(\alpha \, p_{P,h} - \varphi_h)(t)\|^2_{L^2(\Omega_P)}
 \label{stab-bound} \\
  &   \le C\biggl(\int_0^t (\|\gg\|_{\bL^2(\Omega)}^2  +  \|\gg\cdot\nn\|_{L^2(\Sigma)}^2)
  + \|\ff(t)\|_{\bL^2(\Omega_P)}^2 \nonumber\\
  & \qquad \qquad + \|\ff(0)\|_{\bL^2(\Omega_P)}^2 
  + \int_0^t \|\partial_t \ff\| _{\bL^2(\Omega_P)}^2
  + \|p_{P,0}\|_{H^1(\Omega_P)}^2\biggr),
   \nonumber
\end{align}
with a constant $C$ independent of $\lambda_{\max}$. In the above inequality
we have bounded the initial data terms by
$C\|p_{P,0}\|_{H^1(\Omega_P)}^2$. This bound follows from the classical
stability bound for the Stokes problem
\eqref{method-1}--\eqref{method-2}, which allows to obtain
$\|\bu_{h,0}\|_{\bH^1(\Omega_F)} \le C\|p_{P,h,0}\|_{H^1(\Omega_P)}$; 
a stability bound for the saddle-point problem \eqref{method-3},
\eqref{method-5} to obtain 
\[\|\bd_{h,0}\|_{\bH^1(\Omega_P)} +
\|\varphi_{h,0}\|_{L^2(\Omega_p)} \le C
(\|\bu_{h,0}\|_{\bH^1(\Omega_F)} + \|p_{P,h,0}\|_{H^1(\Omega_P)}),\]
(cf. \cite{oyarzua16}), 
and the $H^1$-stability of the $L^2$-projection
$\|p_{P,h,0}\|_{H^1(\Omega_P)} \le C \|p_{P,0}\|_{H^1(\Omega_P)}$ (see, e.g., \cite{bramble02}).

Next we proceed with bounding $p_{F,h}$ and $\varphi_h$.
The inf-sup condition \eqref{inf-sup-stokes} together with \eqref{method-1} give
\begin{align*}
\beta_F\|p_{F,h}\|_{L^2(\Omega_F)} & \le \sup_{\cero \ne \bv_h \in \bV_h^\Sigma}
\frac{b_1^F(\bv_h,p_{F,h})}{\|\bv_h\|_{\bH^1(\Omega_F)}} \\
& =  \sup_{\cero \ne \bv_h \in \bV_h}\frac{- a_2^F(\bu_h,\bv_h)
- b^\Sigma_2(\bv_h,p_{P,h}) - b_3^\Sigma(\bv_h,\partial_t\bd_h) + F^F(\bv_h)}
     {\|\bv_h\|_{\bH^1(\Omega_F)}} \\
     & \le 2 \mu_f \|\bu_h\|_{\bH^1(\Omega_F)}
     + \sum_{j=1}^{d-1}\frac{\gamma\mu_f C_\Sigma^F}{\sqrt\kappa_1}\! \|(\bu_h - \partial_t\bd_h)\cdot\bt^j\|_{\bL^2(\Sigma)} \\
    & \qquad  + C_\Sigma^F C_\Sigma^P \|p_{P,h}\|_{H^1(\Omega_P)}\! + \rho_f \|\gg\|_{\bL^2(\Omega_F)},
\end{align*}
implying
\begin{equation}\label{pF-bound}
  \int_0^t \|p_{F,h}\|_{L^2(\Omega_F)}^2 \le C
  \int_0^t \Big( \|\bu_h\|_{\bH^1(\Omega_F)}^2
  + \sum_{j=1}^{d-1}\|(\bu_h - \partial_t\bd_h)\cdot\bt^j\|_{\bL^2(\Sigma)}^2
  + \|p_{P,h}\|_{H^1(\Omega_P)}^2 + \|\gg\|_{\bL^2(\Omega_F)}^2 \Big).
\end{equation}
Finally, using the inf-sup condition \eqref{inf-sup-elast} and \eqref{method-3}, we obtain
\begin{align*}
\beta_P\|\varphi_h\|_{L^2(\Omega_P)} & \le \sup_{\cero \ne \bw_h \in \bW_h}
\frac{b_1^P(\bw_h,\varphi_h)}{\|\bw_h\|_{\bH^1(\Omega_P)}} \\
& = \sup_{\cero \ne \bw_h \in \bW_h}\!\!\frac{
-b_3^\Sigma(\bu_h,\bw_h)  - b_4^\Sigma(\bw_h,p_{P,h})
- a_1^P(\bd_h,\bw_h) - a_2^\Sigma(\partial_t\bd_h,\bw_h) + F^P(\bw_h)}{\|\bw_h\|_{\bH^1(\Omega_P)}} \\
& \le (C_\Sigma^P)^2 \|p_{P,h}\|_{H^1(\Omega_P)}
+ \sum_{j=1}^{d-1}\frac{\gamma\mu_f C_\Sigma^P}{\sqrt\kappa_1} \|(\bu_h - \partial_t\bd_h)\cdot\bt^j\|_{\bL^2(\Sigma)} \\
& \qquad + 2 \mu_s \|\bd_h\|_{\bH^1(\Omega_P)} 
+ \rho_s \|\ff\|_{\bL^2(\Omega_P)},
\end{align*}
yielding
\begin{equation}\label{varphi-bound}
  \int_0^t \|\varphi_h\|_{L^2(\Omega_P)}^2
  \le C \int_0^t \Big( 
  \|p_{P,h}\|_{H^1(\Omega_P)}^2
  + \sum_{j=1}^{d-1}\|(\bu_h - \partial_t\bd_h)\cdot\bt^j\|_{\bL^2(\Sigma)}^2
  + \|\bd_h\|_{\bH^1(\Omega_P)}^2 + \|\ff\|_{\bL^2(\Omega_P)}^2 \Big).
\end{equation}
Combining \eqref{stab-bound}--\eqref{varphi-bound} and using Gronwall's inequality for
the third term on the right-hand side in \eqref{varphi-bound}, we obtain \eqref{eq:stability}.

The above argument implies that the solution of
\eqref{eq:semi-discrete} under the initial conditions $p_{P,h}(0) =
p_{P,h,0}$, $\bd_h(0) = \bd_{h,0}$, and $\varphi_h(0) = \varphi_{h,0}$
is unique. In particular, taking $p_{P,h,0} = 0$, $\bd_{h,0} = \cero$,
$\varphi_{h,0} = 0$, $\gg = \cero$, and $\ff = \cero$, \eqref{eq:stab-3}
implies that \eqref{stab-bound} holds with right-hand side zero. Together with
\eqref{pF-bound} and \eqref{varphi-bound}, this gives that all components of the
solution are zero, therefore the solution is unique.

We next prove the higher regularity stability bound \eqref{eq:stability-dt}. To that end,
we differentiate in time \eqref{method-1}, \eqref{method-3}, and \eqref{method-5} and
take $(\bv_h, q_{F,h}, \bw_h, q_{P,h}, \psi_h)
= (\bu_h, \partial_t p_{F,h}, \partial_t \bd_h, \partial_t p_{P,h}, \partial_t \varphi_h)$ in
\eqref{eq:method}, obtaining 
\begin{align}
  & \frac12\frac{\partial}{\partial t} 2\mu_f\!\int_{\Omega_F}\!\! \!\beps(\bu_h):\beps(\bu_h)
  + \sum_{j=1}^{d-1}\frac12\frac{\partial}{\partial t}\langle  \frac{\gamma\mu_f}{\sqrt{\kappa}}
  (\bu_h - \partial_t\bd_h)\! \cdot\bt^j,(\bu_h - \partial_t\bd_h)\! \cdot\bt^j\rangle_\Sigma
  \nonumber \\
  & \quad
  + 2\mu_s\int_{\Omega_P}\!\! \beps(\partial_t\bd_h):\beps(\partial_t\bd_h) 
  + C_0 \int_{\Omega_P} (\partial_t p_{P,h})^2
  + \frac12\frac{\partial}{\partial t}
  \int_{\Omega_P} \frac{\kappa}{\mu_f}\nabla p_{P,h} \cdot\nabla p_{P,h}
  \nonumber \\
  & \quad
  + \frac{1}{\lambda} \int_{\Omega_P}(\alpha \,  \partial_tp_{P,h} - \partial_t\varphi_h)^2
  = \rho_s\int_{\Omega_P} \partial_t\ff\cdot\partial_t\bd_h
  + \rho_f \int_{\Omega_P} \frac{\kappa}{\mu_f}\gg\cdot\partial_t \nabla p_{P,h}
  - \rho_f\langle \frac{\kappa}{\mu_f}\gg\cdot\nn,\partial_t p_{P,h}\rangle_\Sigma,
  \label{eq:stab-dt-1}
\end{align}
where we have used that $\partial_t \gg = \cero$. Integration from $0$ to $t \in (0,\tfinal]$ gives
\begin{align*}
  & 
  \mu_f \|\beps(\bu_h)(t)\|_{\bL^2(\Omega_F)}^2
  + \sum_{j=1}^{d-1}\frac12 \gamma\mu_f \|\kappa^{-1/4}(\bu_h - \partial_t\bd_h)\cdot\bt^j(t)\|_{\bL^2(\Sigma)}^2
  + \mu_s \int_0^t \|\beps(\partial_t\bd_h)\|_{\bL^2(\Omega_P)}^2 \nonumber \\
  & \qquad
  + C_0 \int_0^t \|\partial_t p_{P,h}(t)\|_{L^2(\Omega_P)}^2
  + \frac{1}{2\mu_f} \|\kappa^{1/2}\nabla p_{P,h}(t)\|_{L^2(\Omega_P)}^2
  + \frac{1}{\lambda}
  \int_0^t \|\alpha \, \partial_t p_{P,h} - \partial_t\varphi_h\|^2_{L^2(\Omega_P)} \nonumber \\
  &  = \mu_f \|\beps(\bu_h)(0)\|_{\bL^2(\Omega_F)}^2
  + \sum_{j=1}^{d-1}\frac12 \gamma\mu_f \|\kappa^{-1/4}(\bu_h - \partial_t\bd_h)\cdot\bt^j(0)\|_{\bL^2(\Sigma)}^2
  + \frac{1}{2\mu_f} \|\kappa^{1/2}\nabla p_{P,h}(0)\|_{L^2(\Omega_P)}^2 \nonumber \\
  & \ 
  + \rho_s \int_0^t \int_{\Omega_P} \partial_t\ff\cdot\partial_t\bd_h
  + \rho_f \int_{\Omega_P} \frac{\kappa}{\mu_f}\gg\cdot(\nabla p_{P,h}(t) - \nabla p_{P,h}(0)) 
  - \rho_f \langle \frac{\kappa}{\mu_f}\gg\cdot\nn,p_{P,h}(t) - p_{P,h}(0)\rangle_\Sigma,
\end{align*}
where we have integrated by parts the last two terms in \eqref{eq:stab-dt-1} and
used that $\partial_t \gg = \cero$. Next, on the left-hand side we use Korn's inequality
\eqref{Korn}, the Poincar\'e inequality \eqref{Poincare}, and the
permeability bound \eqref{perm}, while on the right-hand side we invoke Cauchy-Schwarz inequality, the trace inequality \eqref{trace}, and
Young's inequality, yielding 
\begin{align*}
  & 
  \mu_f (c_K^F)^2 \|\bu_h(t)\|_{\bH^1(\Omega_F)}^2
  + \sum_{j=1}^{d-1}\frac12 \frac{\gamma\mu_f}{\sqrt\kappa_2} \|(\bu_h - \partial_t\bd_h)\cdot\bt^j(t)\|_{\bL^2(\Sigma)}^2
  + \mu_s (c_K^P)^2  \int_0^t \|\partial_t \bd_h\|_{\bH^1(\Omega_P)}^2 \\
  & \qquad\quad\qquad  
  + C_0 \int_0^t \|\partial_t p_{P,h}(t)\|_{L^2(\Omega_P)}^2
  + \frac{\kappa_1 c_P^2}{2\mu_f} \|p_{P,h}(t)\|_{H^1(\Omega_P)}^2
  + \frac{1}{\lambda}
  \int_0^t \|\alpha \, \partial_t p_{P,h} - \partial_t\varphi_h\|^2_{L^2(\Omega_P)}
  \\
  & \qquad\quad \le \mu_f \|\beps(\bu_h)(0)\|_{\bL^2(\Omega_F)}^2
  + \sum_{j=1}^{d-1}\frac12 \frac{\gamma\mu_f}{\sqrt\kappa_1} \|(\bu_h - \partial_t\bd_h)\cdot\bt^j(0)\|_{\bL^2(\Sigma)}^2
  + \frac{\kappa_2}{2\mu_f} \|\nabla p_{P,h}(0)\|_{L^2(\Omega_P)}^2 \\
  & \qquad\quad\qquad
  + \frac{\epsilon}{2}\left(
  \rho_s\int_0^t \|\partial_t\bd_h\|_{\bL^2(\Omega_P)}^2
  + \frac{\rho_f \kappa_2}{\mu_f}((C_\Sigma^P)^2+1)
  (\|p_{P,h}(0)\|^2_{H^1(\Omega_P)} + \|p_{P,h}(t)\|^2_{H^1(\Omega_P)}) \right) \\
  & \qquad\quad\qquad
  + \frac{1}{2\epsilon} \left( \rho_s \int_0^t \|\partial_t\ff\|_{\bL^2(\Omega_P)}^2
  + \frac{\rho_f \kappa_2}{\mu_f}(\|\gg\|_{\bL^2(\Omega_P)}^2
  + \|\gg\cdot\nn\|_{L^2(\Sigma)}^2) \right).
\end{align*}
In addition, bounding the initial data terms as in \eqref{stab-bound} and taking $\epsilon$
sufficiently small, we can assert that 
\begin{align}
  & \|\bu_h(t)\|_{\bH^1(\Omega_F)}^2
  + \sum_{j=1}^{d-1}\|(\bu_h - \partial_t\bd_h)\cdot\bt^j(t)\|_{\bL^2(\Sigma)}^2
  + \int_0^t \|\partial_t \bd_h\|_{\bH^1(\Omega_P)}^2 \nonumber \\
& \qquad\qquad\quad
  + \int_0^t \|\partial_t p_{P,h}(t)\|_{L^2(\Omega_P)}^2
  + \|p_{P,h}(t)\|_{H^1(\Omega_P)}^2
  + \frac{1}{\lambda}
  \int_0^t \|\alpha \, \partial_t p_{P,h} - \partial_t\varphi_h\|^2_{L^2(\Omega_P)}
  \nonumber \\
  & \qquad\quad \le C \left(\int_0^t \|\partial_t\ff\|_{\bL^2(\Omega_P)}^2
  + \|\gg\|_{\bL^2(\Omega_P)}^2 + \|\gg\cdot\nn\|_{L^2(\Sigma)}^2 + \|p_{P,0}\|_{H^1(\Omega_P)}^2 \right),
  \label{stab-bound-dt}
\end{align}
with a constant $C$ independent of $\lambda_{\max}$. Next, using the inf-sup conditions
\eqref{inf-sup-stokes} and \eqref{inf-sup-elast}, and proceeding similarly to the derivations of \eqref{pF-bound} and
\eqref{varphi-bound}, we obtain
\begin{equation*}
  \|p_{F,h}(t)\|_{L^2(\Omega_F)}^2\! \le C \!
  \Big( \|\bu_h(t)\|_{\bH^1(\Omega_F)}^2
  + \sum_{j=1}^{d-1}\|(\bu_h \!-\! \partial_t\bd_h)\cdot\bt^j(t)\|_{\bL^2(\Sigma)}^2
  + \|p_{P,h}(t)\|_{H^1(\Omega_P)}^2 + \|\gg\|_{\bL^2(\Omega_F)}^2 \Big),
\end{equation*}
and
\begin{equation}\label{varphi-bound-dt}
  \|\varphi_h(t)\|_{L^2(\Omega_P)}^2
  \le \! C \!\Big( 
  \|p_{P,h}(t)\|_{H^1(\Omega_P)}^2
  + \sum_{j=1}^{d-1}\|(\bu_h - \partial_t\bd_h)\cdot\bt^j(t)\|_{\bL^2(\Sigma)}^2\!
  + \! \|\bd_h(t)\|_{\bH^1(\Omega_P)}^2
  + \|\ff(t)\|_{\bL^2(\Omega_P)}^2 \Big).
\end{equation}
Finally, combining \eqref{stab-bound-dt}--\eqref{varphi-bound-dt} and employing \eqref{eq:stability}
for the control of $\|\bd_h(t)\|_{\bH^1(\Omega_P)}$, we obtain the second bound \eqref{eq:stability-dt}.
\end{proof}

\begin{remark}
  We emphasise that, even though initial data was constructed for all variables, the initial
  value problem for \eqref{eq:semi-discrete} involves initial conditions only for
  $p_{P,h}$, $\bd_h$, and $\varphi_h$.
\end{remark}

\subsection{Existence, uniqueness, and stability of the weak solution}
\begin{theorem}\label{thm:well-posed-weak}
For each $\ff \in H^1(0,\tfinal;\bL^2(\Omega_P))$ and $p_{P,0} \in H^1_\star(\Omega_p)$,
  there exist initial data $\bu_{0} \in \bH^1_\star(\Omega_F)$, $p_{F,0} \in L^2(\Omega_F)$,
  $\bd_{0} \in \bH^1_\star(\Omega_P)$, and $\varphi_{0} \in L^2(\Omega_P)$ such that the
  weak formulation \eqref{eq:weak} with $a_1^F = 0$, $c^F = 0$, and $\tilde \alpha = 1$,
  complemented with the initial
  conditions $p_{P}(0) = p_{P,0}$, $\bd(0) = \bd_{0}$, and $\varphi(0) = \varphi_{0}$,
  has a unique solution satisfying
\begin{subequations}
\begin{align}
& \|\bu\|_{L^2(0,\tfinal;\bH^1(\Omega_F))} + \|p_{F}\|_{L^2(0,\tfinal;L^2(\Omega_F))}
  + \|\bd\|_{L^\infty(0,\tfinal;\bH^1(\Omega_P))}
  \nonumber \\
& \qquad
  + \sum_{j=1}^{d-1}\|(\bu - \partial_t\bd)\cdot\bt^j\|_{L^2(0,\tfinal;L^2(\Sigma))}
+ \|p_{P}\|_{L^\infty(0,\tfinal;L^2(\Omega_P))}
+ \|p_{P}\|_{L^2(0,\tfinal;H^1(\Omega_P))} \nonumber \\
& \qquad
+ \|\varphi\|_{L^2(0,\tfinal;L^2(\Omega_P))} 
+ \frac{1}{\sqrt{\lambda}}\|\alpha \, p_{P} - \varphi\|_{L^\infty(0,\tfinal;L^2(\Omega_P))}
\nonumber \\
& \le C \left( \|\gg\|_{L^2(0,\tfinal;\bL^2(\Omega))} + \|\gg\cdot\nn\|_{L^2(0,\tfinal;L^2(\Sigma))}
+ \|\ff\|_{H^1(0,\tfinal;\bL^2(\Omega_P))} + \|p_{P,0}\|_{H^1(\Omega_P)}  \right),
\label{eq:stability-weak}
\end{align}
and
\begin{align}
& \|\bu\|_{L^\infty(0,\tfinal;\bH^1(\Omega_F))} + \|p_{F}\|_{L^\infty(0,\tfinal;L^2(\Omega_F))}
+ \|\partial_t \bd\|_{L^2(0,\tfinal;\bH^1(\Omega_P))} \nonumber \\
& \qquad
  + \sum_{j=1}^{d-1}\|(\bu - \partial_t\bd)\cdot\bt^j\|_{L^\infty(0,\tfinal;L^2(\Sigma))}  
+ \|\partial_t p_{P}\|_{L^2(0,\tfinal;L^2(\Omega_P))}
+ \|p_{P}\|_{L^\infty(0,\tfinal;H^1(\Omega_P))} \nonumber \\
  & \qquad
+ \|\varphi\|_{L^\infty(0,\tfinal;L^2(\Omega_P))}
+ \frac{1}{\sqrt{\lambda}}\|\alpha \, \partial_t p_{P} - \partial_t\varphi\|_{L^2(0,\tfinal;L^2(\Omega_P))}
\nonumber \\
& \le C \left( \|\gg\|_{L^2(0,\tfinal;\bL^2(\Omega))} + \|\gg\cdot\nn\|_{L^2(0,\tfinal;L^2(\Sigma))}
+ \|\ff\|_{H^1(0,\tfinal;\bL^2(\Omega_P))} + \|p_{P,0}\|_{H^1(\Omega_P)}  \right),
\label{eq:stability-dt-weak}
\end{align}
\end{subequations}
with a constant $C$ independent of $\lambda_{\max}$.
\end{theorem}  
\begin{proof}
From Theorem~\ref{thm:well-posed}, there exists an infinite sequence
$\{(\bu_h,p_{F,h},\bd_h,p_{P,h},\varphi_h)\}_{h>0}$ satisfying
\eqref{eq:semi-discrete} for each $h$ such that $\{\bu_h\}_{h>0}$ is
bounded in $L^2(0,\tfinal;\bH^1(\Omega_F))$, $\{p_{F,h}\}_{h>0}$ is
bounded in $L^2(0,\tfinal;\- L^2(\Omega_F))$, $\{\bd_h\}_{h>0}$ is
bounded in $H^1(0,\tfinal;\bH^1(\Omega_P))$, the sequence $\{p_{P,h}\}_{h>0}$ is
bounded in $H^1(0,\tfinal;L^2(\Omega_P))$, as well as in
$L^2(0,\tfinal;H^1(\Omega_P))$, and the sequence $\{\varphi_h\}_{h>0}$ is bounded
in $H^1(0,\tfinal;L^2(\Omega_P))$. Therefore there exist weakly
convergent subsequences, denoted in the same way, such that
\begin{gather*}
   \bu_h \rightharpoonup \bu \text{ in } L^2(0,\tfinal;\bH^1(\Omega_F)), \quad
  p_{F,h} \rightharpoonup p_F \text{ in } L^2(0,\tfinal;L^2(\Omega_F)), \\
   \bd_h \rightharpoonup \bd \text{ in } H^1(0,\tfinal;\bH^1(\Omega_P)), \\
    p_{P,h} \rightharpoonup p_P \text{ in }
  H^1(0,\tfinal;L^2(\Omega_P)) \cap L^2(0,\tfinal;H^1(\Omega_P)), \quad
  \varphi_h \rightharpoonup \varphi \text{ in } H^1(0,\tfinal;L^2(\Omega_P)).
\end{gather*}
Next, we fix a set of test functions $(\bv_h, q_{F,h}, \bw_h, q_{P,h}, \psi_h) \in
C^0(0,\tfinal;\bV_h\times \rQ^F_h\times\bW_h\times\rQ^P_h\times\rZ_h)$ in \eqref{eq:method}, integrate it in time from 0 to $\tfinal$, and take $h \to 0$.
Since all bilinear forms and linear functionals are continuous, cf. Lemma~\ref{lem:cont}, we conclude that
$(\bu,p_{F},\bd,p_{P},\varphi)$ satisfy the time-integrated version of \eqref{eq:weak} with this choice of test functions. Finally, since the space
$C^0(0,\tfinal;\bV_h\times \rQ^F_h\times\bW_h\times\rQ^P_h\times\rZ_h)$ is dense in
$L^2(0,\tfinal;\bH^1_{\star}(\Omega_F)\times L^2(\Omega_F)\times \bH^1_{\star}(\Omega_P) \times
H^1_{\star}(\Omega_P) \times L^2(\Omega_P))$, we conclude that \eqref{eq:weak} holds for a.e.
$t \in (0,\tfinal)$. 

It remains to handle the initial conditions. First, taking $h \to 0$
in $p_{P,h}(0) = p_{P,h,0}$ gives $p_P(0) = p_{P,0}$. We further note
that the control of the terms
$\|(\bu_h - \partial_t\bd_h)\cdot\bt^j\|_{L^\infty(0,\tfinal;L^2(\Sigma))}$
and $\|\bu_h\|_{L^\infty(0,\tfinal;\bH^1(\Omega_F))}$
in \eqref{eq:stability-dt} implies that for all $t \in [0,\tfinal]$,
$\partial_t \bd_h(t)\cdot\bt^j \rightharpoonup \partial_t \bd(t)\cdot\bt^j$ in $L^2(\Sigma)$.
Taking $t = 0$ and $h \to 0$ in \eqref{eq:theta} and using that $\partial_t \bd_h(0)\cdot\bt^j = 0$ on $\Sigma$,
we conclude that $\partial_t \bd(0)\cdot\bt^j = 0$ on $\Sigma$. Next, the stability of the Stokes and elasticity problems for the initial data in the proof of
Theorem~\ref{thm:well-posed} implies that there exist weakly convergent subsequences such that 
\begin{gather*}
   \bu_{h,0} \rightharpoonup \bu_0 \text{ in } \bH^1(\Omega_F), \ \
  p_{F,h,0} \rightharpoonup p_{F,0} \text{ in } L^2(\Omega_F), \ \
   \bd_{h,0} \rightharpoonup \bd_0 \text{ in } \bH^1(\Omega_P), \ \
  \varphi_{h,0} \rightharpoonup \varphi_0 \text{ in } L^2(\Omega_P).
\end{gather*}
Then, taking $t \to 0$ in \eqref{weak-1}--\eqref{weak-2},
\eqref{weak-3}, and \eqref{weak-5} and using that the initial data
satisfies the same equations, we conclude that $\bu(0) = \bu_{0}$,
$p_{F}(0) = p_{F,0}$, $\bd(0) = \bd_{0}$, and $\varphi(0) =
\varphi_{0}$.

Finally, the uniqueness of the solution under the initial conditions
$p_{P}(0) = p_{P,0}$, $\bd(0) = \bd_{0}$, $\varphi(0) = \varphi_{0}$, and
the stability bounds \eqref{eq:stability-weak} and \eqref{eq:stability-dt-weak},
follow in the same way as in the proof of Theorem~\ref{thm:well-posed}.
\end{proof}

\section{Fully discrete scheme
}\label{sec:error}
We apply a time discretisation to \eqref{eq:method}
using backward Euler's method with fixed time step $\Delta t = \tfinal/N$.
Let $t_n = n \Delta t$, $n = 0,\ldots,N$, be the discrete times.
Starting from the discrete initial data constructed in the proof of Theorem~\ref{thm:well-posed}, at each time iteration $n=1,\ldots,N$ we look for $(\bu^{n}_h,p^{n}_{F,h},\bd^{n}_h,p^{n}_{P,h},\varphi^{n}_h)\in \bV_h\times \rQ^F_h\times\bW_h\times\rQ^P_h\times\rZ_h=:\mathbf{H}_h$ such that 
\begin{subequations}\label{eq:scheme} 
\begin{align}
  a_2^F(\bu_h^n,\bv_h) + b^F_1(\bv_h,p_{F,h}^n) 
  + b^\Sigma_2(\bv_h,p_{P,h}^n) + b_3^\Sigma(\bv_h,\partial_t^n\bd_h)
 & = F^{F,n}(\bv_h) & \! \forall \bv_h \in \bV_h, \label{scheme-1}\\
 - b^F_1(\bu_h^n,q_{F,h}) &= 0 & \! \forall q_{F,h}\in \rQ_h^F, \label{scheme-2} \\
 b_3^\Sigma(\bu_h^n,\bw_h)  +  b_4^\Sigma(\bw_h,p_{P,h}^n) + a_1^P(\bd_h^n,\bw_h) &\nonumber \\
 + a_2^\Sigma(\partial_t^n\bd_h,\bw_h)
+ b_1^P(\bw_h,\varphi_h^n) & = F^P(\bw_h) & \! \forall \bw_h \in \bW_h, \label{scheme-3} \\
 - b_2^\Sigma(\bu_h^n,q_{P,h})  - b_4^\Sigma(\partial_t^n \bd_h,q_{P,h})
+ a_3^P(\partial_t^n p_{P,h},q_{P,h})& \nonumber \\
 + a_4^P(p_{P,h}^n,q_{P,h}) 
- b_2^P(\partial_t^n \varphi_h,q_{P,h}) &= G^n(q_{P,h}) & \! \forall q_{P,h} \in \rQ_h^P,
\label{scheme-4}\\
 - b_1^P(\bd_h^n,\psi_h) -  b_2^P(\psi_h,p_{P,h}^n) + a_5^P(\varphi_h^n,\psi_h)& = 0 &
\! \forall \psi_h \in \rZ_h, \label{scheme-5}
\end{align}\end{subequations}
where, for a generic scalar or vector field $f$, we set $f^n:= f(t_n)$ and $\partial_t^n f:= \frac{1}{\Delta t}(f^n - f^{n-1})$. 
For convenience we also define the global discrete time derivative $\tilde\partial_t f$ such that $(\tilde\partial_t f)^n := \partial_t^n f$ for
$n = 1,\ldots,N$. 
The method requires solving at each time step the algebraic system
\begin{equation}\label{eq:fully-discrete}
\left[\begin{array}{cc:ccc}
\cA_2^F & (\cB_1^F)' & \frac{1}{\Delta t}(\cB_3^\Sigma)' &  (\cB_2^\Sigma)' & \cero \\[1ex]
-\cB_1^F & \cero & \cero & \cero &\cero \\
\hdashline
\cB_3^\Sigma & \cero& \cA_1^P + \frac{1}{\Delta t}\cA_2^\Sigma & (\cB_4^\Sigma)' & (\cB_1^P)' \vphantom{\int^{X^X}}\\[1ex]
-\cB_2^\Sigma & \cero & -\frac{1}{\Delta t}\cB_4^\Sigma  &\frac{1}{\Delta t}\cA_3^P + \cA_4^P & -\frac{1}{\Delta t}(\cB_2^P)' \\[1ex]
\cero & \cero & -\cB_1^P & -\cB_2^P & \cA_5^P 
  \end{array}\right]
\left[\begin{array}{c}
\bu^{n}_h\\[1ex]
p^{n}_{F,h} \\[1ex]
\bd^{n}_h\\[1ex]
p^{n}_{P,h}\\[1ex]
\varphi^{n}_h
\end{array}\right]=
\left[\begin{array}{c}
\tilde\cF^{F,n} \\[1ex]  
\cero \\[1ex] 
\tilde\cF^{P,n} \\[1ex] 
\tilde\cG^{n} \\[1ex] 
\cero \end{array}
\right],
\end{equation}
where the tilde-notation on the right-hand side
vectors indicate that they also receive contributions from the backward
Euler time-discretisation.

\begin{theorem}\label{well-posed-discrete}
The fully discrete method \eqref{eq:fully-discrete} has a unique solution. 
  \end{theorem}

\begin{proof}
  Consider the matrix obtained from the matrix in \eqref{eq:fully-discrete} by scaling the third and fifth rows by $\frac{1}{\Delta t}$. It has the same structure as the matrix $\cN + \cM$ that appears in the proof of Theorem~\ref{thm:well-posed}, which is shown to be nonsingular. Therefore the scaled matrix is nonsingular, and so is the matrix in \eqref{eq:fully-discrete}.
  \end{proof}

\subsection{Error estimates for the fully discrete scheme}\label{sec:error-fully}
We proceed with the error analysis for the fully discrete scheme. We will make use of the discrete space-time norms for $f: t_n \to \rV$, $n = 1,\ldots,N$,
$$
\|f\|^2_{l^2(0,\tfinal;\rV)} := \sum_{n=1}^N \Delta t \|f^n\|_{\rV}^2, \quad
\|f\|_{l^\infty(0,\tfinal;\rV)} := \max\limits_{n=1,\ldots,N} \|f^n\|_{\rV},
$$

\cred{In addition, we consider finite-dimensional subspaces of continuous and piecewise polynomials of generic degrees 
\begin{equation}\label{eq:degrees}
k_{\bv},k_{\bw},k_{q_P}\geq 1,\quad k_{q_F},k_{z} \geq 0, \quad \text{
for the spaces} \quad \bV_h, \bW_h, \rQ_h^P, \rQ_h^F, \rZ_h,\end{equation} 
respectively.} 

Let $I^\bV:\bH^1_\star(\Omega_F) \to \bV_h$, $I^\bW:\bH^1_\star(\Omega_P) \to \bW_h$, and $I^{\rQ^P}: H^1_\star(\Omega_P) \to \rQ_h^P$ be the Scott-Zhang interpolants \cite{Scott-Zhang}. In addition, let $I^{\rQ^F}:L^2(\Omega_F) \to \rQ_h^F$ and $I^\rZ:L^2(\Omega_P) \to \rZ_h$ be the $L^2$-orthogonal projections. These operators, {alongside the polynomial degrees \eqref{eq:degrees}} have the approximation properties (see, e.g., \cite{Scott-Zhang,Ciarlet}) 
\cred{\begin{subequations}\begin{align}
  & \|\bv - I^\bV \bv\|_{\bH^s(\Omega_F)} \le C h^{r_{\bv} - s}\|\bv\|_{\bH^{r_{\bv}}(\Omega_F)}, 
  && 1 \le r_{\bv} \le k_{\bv}+1, \ 0 \le s \le 1, \label{IV-approx}\\
  & \|q_F - I^{\rQ^F} q_F\|_{L^2(\Omega_P)} \le C h^{r_{q_F}}\|q_F\|_{H^{r_{q_F}}(\Omega_F)}, 
  && 0 \le r_{q_F} \le k_{q_F}+1, \label{IQF-approx}\\
  & \|\bw - I^\bW \bw\|_{\bH^s(\Omega_P)} \le C h^{r_{\bw} - s}\|\bw\|_{\bH^{r_{\bw}}(\Omega_P)}, 
  && 1 \le r_{\bw} \le k_{\bw}+1, \ 0 \le s \le 1, \label{IW-approx}\\
  & \|q_P - I^{\rQ^P} q_P\|_{H^s(\Omega_P)} \le C h^{r_{q_P} - s}\|q_P\|_{H^{r_{q_P}}(\Omega_P)}, 
  && 1 \le r_{q_P} \le k_{q_P}+1, \ 0 \le s \le 1, \label{IQP-approx}\\
  & \|\psi - I^{\rZ} \psi\|_{L^2(\Omega_P)} \le C h^{r_z}\|\psi\|_{H^{r_z}(\Omega_P)}, 
  && 0 \le r_{z} \le k_{z}+1. \label{IZ-approx}
\end{align}\end{subequations}}

\begin{theorem}\label{thm:error-bound}
Assume that the weak solution of \eqref{eq:weak} is sufficiently smooth. Then, for the fully discrete solution \eqref{eq:scheme}, there exists a constant $C$ independent of $\lambda_{\max}$, $h$, and $\Delta t$, such \cred{that
\begin{align}
  & \|\bu - \bu_h\|_{l^2(0,\tfinal;\bH^1(\Omega_F))}
  + \|p_{F} - p_{F,h}\|_{l^2(0,\tfinal;L^2(\Omega_F))}
  + \|\bd - \bd_h\|_{l^\infty(0,\tfinal;\bH^1(\Omega_P))}
  \nonumber \\
& \qquad 
  + \sum_{j=1}^{d-1}\|(\bu - \tilde\partial_t\bd_h)\cdot\bt^j\|_{l^2(0,\tfinal;L^2(\Sigma))}
+ \|p_{P} - p_{P,h}\|_{l^\infty(0,\tfinal;L^2(\Omega_P))}
+ \|p_{P} - p_{P,h}\|_{l^2(0,\tfinal;H^1(\Omega_P))} \nonumber \\
& \qquad 
+ \|\varphi - \varphi_h\|_{l^2(0,\tfinal;L^2(\Omega_P))} 
+ \frac{1}{\sqrt{\lambda}}\|(\alpha \, p_{P} - \varphi) - (\alpha \, p_{P,h} - \varphi_h)\|_{l^\infty(0,\tfinal;L^2(\Omega_P))}
\nonumber \\
& \quad \le C \Big( h^{r_{\bv}}\|\bu\|_{H^1(0,\tfinal;\bH^{r_{\bv}+1}(\Omega_F))}
+ h^{r_{q_F}}\|p_F\|_{H^1(0,\tfinal;H^{r_{q_F}}(\Omega_F))}
+ h^{r_{\bw}}\|\bd\|_{W^{2,\infty}(0,\tfinal;\bH^{r_{\bw}+1}(\Omega_P))} \nonumber \\
& \qquad\qquad + h^{r_{q_P}}\|p_P\|_{H^1(0,\tfinal;H^{r_{q_P}+1}(\Omega_P))}
+ h^{r_z}\|\varphi\|_{H^1(0,\tfinal;H^{r_z}(\Omega_P))}
\nonumber \\
& \qquad\qquad + \Delta t \big( \|\bd\|_{H^3(0,T;\bL^2(\Sigma))} + \|p_P\|_{H^2(0,T;L^2(\Omega_P)}
+ \|\varphi\|_{H^2(0,T;L^2(\Omega_P)}\big)
\Big), 
\label{eq:error-bound}
\end{align}
with $1 \le r_{\bv} \le k_{\bv} $, $0 \le r_{q_F} \le k_{q_F}+1$, $1 \le r_{\bw} \le k_{\bw}$, $1 \le r_{q_P} \le k_{q_P}$, and $0 \le r_{z} \le k_{z} +1$.} 

\end{theorem}

\begin{proof}
We decompose the numerical errors into approximation and discretisation components:
\begin{align*}
& \bu - \bu_h = (\bu - I^\bV\bu) + (I^\bV\bu - \bu_h) =: e_{\bu,I} + e_{\bu,h},\\
& p_F - p_{F,h} = (p_F - I^{\rQ^F}p_F) + (I^{\rQ^F}p_F - p_{F,h}) =: e_{p_F,I} + e_{p_F,h},\\
& \bd - \bd_h = (\bd - I^\bW\bd) + (I^\bW\bd - \bd_h) =: e_{\bd,I} + e_{\bd,h},\\
& p_P - p_{P,h} = (p_P - I^{\rQ^P}p_P) + (I^{\rQ^P}p_P - p_{P,h}) =: e_{p_P,I} + e_{p_P,h},\\
& \varphi - \varphi_h = (\varphi - I^\rZ\varphi) + (I^\rZ\varphi - \varphi_h) =: e_{\varphi,I} + e_{\varphi,h}.
\end{align*}
Denote the time discretisation errors as $r_\phi^n := \phi(t_n) - \partial_t^n \phi$,
for $\phi \in \{\bd, p_P, \varphi \}$. 
Subtracting \eqref{eq:scheme} from \eqref{eq:weak}, we obtain the error system
\begin{subequations}\label{eq:error} 
\begin{align}
&  a_2^F(e_{\bu,h}^n,\bv_h) + b^F_1(\bv_h,e_{p_F,h}^n) 
  + b^\Sigma_2(\bv_h,e_{p_P,h}^n) + b_3^\Sigma(\bv_h,\partial_t^n e_{\bd,h}) \nonumber \\
  & \quad = -a_2^F(e_{\bu,I}^n,\bv_h) - b^F_1(\bv_h,e_{p_F,I}^n) 
  - b^\Sigma_2(\bv_h,e_{p_P,I}^n) - b_3^\Sigma(\bv_h,\partial_t^n e_{\bd,I})
  - b_3^\Sigma(\bv_h,r_{\bd}^n), \label{error-1}\\
& - b^F_1(e_{\bu,h}^n,q_{F,h}) = b^F_1(e_{\bu,I}^n,q_{F,h})  , \label{error-2} \\
  & b_3^\Sigma(e_{\bu,h}^n,\bw_h)  +  b_4^\Sigma(\bw_h,e_{p_P,h}^n) + a_1^P(e_{\bd,h}^n,\bw_h) + a_2^\Sigma(\partial_t^n e_{\bd,h},\bw_h) + b_1^P(\bw_h,e_{\varphi,h}^n) =
  - b_3^\Sigma(e_{\bu,I}^n,\bw_h)  \nonumber \\
  & \quad -  b_4^\Sigma(\bw_h,e_{p_P,I}^n)
  - a_1^P(e_{\bd,I}^n,\bw_h) - a_2^\Sigma(\partial_t^n e_{\bd,I},\bw_h)
  - b_1^P(\bw_h,e_{\varphi,I}^n) - a_2^\Sigma(r_{\bd}^n,\bw_h), \label{error-3} \\
& - b_2^\Sigma(e_{\bu,h}^n,q_{P,h})  - b_4^\Sigma(\partial_t^n e_{\bd,h},q_{P,h})
+ a_3^P(\partial_t^n e_{p_P,h},q_{P,h}) 
+ a_4^P(e_{p_P,h}^n,q_{P,h}) 
- b_2^P(\partial_t^n e_{\varphi,h},q_{P,h}) \nonumber \\
& \quad = b_2^\Sigma(e_{\bu,I}^n,q_{P,h})  + b_4^\Sigma(\partial_t^n e_{\bd,I},q_{P,h})
- a_3^P(\partial_t^n e_{p_P,I},q_{P,h}) 
- a_4^P(e_{p_P,I}^n,q_{P,h}) \nonumber \\
& \quad + b_2^P(\partial_t^n e_{\varphi,I},q_{P,h})
- a_3^P(r_{p_P}^n,q_{P,h}) - b_2^P(r_{\varphi}^n,q_{P,h}),
\label{error-4}\\
& - b_1^P(\partial_t^n e_{\bd,h},\psi_h) -  b_2^P(\psi_h,\partial_t^n e_{p_P,h})
+ a_5^P(\partial_t^n e_{\varphi,h},\psi_h) =
b_1^P(\partial_t^n e_{\bd,I},\psi_h) +  b_2^P(\psi_h,\partial_t^n e_{p_P,I}). \label{error-5}
\end{align}\end{subequations}
Equation \eqref{error-5} has been obtained by taking the divided difference of the error equation at $t_n$ and $t_{n-1}$ for $n = 1,\ldots,N$, using that it is satisfied by the initial data. We also used the orthogonality property of $I^\rZ$ to conclude that $a_5^P(\partial_t^n e_{\varphi,I},\psi_h) = 0$. Now, taking
$(\bv_h, q_{F,h}, \bw_h,  q_{P,h},   \psi_h)
= (e_{\bu,h}^n, e_{p_F,h}^n, \partial_t^n e_{\bd,h}, e_{p_P,h}^n, e_{\varphi,h}^n)$ in \eqref{eq:error}, summing the equations, and using the identity
\begin{equation*}
  (\partial_t^n \phi)  \phi^n = \frac12 \partial_t^n (\phi^2)
  + \frac12 \Delta t (\partial_t^n \phi)^2,
\end{equation*}
results in
\begin{equation}\begin{split}\label{discr-error-1}
  & 2\mu_f\!\int_{\Omega_F}\! \!\beps(e_{\bu,h}^n):\beps(e_{\bu,h}^n)
  + \sum_{j=1}^{d-1}\langle  \frac{\gamma\mu_f}{\sqrt{\kappa}}
  (e_{\bu,h}^n - \partial_t^n e_{\bd,h}) \cdot\bt^j\,,(e_{\bu,h}^n - \partial_t^n e_{\bd,h}) \cdot\bt^j\rangle_\Sigma \\
  & \qquad
  + \frac12\partial_t^n 2\mu_s\int_{\Omega_P} \beps(e_{\bd,h}):\beps(e_{\bd,h})
  + \frac12 \Delta t \, 2\mu_s\int_{\Omega_P} \partial_t^n\beps(e_{\bd,h}):\partial_t^n\beps(e_{\bd,h}) \\
  & \qquad  
  + \frac12 \partial_t^n C_0 \int_{\Omega_P} (e_{p_P,h})^2
  + \frac12 \Delta t C_0 \int_{\Omega_P} (\partial_t^n e_{p_P,h})^2  
  + \int_{\Omega_P} \frac{\kappa}{\mu_f}\nabla e_{p_P,h}^n \cdot\nabla e_{p_P,h}^n \\
  & \qquad    
  + \frac12\partial_t^n
  \frac{1}{\lambda} \int_{\Omega_P}(\alpha \, e_{p_P,h} - e_{\varphi,h})^2
  + \frac12 \Delta t
  \frac{1}{\lambda} \int_{\Omega_P}
  \left(\partial_t^n(\alpha \, e_{p_P,h} - e_{\varphi,h})\right)^2
  = \cL^n,
  \end{split}
\end{equation}
where $\cL^n$ is the collection of terms that appear on the right-hand sides in \eqref{eq:error}. Using the continuity of the bilinear forms 
(cf. Lemma~\ref{lem:cont}) in combination with Young's inequality, we have
\begin{equation}\begin{split}\label{discr-error-2}
    \cL^n & \le \epsilon\big(\|e_{\bu,h}^n\|_{\bH^1(\Omega_F)}^2
    + \|e_{p_F,h}^n\|_{L^2(\Omega_F)}^2 + \|e_{p_P,h}^n\|_{H^1(\Omega_P)}^2
    + \|e_{\varphi,h}^n\|_{L^2(\Omega_P)}^2\big)
    \\
    & \quad
    + C_\epsilon\big(\|e_{\bu,I}^n\|_{\bH^1(\Omega_F)}^2
    + \|e_{p_F,I}^n\|_{L^2(\Omega_F)}^2 + \|e_{p_P,I}^n\|_{H^1(\Omega_P)}^2
    + \|\partial_t^n e_{\bd,I}\|_{\bH^1(\Omega_P)}^2
    + \|\partial_t^n e_{p_P,I}\|_{H^1(\Omega_P)}^2
    \\
    & \quad
    + \|\partial_t^n e_{\varphi,I}\|_{L^2(\Omega_P)}^2
    + \|r_{\bd}^n\|_{\bL^2(\Sigma)}^2
    + \|r_{p_P}^n\|_{L^2(\Omega_P)}^2 + \|r_{\varphi}^n\|_{L^2(\Omega_P)}^2 \big)
    + \cH(\partial_t^n e_{\bd,h}),    
\end{split}
\end{equation}
where $\cH(\partial_t^n e_{\bd,h})$ is the collection of terms on the right-hand side of \eqref{error-3} with $w_h = \partial_t^n e_{\bd,h}$. Multiplying \eqref{discr-error-1} by $\Delta t$, summing for $n$ from 1 to $k \in \{1,\ldots,N\}$,
and using \eqref{discr-error-2}, we obtain
\begin{align}\label{discr-error-3}
    & \Delta t \sum_{n=1}^k \Big(\|e_{\bu,h}^n\|_{\bH^1(\Omega_F)}^2
    + \sum_{j=1}^{d-1}\|(e_{\bu,h}^n - \partial_t^n e_{\bd,h}) \cdot\bt^j\|_{L^2(\Sigma)}^2
    + \|e_{p_P,h}^n\|_{H^1(\Omega_P)}^2 \Big)\nonumber \\
    & \quad + \|e_{\bd,h}^k\|_{\bH^1(\Omega_P)}^2 + \|e_{p_P,h}^k\|_{L^2(\Omega_P)}^2
    + \frac{1}{\lambda}\|\alpha e_{p_P,h}^k - e_{\varphi,h}^k\|_{L^2(\Omega_P)}^2\nonumber \\
    & \quad \le C \Delta t \sum_{n=1}^k \Big(
    \epsilon\big(\|e_{\bu,h}^n\|_{\bH^1(\Omega_F)}^2
    + \|e_{p_F,h}^n\|_{L^2(\Omega_F)}^2 + \|e_{p_P,h}^n\|_{H^1(\Omega_P)}^2
    + \|e_{\varphi,h}^n\|_{L^2(\Omega_P)}^2\big)
    \nonumber \\
    & \qquad
    + C_\epsilon\big(\|e_{\bu,I}^n\|_{\bH^1(\Omega_F)}^2
    + \|e_{p_F,I}^n\|_{L^2(\Omega_F)}^2 + \|e_{p_P,I}^n\|_{H^1(\Omega_P)}^2
    + \|\partial_t^n e_{\bd,I}\|_{\bH^1(\Omega_P)}^2
    + \|\partial_t^n e_{p_P,I}\|_{H^1(\Omega_P)}^2
    \nonumber \\
    & \qquad
    + \|\partial_t^n e_{\varphi,I}\|_{L^2(\Omega_P)}^2
    + \|r_{\bd}^n\|_{\bL^2(\Sigma)}^2
    + \|r_{p_P}^n\|_{L^2(\Omega_P)}^2 + \|r_{\varphi}^n\|_{L^2(\Omega_P)}^2\big)
    + \cH(\partial_t^n e_{\bd,h}) \Big) \nonumber \\
    & \qquad
    + C \big(\|e_{\bd,h}^0\|_{\bH^1(\Omega_P)}^2 + \|e_{p_P,h}^0\|_{L^2(\Omega_P)}^2
    + \frac{1}{\lambda}\|\alpha e_{p_P,h}^0 - e_{\varphi,h}^0\|_{L^2(\Omega_P)}^2\big),   
\end{align}
where we also used Korn's inequality \eqref{Korn}, the Poincar\'e inequality
\eqref{Poincare}, and the permeability bound \eqref{perm}.
Next, for each term in
$\cH(\partial_t^n e_{\bd,h})$ we use summation by parts:
\begin{equation*}
  \Delta t \sum_{n=1}^k \phi^n \partial_t^n e_{\bd,h}
  = \phi^k e_{\bd,h}^k - \phi^1 e_{\bd,h}^0
  - \Delta t \sum_{n=2}^k \partial_t^n \phi \, e_{\bd,h}^{n-1},
\end{equation*}
where $\phi$ stands for any of the functions that appear in $\cH(\partial_t^n e_{\bd,h})$. Then, for the first term in $\cH(\partial_t^n e_{\bd,h})$ we write, using Young's inequality,
\begin{equation}\begin{split}\label{discr-error-4}
  & \Delta t \sum_{n=1}^k b_3^\Sigma(e_{\bu,I}^n,\partial_t^n e_{\bd,h}) = b_3^\Sigma(e_{\bu,I}^k,e_{\bd,h}^k) - b_3^\Sigma(e_{\bu,I}^1,e_{\bd,h}^0)
  - \Delta t \sum_{n=2}^k b_3^\Sigma(\partial_t^ne_{\bu,I},e_{\bd,h}^{n-1}) \\
  & \quad \le \epsilon \|e_{\bd,h}^k\|_{\bH^1(\Omega_P)}^2
  + C_\epsilon \|e_{\bu,I}^k\|_{\bH^1(\Omega_F)}^2
  + C\big(\|e_{\bd,h}^0\|_{\bH^1(\Omega_P)}^2 + \|e_{\bu,I}^1\|_{\bH^1(\Omega_F)}^2\big) \\
  & \qquad + C\Delta t \sum_{n=2}^{k} \big(\|e_{\bd,h}^{n-1}\|_{\bH^1(\Omega_P)}^2 
  + \|\partial_t^n e_{\bu,I}\|_{\bH^1(\Omega_F)}^2\big).
  \end{split}
\end{equation}
Now, combining \eqref{discr-error-3}--\eqref{discr-error-4}, and using for the rest of the terms in $\cH(\partial_t^n e_{\bd,h})$ bounds that are similar to \eqref{discr-error-4}, we arrive at
\begin{align}\label{discr-error-5}
    & \Delta t \sum_{n=1}^k \Big(\|e_{\bu,h}^n\|_{\bH^1(\Omega_F)}^2
    + \sum_{j=1}^{d-1}\|(e_{\bu,h}^n - \partial_t^n e_{\bd,h}) \cdot\bt^j\|_{L^2(\Sigma)}^2
    + \|e_{p_P,h}^n\|_{H^1(\Omega_P)}^2 \Big)\nonumber \\
    & \quad + \|e_{\bd,h}^k\|_{\bH^1(\Omega_P)}^2 + \|e_{p_P,h}^k\|_{L^2(\Omega_P)}^2
    + \frac{1}{\lambda}\|\alpha e_{p_P,h}^k - e_{\varphi,h}^k\|_{L^2(\Omega_P)}^2\nonumber \\
    & \quad \le C \Delta t \sum_{n=1}^k \Big(
    \epsilon\big(\|e_{\bu,h}^n\|_{\bH^1(\Omega_F)}^2
    + \|e_{p_F,h}^n\|_{L^2(\Omega_F)}^2 + \|e_{p_P,h}^n\|_{H^1(\Omega_P)}^2
    + \|e_{\varphi,h}^n\|_{L^2(\Omega_P)}^2\big)
    \nonumber \\
    & \qquad
    + \|e_{\bu,I}^n\|_{\bH^1(\Omega_F)}^2
    + \|e_{p_F,I}^n\|_{L^2(\Omega_F)}^2 + \|e_{p_P,I}^n\|_{H^1(\Omega_P)}^2
    + \|r_{\bd}^n\|_{\bL^2(\Sigma)}^2
    + \|r_{p_P}^n\|_{L^2(\Omega_P)}^2
    + \|r_{\varphi}^n\|_{L^2(\Omega_P)}^2
    \nonumber \\
    & \qquad
    + \|\partial_t^n e_{\bd,I}\|_{\bH^1(\Omega_P)}^2
    + \|\partial_t^n e_{p_P,I}\|_{H^1(\Omega_P)}^2
+ \|\partial_t^n e_{\varphi,I}\|_{L^2(\Omega_P)}^2
+ \|\partial_t^n e_{\bu,I}\|_{\bH^1(\Omega_F)}^2 \Big) \nonumber \\
& \qquad + C \Big( \Delta t \sum_{n=2}^{k} \big(\|e_{\bd,h}^{n-1}\|_{\bH^1(\Omega_P)}^2
+ \|\partial_t^n \partial_t^n e_{\bd,I}\|_{\bH^1(\Omega_P)}^2
+ \|\partial_t^n r_{\bd}\|_{\bL^2(\Sigma)}^2 \big)
\nonumber \\
& \qquad \quad
+\|e_{\bd,h}^0\|_{\bH^1(\Omega_P)}^2 + \|e_{p_P,h}^0\|_{L^2(\Omega_P)}^2
+ \frac{1}{\lambda}\|\alpha e_{p_P,h}^0 - e_{\varphi,h}^0\|_{L^2(\Omega_P)}^2
+ \|\partial_t^k e_{\bd,I}\|_{\bH^1(\Omega_P)}^2
+ \|\partial_t^1 e_{\bd,I}\|_{\bH^1(\Omega_P)}^2
\nonumber \\
& \qquad \quad
+ \|e^k_{\bu,I}\|_{\bH^1(\Omega_F)}^2 + \|e^k_{\bd,I}\|_{\bH^1(\Omega_P)}^2
+ \|e^k_{p_P,I}\|_{H^1(\Omega_P)}^2 + \|e^k_{\varphi,I}\|_{L^2(\Omega_P)}^2
+ \|r^k_{\bd}\|_{\bL^2(\Sigma)}^2
\nonumber \\
& \qquad \quad
+ \|e^1_{\bu,I}\|_{\bH^1(\Omega_F)}^2 + \|e^1_{\bd,I}\|_{\bH^1(\Omega_P)}^2
+ \|e^1_{p_P,I}\|_{H^1(\Omega_P)}^2 + \|e^1_{\varphi,I}\|_{L^2(\Omega_P)}^2
+ \|r^1_{\bd}\|_{\bL^2(\Sigma)}^2
\Big).
\end{align}
We continue with bounding $\|e_{p_F,h}^n\|_{L^2(\Omega_F)}^2$ and $\|e_{\varphi,h}^n\|_{L^2(\Omega_P)}^2$, which appear on the right-hand side above. The inf-sup condition \eqref{inf-sup-stokes} and \eqref{error-1} imply
\begin{align*}
&    \Delta t \sum_{n=1}^k \|e_{p_F,h}^n\|_{L^2(\Omega_F)}^2 \le C \Delta t \sum_{n=1}^k
    \Big( \|e_{\bu,h}^n\|_{\bH^1(\Omega_F)}^2
    + \sum_{j=1}^{d-1}\|(e_{\bu,h}^n - \partial_t^n e_{\bd,h}) \cdot\bt^j\|_{L^2(\Sigma)}^2
    + \|e_{p_P,h}^n\|_{H^1(\Omega_P)}^2
\\
    &\quad
        + \|e_{\bu,I}^n\|_{\bH^1(\Omega_F)}^2
    + \|e_{p_F,I}^n\|_{L^2(\Omega_F)}^2 + \|e_{p_P,I}^n\|_{H^1(\Omega_P)}^2
    + \|\partial_t^n e_{\bd,I}\|_{\bH^1(\Omega_P)}^2
    + \|r_{\bd}^n\|_{\bL^2(\Sigma)}^2 \Big).
  \end{align*}
On the other hand, 
the inf-sup condition \eqref{inf-sup-elast} and \eqref{error-3} allow us to get 
\begin{equation}\begin{split}\label{phi-discr}
    & \Delta t \sum_{n=1}^k \|e_{\varphi,h}^n\|_{L^2(\Omega_P)}^2
    \le C \Delta t \sum_{n=1}^k \Big(
    \|e_{p_P,h}^n\|_{H^1(\Omega_P)}^2
    + \sum_{j=1}^{d-1}\|(e_{\bu,h}^n - \partial_t^n e_{\bd,h}) \cdot\bt^j\|_{L^2(\Sigma)}^2
    + \|e_{\bd,h}^n\|_{\bH^1(\Omega_P)}^2\\
    & \
    + \|e_{\bu,I}^n\|_{\bH^1(\Omega_F)}^2
    + \|e_{p_P,I}^n\|_{H^1(\Omega_P)}^2
    + \|e_{\bd,I}^n\|_{\bH^1(\Omega_P)}^2
    + \|\partial_t^n e_{\bd,I}\|_{\bH^1(\Omega_P)}^2
    + \|e_{\varphi,I}^n\|_{L^2(\Sigma)}^2
    + \|r_{\bd}^n\|_{\bL^2(\Sigma)}^2 \Big).
  \end{split}
\end{equation}
Before combining \eqref{discr-error-5}--\eqref{phi-discr}, we note that
the terms involving $\partial_t^n$  on the right-hand sides
require special treatment. In particular, it holds for $\phi(t)$ that
\begin{equation*}
  (\partial_t^n \phi)^2
  = \frac{1}{\Delta t^2}
  \Big(\int_{t_{n-1}}^{t_n} \partial_t \phi \Big)^2
\le \frac{1}{\Delta t}
\int_{t_{n-1}}^{t_n} (\partial_t \phi)^2,
\end{equation*}
implying, for $\phi:[0,T] \to V$, where $V$ is a Banach space with norm $\|\cdot\|_V$, that
\begin{equation*}
  \Delta t \sum_{n=1}^k \|\partial_t^n \phi\|_V^2
  \le \int_0^{t_k} \|\partial_t \phi\|_V^2.
\end{equation*}
%
We then have
\begin{equation}\begin{split}\label{discr-error-6}
& \Delta t \sum_{n=1}^k \big(\|\partial_t^n e_{\bd,I}\|_{\bH^1(\Omega_P)}^2
    + \|\partial_t^n e_{p_P,I}\|_{H^1(\Omega_P)}^2
+ \|\partial_t^n e_{\varphi,I}\|_{L^2(\Omega_P)}^2
+ \|\partial_t^n e_{\bu,I}\|_{\bH^1(\Omega_F)}^2 \big) \\
& \quad \le C \int_0^{t_k}\big(
\|\partial_t e_{\bd,I}\|_{\bH^1(\Omega_P)}^2
    + \|\partial_t e_{p_P,I}\|_{H^1(\Omega_P)}^2
+ \|\partial_t e_{\varphi,I}\|_{L^2(\Omega_P)}^2
+ \|\partial_t e_{\bu,I}\|_{\bH^1(\Omega_F)}^2 \big).
\end{split}\end{equation}
To bound the term $\|\partial_t^n \partial_t^n e_{\bd,I}\|_{\bH^1(\Omega_P)}^2$ in \eqref{discr-error-5}, for any $\phi(t)$ we have,
using the integral mean value theorem and the mean value theorem,
$$
(\partial_t^n \partial_t^n \phi)^2
= \frac{1}{\Delta t^4} \Big(\int_{t_{n-1}}^{t^n} \partial_t \phi
- \int_{t_{n-2}}^{t^{n-1}} \partial_t \phi \Big)^2
= \frac{1}{\Delta t^2}\big(\partial_t \phi(\xi^n) - \partial_t \phi(\xi^{n-1})\big)^2
= \partial_{tt} \phi(\xi), \quad \xi \in [t_{n-2},t_n].
$$
Therefore it holds that
\begin{equation*}
  \Delta t \sum_{n=2}^k \|\partial_t^n \partial_t^n e_{\bd,I}\|_{\bH^1(\Omega_P)}^2
  \le C \mathop{\esssup}\limits_{t\in (0,t_k)}\|\partial_{tt}e_{\bd,I}\|_{\bH^1(\Omega_P)}^2.
\end{equation*}
Next, we need to bound the time discretisation error. \cred{Taylor's} expansion gives
$$
r_\phi^n = \frac{1}{\Delta t}\int_{t_n}^{t_{n-1}} \partial_{tt} \phi(t)(t_{n-1} - t) \dt,
$$
thus, for $\phi:[0,T] \to V$,
\begin{equation*}
  \|r_\phi^n\|_V \le C \Delta t \mathop{\esssup}\limits_{t\in (t_{n-1},t_n)}
  \|\partial_{tt} \phi\|_V \quad \text{and} \quad 
  \Delta t \sum_{n=1}^k \|r_\phi^n\|_V^2
  \le C \Delta t^2 \int_0^{t_k} \|\partial_{tt} \phi\|_V^2.
\end{equation*}
Similarly,
\begin{equation*}
  \Delta t \sum_{n=1}^k \|\partial_t^n r_\phi\|_V^2
  \le C \Delta t^2 \int_0^{t_k} \|\partial_{ttt} \phi\|_V^2. 
\end{equation*}
Finally, we need a bound on the initial discretisation error. Recalling the construction of the discrete initial data in the proof of
Theorem~\ref{thm:well-posed} and the definition of the continuous initial data in the proof of Theorem~\ref{thm:well-posed-weak}, we note that $(\bu_h^0,p_{F,h}^0)$ is the Stokes elliptic projection of $(\bu(0),p_{F,h}(0))$ based on \eqref{method-1}--\eqref{method-2} at $t=0$ with $\partial_t \bd_h = 0$ and
the term $b_2^\Sigma(p_{P,0} - p_{P,h,0},\bv_h)$ on the right-hand side. In addition,
$(\bd_h^0,\varphi_h^0)$ is the elliptic projection of $(\bd(0),\varphi(0))$ based on the stable problem \eqref{method-3}--\eqref{method-5} at $t = 0$ with $\partial_t \bd_h = 0$ and the terms $b_3^\Sigma(\bu(0) - \bu_h^0,\bw_h)$,
$b_4^\Sigma(\bw_h,p_{P,0} - p_{P,h,0})$, and $b_2^P(\psi_h,p_{P,0} - p_{P,h,0})$ on the right-hand side. Classical finite element analysis for these two problems implies that
\begin{equation}\begin{split}\label{init-error}
&  \|e_{\bu,h}^0\|_{\bH^1(\Omega_F)} + \|e_{p_F,h}^0\|_{L^2(\Omega_F)}
  + \|e_{\bd,h}^0\|_{\bH^1(\Omega_P)} + \|e_{p_P,h}^0\|_{L^2(\Omega_P)}
  + \|e_{\varphi,h}^0\|_{L^2(\Omega_P)} \\
& \quad \le C \big(\|e_{\bu,I}^0\|_{\bH^1(\Omega_F)} + \|e_{p_F,I}^0\|_{L^2(\Omega_F)}
  + \|e_{\bd,I}^0\|_{\bH^1(\Omega_P)} + \|e_{p_P,I}^0\|_{L^2(\Omega_P)}
  + \|e_{\varphi,I}^0\|_{L^2(\Omega_P)} \big).
  \end{split}
\end{equation}
The assertion of the theorem follows from combining \eqref{discr-error-5}--\eqref{init-error}, using the discrete Gronwall inequality \cite[Lemma~1.4.2]{QV-NumPDE} for the term $\Delta t \sum_{n=2}^{k} \|e_{\bd,h}^{n-1}\|_{\bH^1(\Omega_P)}^2$, 
and applying the triangle inequality and the approximation properties \eqref{IV-approx}--\eqref{IZ-approx}.
\end{proof}

\section{Representative computational results}\label{sec:results}
All routines have been implemented using the open source finite element library \texttt{FEniCS} \cite{alnaes15}, as well as the specialised module \texttt{multiphenics} \cite{multiphenics} for handling subdomain- and boundary- restricted terms that 
we require to impose transmission conditions across interfaces. The solvers are monolithic and the solution of all linear 
systems is performed with the distributed direct solver MUMPS. We present \cred{four} examples: convergence tests (example 1), channel flow behaviour (example 2),  \cred{a simulation of subsurface flow with heterogeneous random permeability (example 3), and} the solution of an axisymmetric problem using parameters relevant to eye poromechanics \cred{(example 4). For examples 1,2 and 4 we use Taylor-Hood elements for the pairs velocity-pressure and displacement - total pressure, plus continuous and piecewise  quadratic elements for Biot fluid pressure. For example 3, the inf-sup stable pair used is the MINI element, and the Biot fluid pressure is approximated with continuous and piecewise linear elements.} 

\subsection{Convergence tests against manufactured solutions}
The accuracy of the spatio-temporal discretisation is verified using the following closed-form solutions 
defined on the domains $\Omega_F = (-1,1)\times(0,2)$, $\Omega_P = (-1,1)\times(-2,0)$, 
separated by the interface $\Sigma = (-1,1)\times\{0\}$
\begin{equation}\label{eq:manuf}\begin{split}
\bu = \sin(t)\begin{pmatrix}
-\cos(\pi x)\sin(\pi y)\\
\sin(\pi x)\cos(\pi y)\end{pmatrix}, \quad 
	p_F = \sin(t)\cos(\pi x)\cos(\pi y), \\
\bd = \cos(t)\curl(\sin(\pi xy)), \quad 
	p_P = \cos(t)\sin(\pi x)\sin(\pi y), \quad  \varphi  = \alpha p_P- \lambda \vdiv\bd.
\end{split}\end{equation}

\begin{table}[!t]
	\setlength{\tabcolsep}{3pt}
	\renewcommand{\arraystretch}{1.2}
	\centering 
	\begin{tabular}{|r|ccccccccccc|}
		\hline 
		DoF & $h$ & $\texttt{e}_{\bu}$ & \texttt{rate} & $\texttt{e}_{p_F}$ & \texttt{rate} & $\texttt{e}_{\bd}$ & \texttt{rate} & $\texttt{e}_{p_P}$ & \texttt{rate} & $\texttt{e}_{\varphi}$ & \texttt{rate} \tabularnewline
		\hline
		\hline
144 & 1.414 &     4.70604 & -- & 0.82152 &  -- & 21.4632 & -- &  2.80468 &  -- & 19.3402 &   -- \tabularnewline
456 & 0.7071 &    1.66701 & 1.497 & 0.29604 & 1.472 & 9.75813 & 1.137 & 0.84917 & 1.721 & 8.20939 & 1.236\tabularnewline
1608 & 0.3536 &   0.40411 & 2.044 & 0.06081 & 2.260 & 2.78206 & 1.811 & 0.24942 & 1.769 & 1.21860 & 2.422\tabularnewline
6024 & 0.1768 &   0.09746 & 2.051 & 0.01376 & 2.138 & 0.72730 & 1.936 & 0.06567 & 1.943 & 0.25337 & 2.240\tabularnewline
23304 & 0.0884 &  0.02405 & 2.014 & 0.00328 & 2.052 & 0.18545 & 1.984 & 0.01661 & 1.985 & 0.05166 & 2.231\tabularnewline
91656 & 0.0442 &  0.00597 & 2.005 & 0.00081 & 2.018 & 0.04683 & 1.991 & 0.00410 & 1.991 & 0.00472 & 2.098\tabularnewline
363528 & 0.0221 & 0.00151 & 2.001 & 0.00022 & 2.008 & 0.01172 & 1.998 & 0.00100 & 1.995 & 0.00061 & 2.015\tabularnewline
\hline 
	\end{tabular} 
\caption{Example 1. Experimental errors associated with the spatial discretisation and convergence rates for the approximate solutions 
		$\bu_h$, $p_{F,h}$, $\bd_h$,  $p_{P,h}$, and $\varphi_h$ using $\mathbb{P}^2_2-\mathbb{P}_1-\mathbb{P}^2_2-\mathbb{P}_2-\mathbb{P}_1$ elements and computed at the last time step.} \label{table:ex01}
\end{table}

We consider $\Gamma_P^{p}$ to be the bottom horizontal segment, $\Gamma_P^{\bd}$ to be the lower vertical walls, 
$\Gamma_F^{\bu}$ to be the top horizontal segment, and 
$\Gamma_F^{\bsigma}$ to be the upper vertical walls.
the synthetic model parameters are taken as 
\begin{gather*} 
\lambda = 1000, \quad \mu_s = 1, \quad \mu_f = 0.1, \quad 
\alpha =   \gamma = 1, \quad  
c_0 = 0.01, \quad \kappa = 0.001, \quad \rho_s = 1.2, \quad \rho_f = 1, \end{gather*}
all regarded adimensional and do not have physical relevance in this case, as we will be simply testing the convergence of the finite element approximations. The manufactured solutions \eqref{eq:manuf} are used to prescribe initial conditions, essential non-homogeneous velocity and displacement 
boundary conditions, as well as natural non-homogeneous flux conditions for fluid pressure. These 
functions do not necessarily 
fulfil the interface conditions, so additional terms are required giving modified relations on $\Sigma$:
\begin{align*}
\bu\cdot\nn  - (\partial_t \bd  - \frac{\kappa}{\mu_f} \nabla p_P )\cdot \nn & = m^1_{\Sigma,\text{ex}} ,\qquad 
(2\mu_f \beps(\bu) - p_F\bI)\nn - (2\mu_s\beps(\bd) -\varphi \bI ) \nn  = \bm^2_{\Sigma,\text{ex}},\\
\nn\cdot (2\mu_f \beps(\bu) - p_F\bI)\nn +  \cred{p_P}  &= m^3_{\Sigma,\text{ex}},\qquad 
 \nn\cdot (2\mu_f \beps(\bu) - p_F\bI)\bt + \frac{\gamma\mu_f}{\sqrt{\kappa}} (\bu - \partial_t \bd)\cdot \bt   = m^4_{\Sigma,\text{ex}},
\end{align*}
and the additional scalar and vector terms $m^i_{\Sigma,\text{ex}}$ (computed with the exact solutions \eqref{eq:manuf}) entail the following changes in the 
linear functionals 
\begin{gather*}
F^F(\bv) =  \rho_f\int_{\Omega_F} \gg\cdot\bv + \langle m^3_{\Sigma,\text{ex}},\bv\cdot\nn\rangle_\Sigma
+ \langle m^4_{\Sigma,\text{ex}},\bv\cdot\bt\rangle_\Sigma,  \\
 F^P(\bw) = \rho_s\int_{\Omega_P} \ff\cdot\bw + \int_{\Omega_P} \bm^2_{\Sigma,\text{ex}}\cdot\bw 
 + \langle m^3_{\Sigma,\text{ex}},\bw\cdot\nn\rangle_\Sigma
 + \langle m^4_{\Sigma,\text{ex}},\bw\cdot\bt\rangle_\Sigma,\\
G(q_P) =  \int_{\Omega_P} \rho_f\gg\cdot\nabla q_P - \langle \rho_f\gg\cdot\nn,q_P\rangle_\Sigma - \langle m^1_{\Sigma,\text{ex}},q_P\rangle_\Sigma.
  \end{gather*}

\begin{table}[t]
		\setlength{\tabcolsep}{3pt}
		\renewcommand{\arraystretch}{1.2}
		\centering 
		\begin{tabular}{|r|cccccccccc|}
			\hline 
			$\Delta t$ & $\hat{\texttt{e}}_{\bu}$ & \texttt{rate} & $\hat{\texttt{e}}_{p_F}$ & \texttt{rate} & $\hat{\texttt{e}}_{\bd}$ & \texttt{rate} & $\hat{\texttt{e}}_{p_P}$ & \texttt{rate} 
			& $\hat{\texttt{e}}_{\varphi}$ & \texttt{rate}  \tabularnewline
			\hline
			\hline
	           0.5 & 10.549 &   -- & 2.5844 &      -- &  43.764 &      -- &  8.6738 &      -- & 38.734 &      --\tabularnewline
	      0.25 & 5.1408 &    0.984 & 1.2710 &    1.163 & 21.840 &    1.211 & 4.3673 &    1.125 & 19.371 &    1.055 \tabularnewline
	      0.125 & 2.2689 &   1.205 & 0.6485 &    1.182 & 10.661 &    1.204 & 2.1690 &    1.032 & 9.6901 &    0.962 \tabularnewline
	      0.0625 & 1.1365 &  1.107 & 0.3216 &    1.095 & 5.4517 &    1.121 & 1.0795 &    1.025 & 4.8614 &     0.981\tabularnewline
	      0.03125 & 0.6813 & 1.004 & 0.1615 &    1.071 & 2.7326 &    1.043 & 0.5527 &    0.987 & 2.4396 &    0.992\tabularnewline
			\hline 
		\end{tabular} 
		\caption{Example 1. Experimental cumulative errors associated with the temporal discretisation and convergence rates for the approximate solutions 
		$\bu_h$, $p_{F,h}$, $\bd_h$,  $p_{P,h}$, and $\varphi_h$ using a backward Euler scheme.} \label{table:ex01b}
	\end{table}

We generate successively refined simplicial grids and use a sufficiently small (non dimensional) time step $\Delta t = 0.01$   
and simulate a relatively short time horizon $t_{\mathrm{final}} = 3\Delta t$, to guarantee that 
the error produced by the time discretisation does not dominate. Errors between the 
approximate and exact solutions are tabulated against the number of degrees of freedom in Table~\ref{table:ex01}. This error history confirms the optimal convergence of the finite element scheme (in this case, second-order) 
for all variables in their respective norms, where a slightly better rate is seen for the total pressure. \cred{In the table caption, $\mathbb{P}_k$ denotes 
the space of piecewise polynomial functions being of total degree up to $k$. }

For this example we have reincorporated the acceleration and the nonlinear convective term, and the Newton-Raphson algorithm takes, in average, three iterations to reach the prescribed tolerance of $10^{-8}$ on the 
residuals.

The convergence in time achieved by the backward Euler method is verified by partitioning the \cred{time} interval $(0,1)$ into successively refined uniform discretisations and computing accumulated errors
\[ \hat{\texttt{e}}_s = \biggl( \sum_{n=1}^{N} \Delta t \|s(t_{n+1})- s_h^{n+1}\|^2_{\star}\biggr)^{1/2},\] 
where {$\| \cdot \|_\star$} denotes the appropriate space norm for the generic vector or scalar field $s$. For this test we use a fixed mesh involving 10K DoFs. 
The results are shown in Table \ref{table:ex01b}, confirming the optimal first-order convergence.

\subsection{Channel filtration and stress build-up on interface deformation} 

\begin{figure}[!t]
\begin{center}
\includegraphics[width=0.292\textwidth]{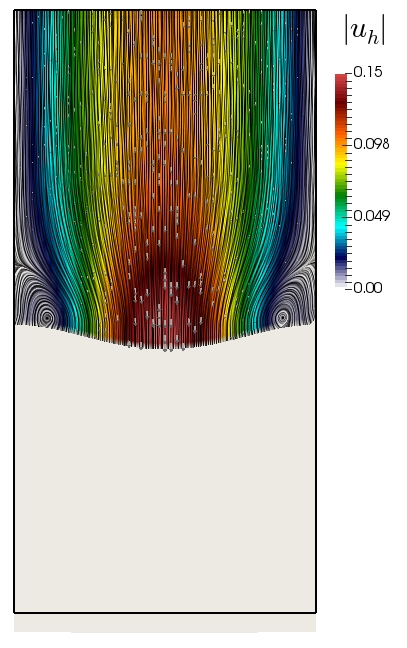}
\includegraphics[width=0.292\textwidth]{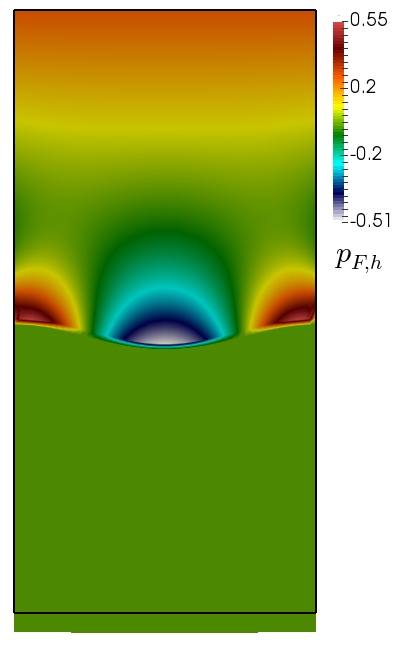}
\includegraphics[width=0.292\textwidth]{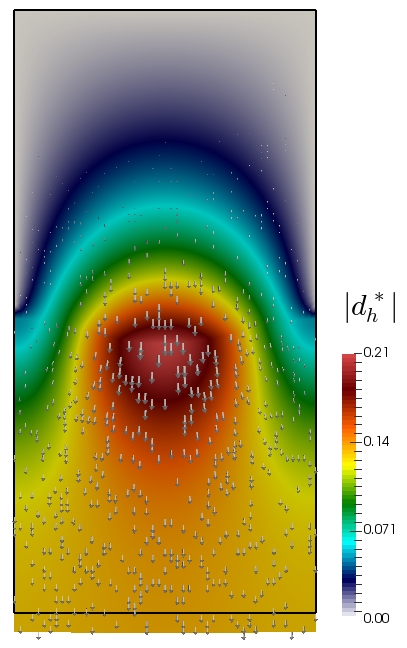}\\
\includegraphics[width=0.292\textwidth]{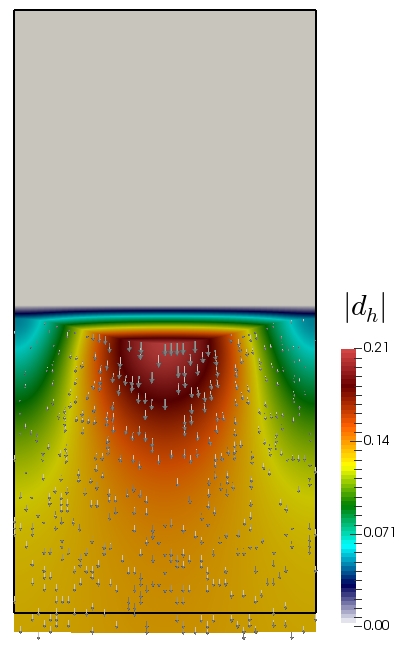}
\includegraphics[width=0.292\textwidth]{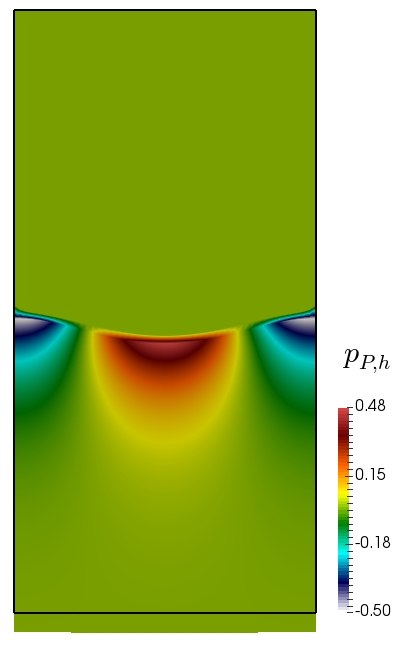}
\includegraphics[width=0.292\textwidth]{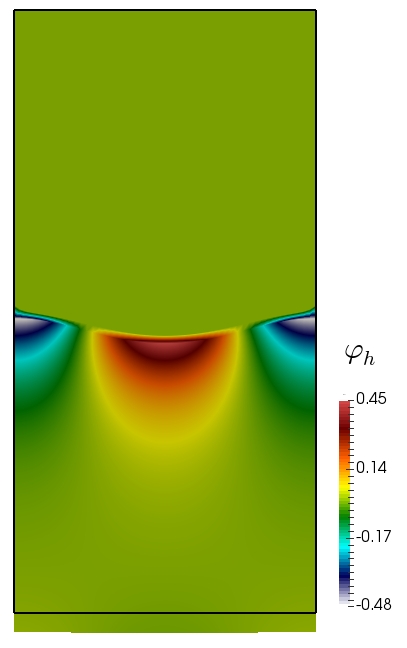}
\end{center}

\vspace{-3mm}
\caption{Example 2. Infiltration into a deformable porous medium. Velocity in the fluid domain, fluid pressure, extended displacement of fluid and solid domains, solid displacement, fluid pressure in the poroelastic domain, and total pressure. All snapshots are 
taken at time $t = 2$, and the black outer line indicates the location of the undeformed domain.}\label{ex02:results}
\end{figure}

{Although the model stated in Section~\ref{sec:model} holds in the limit of small strains, it is possible to have large displacements, likely located near the interface (and without violating the model assumptions). In this scenario, the discretisation might no longer be suitable. A simple remedy consists in smoothly moving the fluid domain and the fluid mesh to avoid distortions generated near the interface.} We use a standard harmonic extension (see e.g. \cite{bukac15}) 
that is solved at each time step, just after \eqref{eq:fully-discrete}: Find $\bd_h^* = \bd_h + \widehat{\bd}_h$ such that 
\begin{equation}\label{harmonic}
- D \Delta \widehat{\bd}_h = \cero \quad \text{ in $\Omega_F$}, \qquad \widehat{\bd}_h = \bd_h \text{ on $\Sigma$}, \qquad \text{ and } \quad  \widehat{\bd}_h = \cero \text{ on $\partial\Omega_F.$} 
\end{equation}
And then we perform an $L^2-$projection of both $\bd_h$ and $\widehat{\bd}_h$ into $\bW_h + \bV_h$ and add them to obtain the global displacement $\bd_h^*$. 

We illustrate the effect of using \eqref{harmonic} by looking at the behaviour of normal filtration into a 2D deformable porous medium. The same domains as in the accuracy tests are employed here (that is,  \cred{the single phase fluid} domain located on top of the poroelastic domain), however the 
boundary treatment is as follows, assuming that the flow is driven by pressure differences only. On the top segment we impose the fluid 
pressure $p_F^{\mathrm{in}} = p_0 \sin^2(\pi t)$ with $p_0 = 2$, and on the outlet (the bottom segment) 
the fluid pressure $p_P^{\mathrm{out}} = 0$. On the vertical walls of $\Omega_F$ 
we set $\bu = \cero$ while on the vertical walls of $\Omega_P$ we set the slip condition 
$\bd\cdot\nn = 0$ (and the porous structure is free to deform on the outlet boundary{, i.e., zero traction imposed}). 
The permeability is $\kappa=0.02$ and the remaining 
parameters are 
\begin{gather*}
\lambda = 10, \quad \mu_s = 5, \quad \rho_s = 1.1, \quad 
\rho_f = 1, \quad \alpha = 0.6, \quad 
\gamma = 0.1, \quad C_0= 0.01,
\end{gather*} 
and we assume that there are no body forces nor gravity acting on the system. In contrast with the convergence tests, for this example 
we use \cred{piecewise linear and continuous} finite elements for the approximation of $p_P$.

\begin{figure}[!t]
\begin{center}
\includegraphics[width=0.292\textwidth]{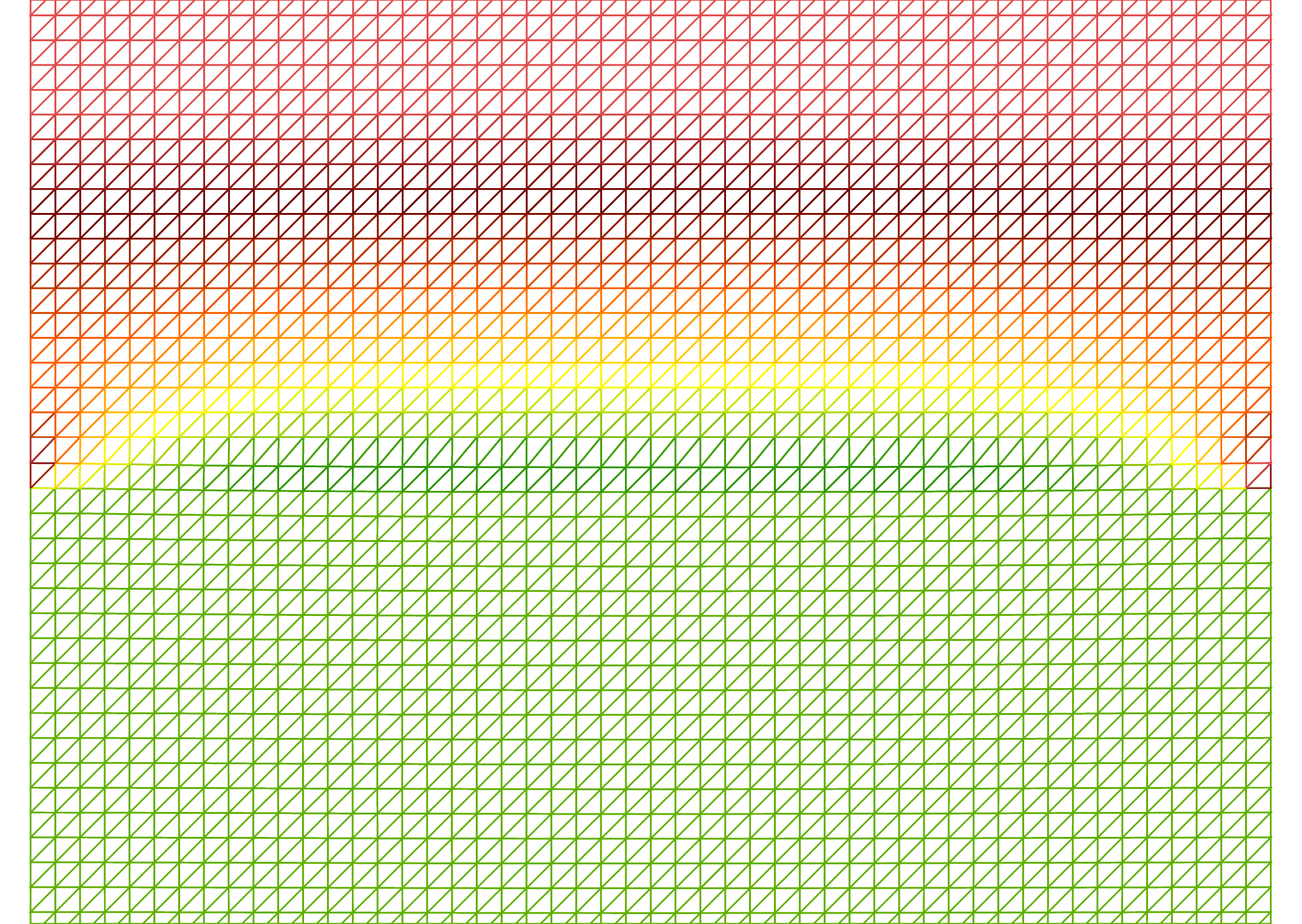}
\includegraphics[width=0.292\textwidth]{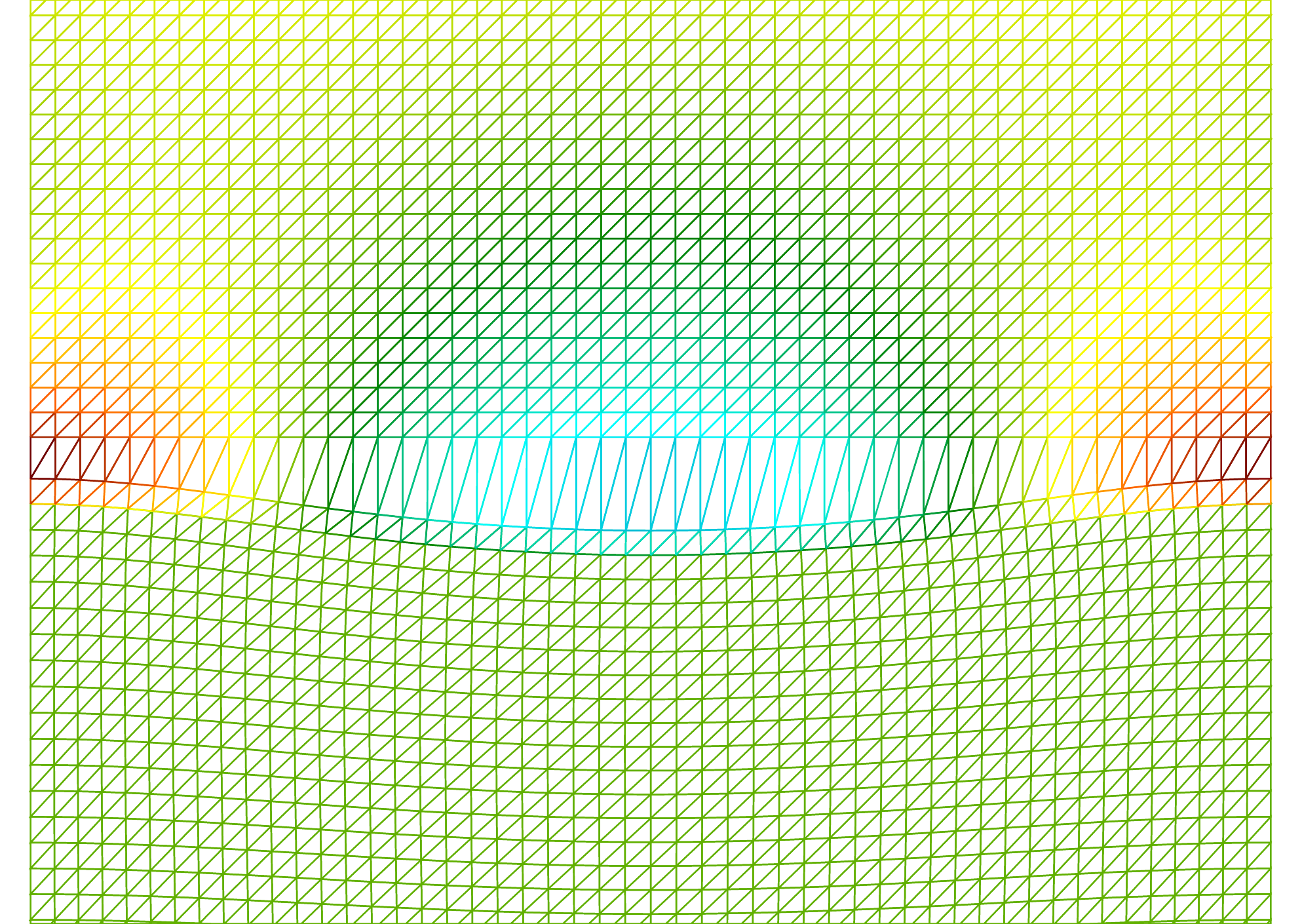}
\includegraphics[width=0.292\textwidth]{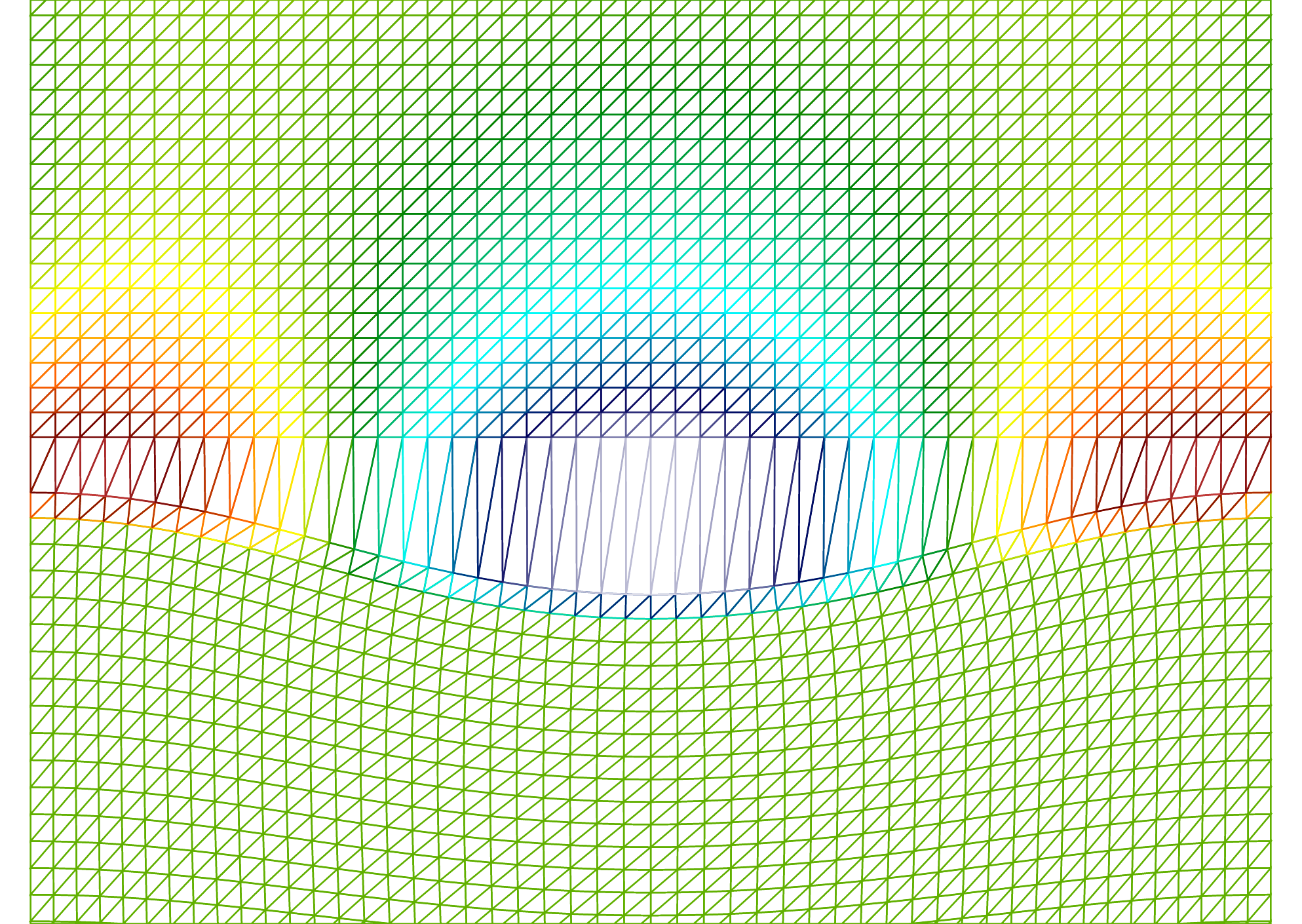}\\[1ex]
\includegraphics[width=0.292\textwidth]{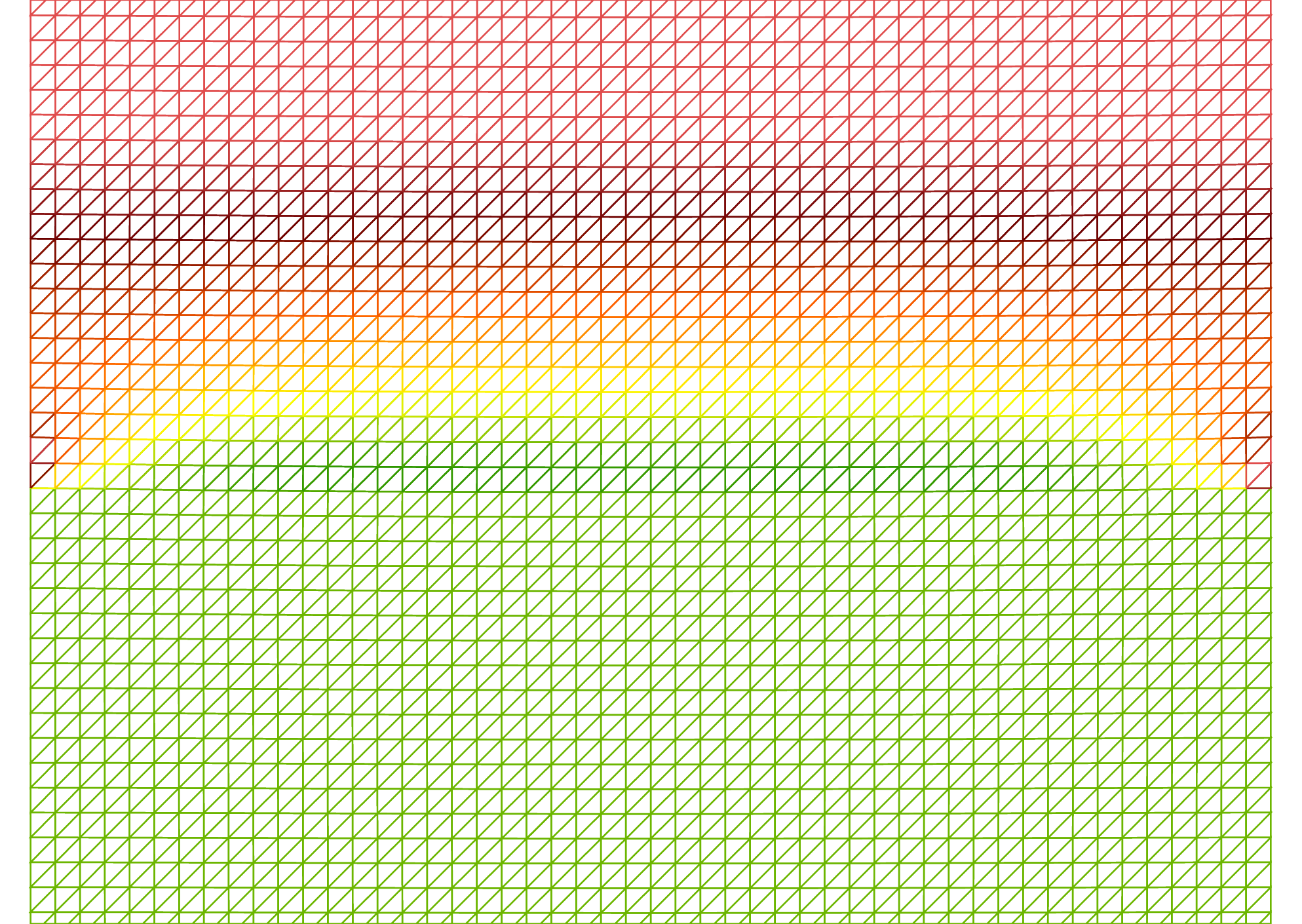}
\includegraphics[width=0.292\textwidth]{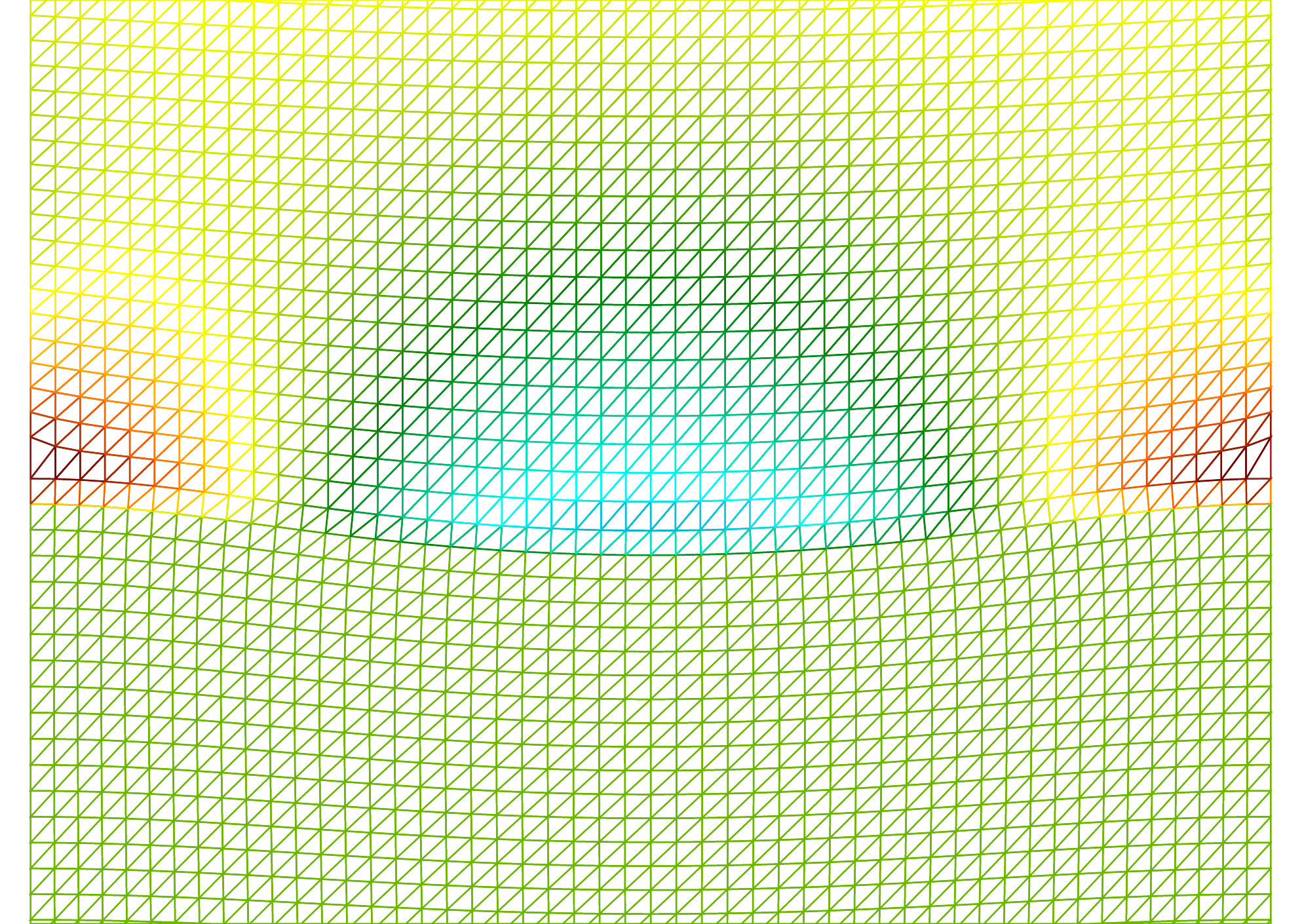}
\includegraphics[width=0.292\textwidth]{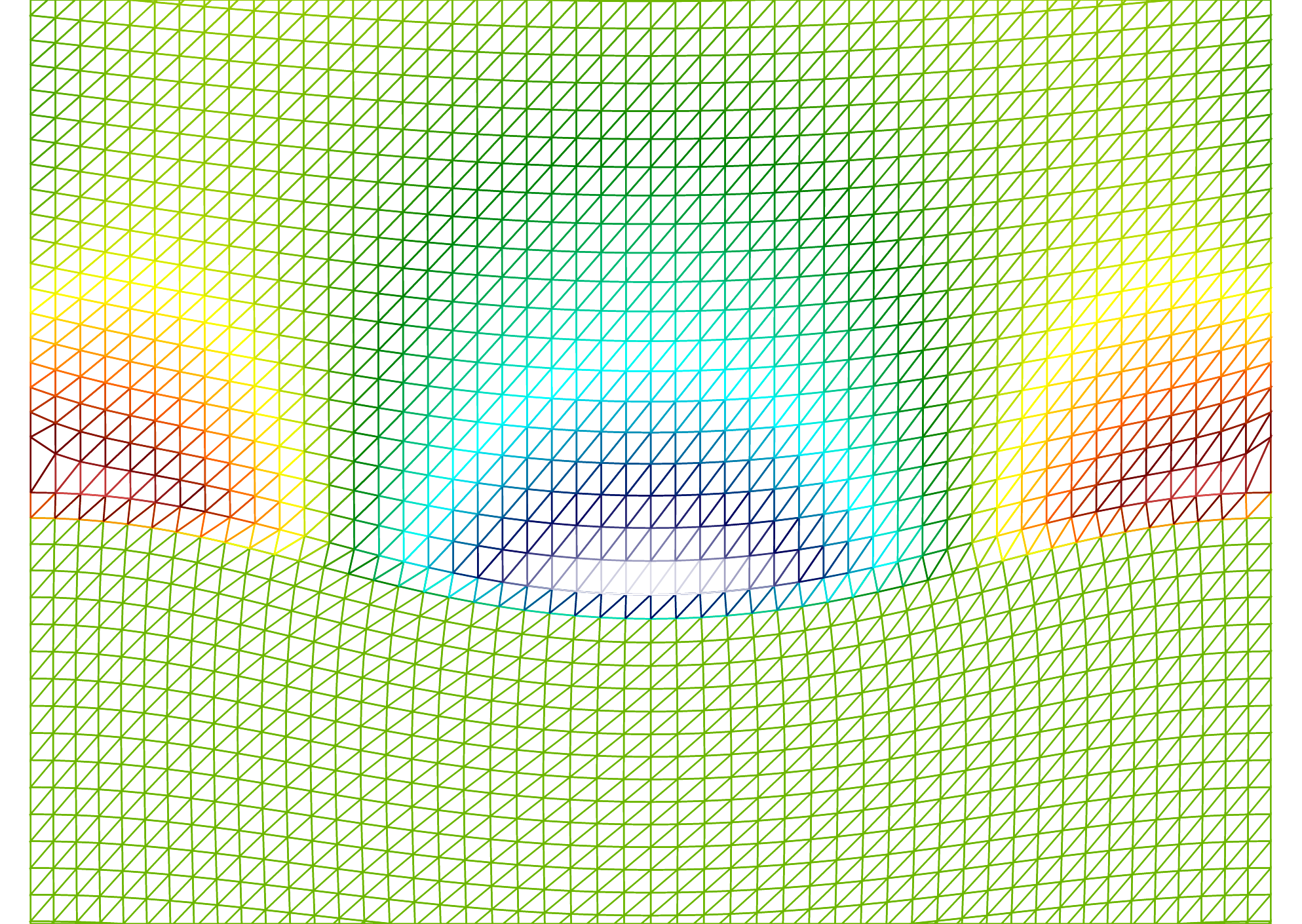}
\end{center}

\vspace{-3mm}
\caption{Example 2. Zoom of the meshes on the interface at times $t=0$, $t=1$ and $t=2$. Effect of using or not the harmonic extension to move the fluid domain (bottom and top, respectively).}\label{ex02:meshes}
\end{figure}

The numerical results are presented in Figure~\ref{ex02:results}. {The effect of the interface can be clearly seen in the top left panel where recirculation vortices replace the parabolic profile at the inlet; also, in the poroelastic domain, we see that, close to the interface, the solid displacement and the fluid pressure are heterogeneous in the horizontal direction, before recovering the expected constant value (constant in the horizontal direction) expected in the far field.}  We also plot the evolution of the mesh deformation near the interface. From Figure~\ref{ex02:meshes} one can see that for large enough {interfacial displacements}, the elements close to it exhibit a large distortion.

\cred{
\subsection{Simulation of subsurface fracture flow}
Next we include a test case that illustrates the applicability of the formulation in hydraulic fracturing. The problem setup follows \cite[Section 5.2.2]{ambar19} (except that we do not model tracer transport), considering a rectangular domain $\Omega = (0,3.048)\,[\text{m}]\times(0,6.096)$\,[m] including a relatively large fracture regarded as a macro void, or open channel $\Omega_F$ filled with an incompressible fluid (see Figure~\ref{fig:exF}, left), and the Biot domain is $\Omega_P = \Omega\setminus\Omega_F$. The heterogeneous (but isotropic in the $xy$-plane) permeability $\kappa(\bx)$  is the non-smooth pattern taken from the Cartesian SPE10 benchmark data / model 2 (see, e.g., \cite{aarnes07,christie01}), which we rescale as in \cite{ambar18} and project onto a piecewise constant field defined on an unstructured triangular mesh for the poroelastic geometry. There are 85 distinct layers within two general categories. We choose layer 80 from the dataset, which corresponds to the Upper Ness region exhibiting a fluvial fan pattern (flux channels of higher permeability and porosity). 
The Lam\'e parameters (in [KPa]) are highly heterogeneous and determined from the Poisson ratio $\nu = 0.2$ and the Young modulus $E (\bx)= 10^7(1-2\phi(\bx))^{2.1}$, where $\phi(\bx)$ is the porosity field also taken from layer 80 of the benchmark dataset (see Figure~\ref{fig:exF}, right). It is plotted in logarithmic scale and the contrast is of about $10^8$. No gravity and no external loads are considered, and 
the remaining parameters are $\alpha = \gamma = 1$, $c_0 = 6.89\cdot 10^{-2}\,[\text{KPa}^{-1}]$, $\mu_f = 10^{-6}\,[\text{KPa}\cdot\text{s}]$.

\begin{figure}[!t]
\begin{center}
\includegraphics[height=0.365\textwidth]{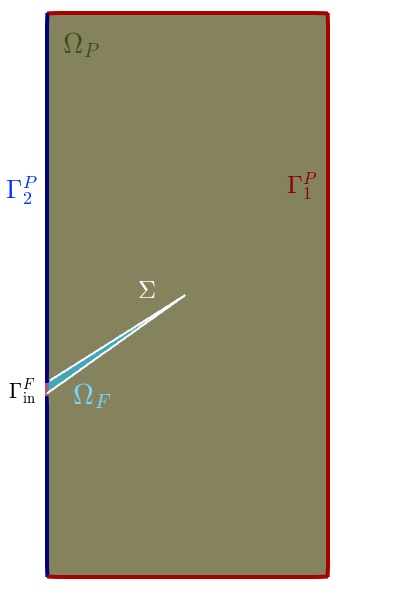}
\includegraphics[height=0.365\textwidth]{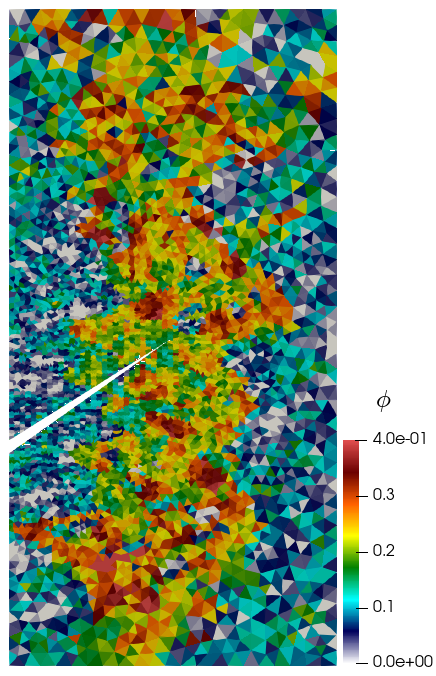} 
\includegraphics[height=0.365\textwidth]{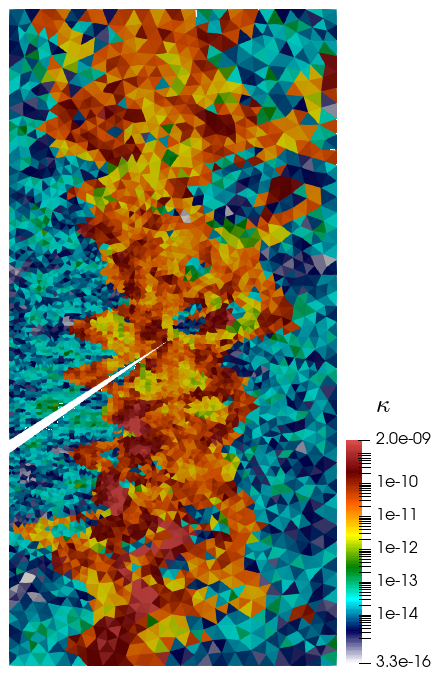} 
\includegraphics[height=0.365\textwidth]{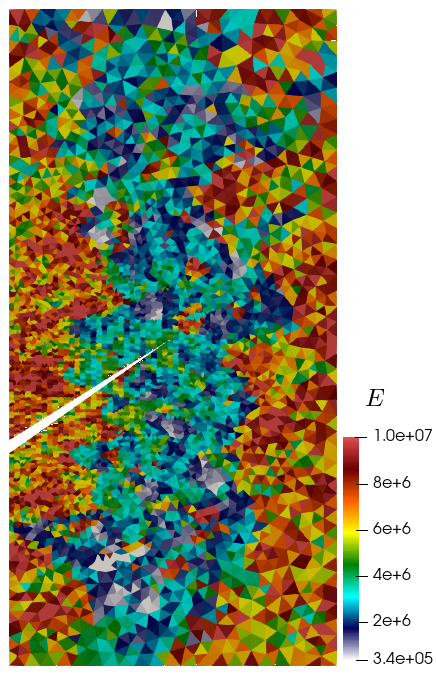}
\end{center}

\vspace{-3mm}
\caption{\cred{Example 3. Schematic representation of sub-domains, location of the free-porous interface, and configuration of sub-boundaries (left panel); and material properties (porosity $\phi(\bx)$, permeability $\kappa(\bx)$, and Young modulus $E(\bx)$) from layer 80 of the SPE10 benchmark dataset for reservoir simulations, herein projected onto a coarse mesh for the poroelastic sub-domain.}}\label{fig:exF}
\end{figure}

Similarly as in \cite{ambar18}, we set 
the flow initially at rest $\bd(0) = \cero$ and the initial Biot  pressures are $p_P(0) = \varphi(0) = 1000$\,[KPa] (since $\alpha = 1$). On the inlet boundary $\Gamma^F_{\mathrm{in}}$ (the vertical segment on the Stokes boundary) we impose the inflow velocity $\bu = (10,0)^{\tt t}\,$[m/s], on the bottom, right and top sub-boundaries $\Gamma^P_1$ of the Biot domain we prescribe sliding conditions $\bd\cdot\nn = 0$\,[m] together with a compatible normal-tangential stress condition and a fixed Biot pressure $p_P = 1000$\,[KPa], and on the right sub-boundaries $\Gamma^P_2$ of the Biot domain we set stress-free conditions. The time discretisation uses the fixed time step $\Delta t = 60$\,[s] and we run the simulation of injection over a period of $T = 10$\,[hours]. The unstructured triangular mesh has 1629 elements for the Stokes sub-mesh and 18897 elements for the Biot domain. For this test we have used the MINI element, consisting of continuous and piecewise linear elements with bubble enrichment for velocity and displacement, and continuous and piecewise linear elements for all remaining fields.  
The injected fluid imposes an increase of pressure on the interface and from there the expected channel-like progressive filtration from the Stokes to the Biot domain is clearly observed in the Biot pressure plot in the left panel of Figure~\ref{fig:exF-sols}, showing higher fluxes near the tip of the fracture. The remaining panels show, at the final time, snapshots of Stokes velocity, Biot displacement, and post-processed Biot filtration velocity. The deformation of the poroelastic structure  closer to the fracture is relatively small, which is why we do not consider in this test the harmonic extension of the fluid domain addressed in the previous example. 
}

\begin{figure}[!t]
\begin{center}
\includegraphics[height=0.46\textwidth]{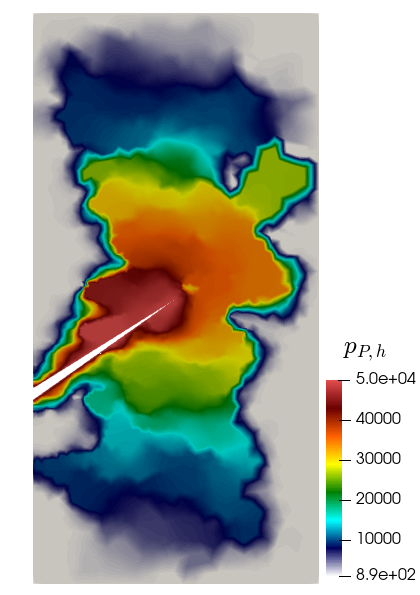}
\includegraphics[height=0.46\textwidth]{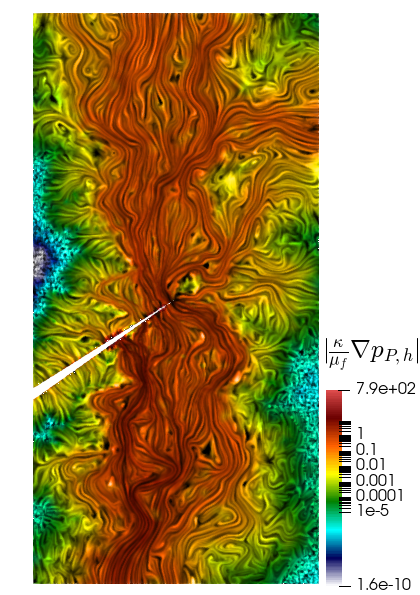}
\includegraphics[height=0.46\textwidth]{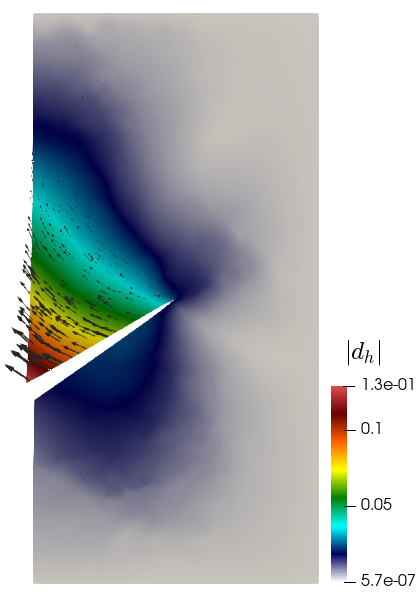}\\
\includegraphics[height=0.25\textwidth]{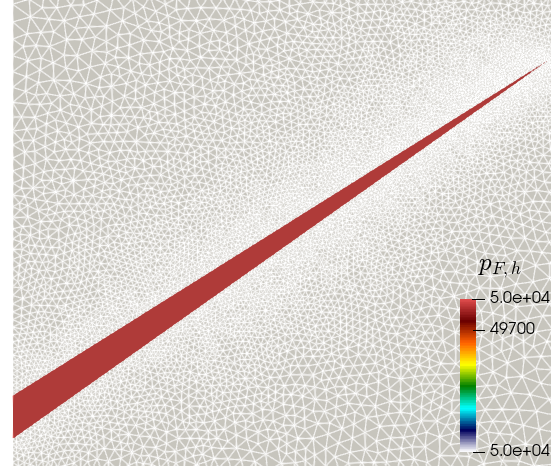}
\includegraphics[height=0.25\textwidth]{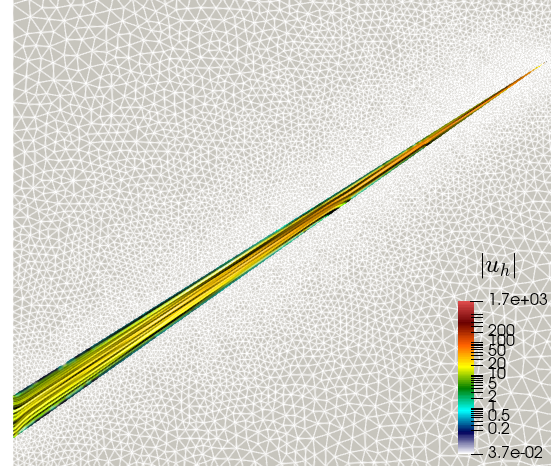}
\end{center}

\vspace{-3mm}
\caption{\cred{Example 3. Snapshots of the approximate solutions (Biot fluid pressure, post-processed Biot fluid velocity and line integral convolution, Biot displacement magnitude and displacement arrows on the deformed domain, Stokes fluid pressure, and Stokes velocity magnitude with line integral convolution) for fluid injection into a fracture porous medium using the SPE10-based benchmark test. Results were obtained after $t = 10$\,[hours].}}\label{fig:exF-sols}
\end{figure}
 
\subsection{Application to interfacial flow in the eye}

\begin{figure}[!tb]
\begin{center}
\includegraphics[width=0.4\textwidth]{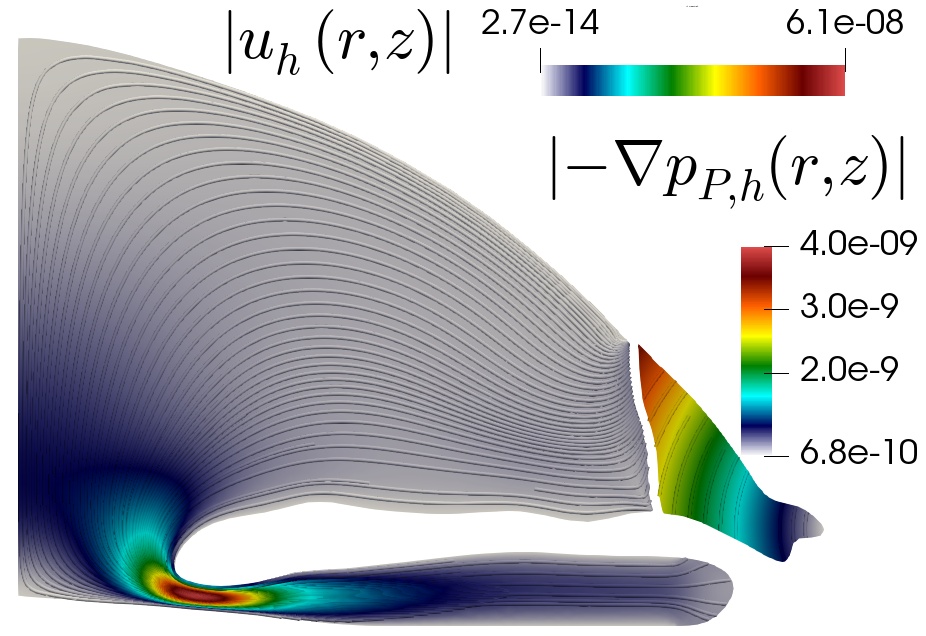}
\includegraphics[width=0.4\textwidth]{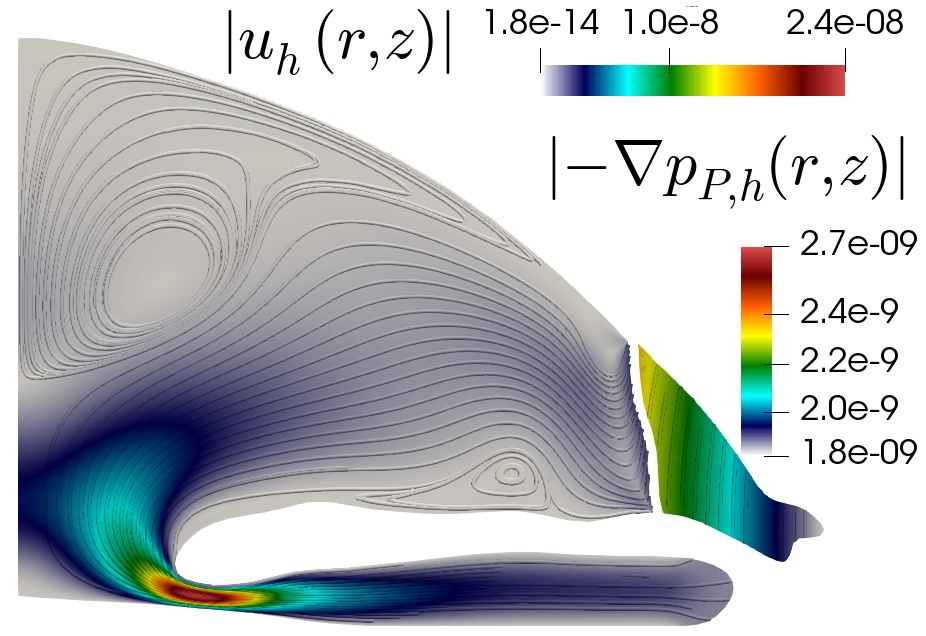}\\
\includegraphics[width=0.4\textwidth]{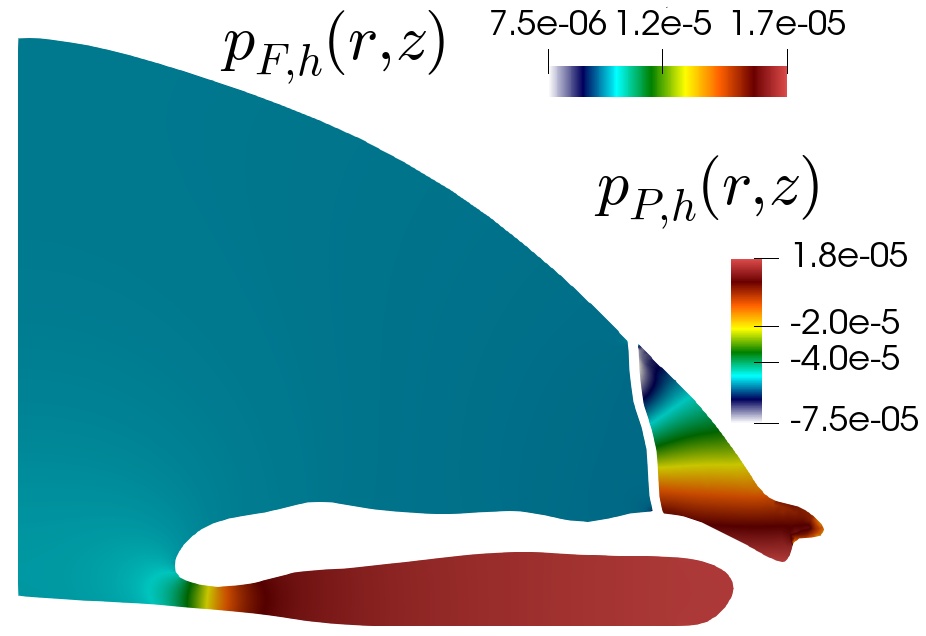}
\includegraphics[width=0.4\textwidth]{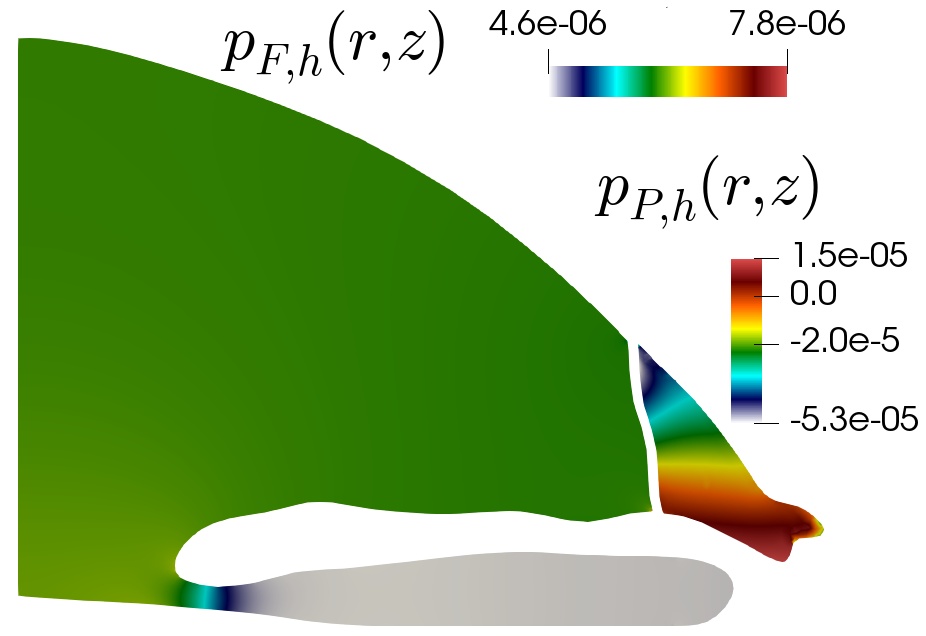}\\
\includegraphics[width=0.4\textwidth]{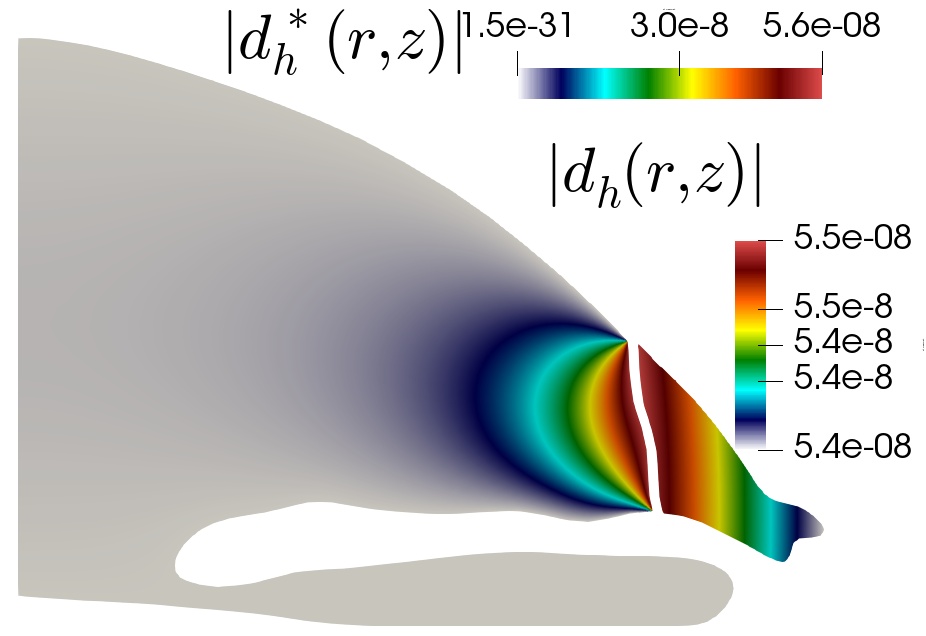}
\includegraphics[width=0.4\textwidth]{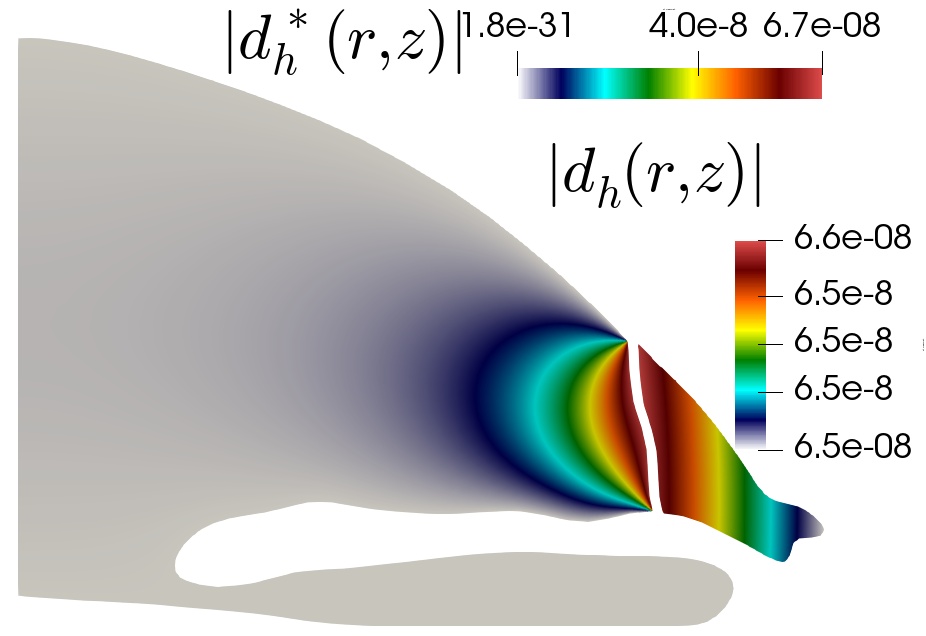}\\
\includegraphics[width=0.4\textwidth]{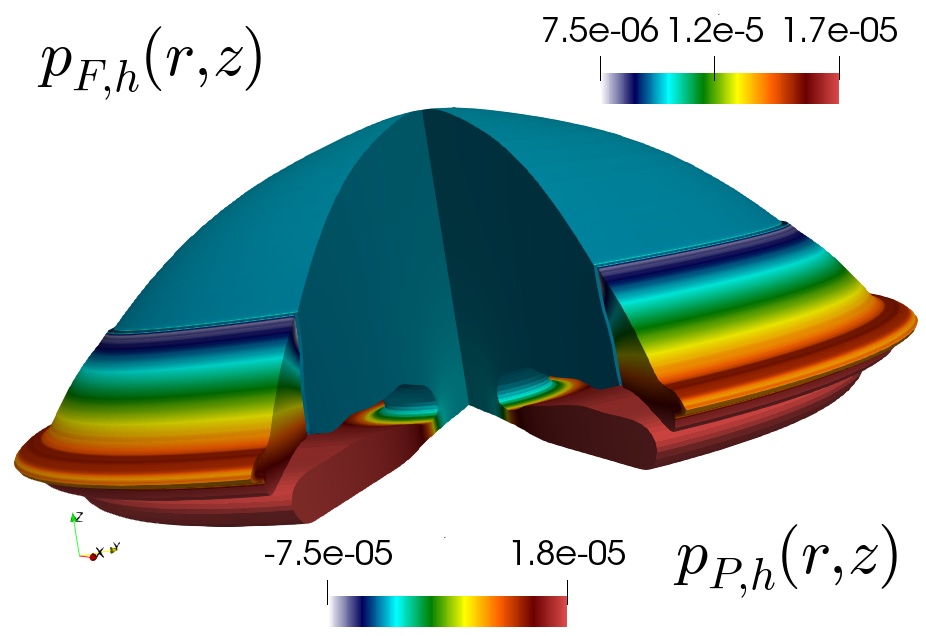}
\includegraphics[width=0.4\textwidth]{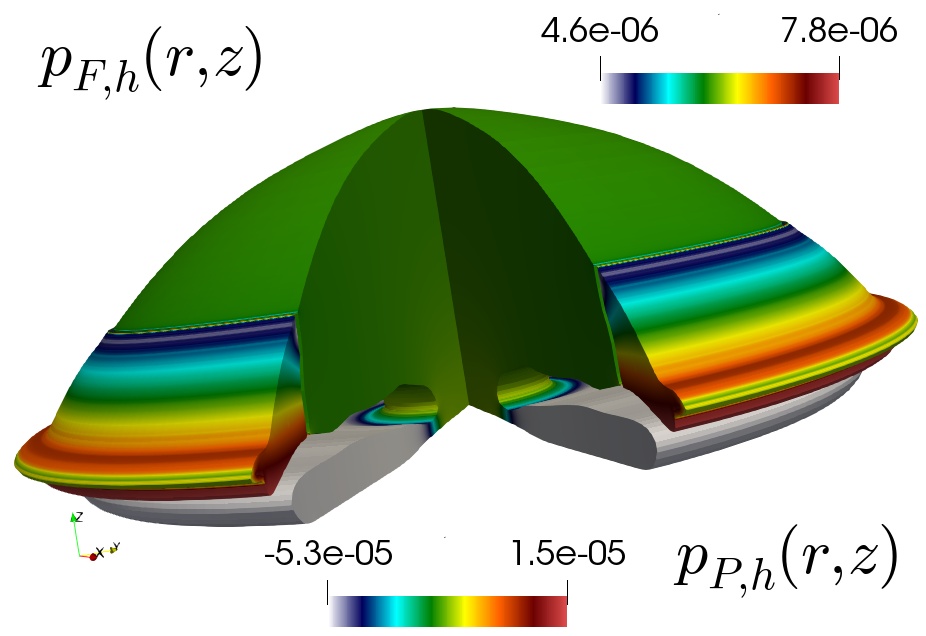}
\end{center}

\vspace{-3mm}
\caption{Example \cred{4}. Axisymmetric interfacial flow in the eye {for a linearly varying permeability profile}. Velocity magnitude and streamlines (top), fluid pressure (second row), displacement magnitude (third row), porous total pressure,  and fluid pressures on both domains extruded to the 3D case (bottom centre and bottom right). All solutions are shown at times $t = 4.2$\,[s] (left) and $t = 5$\,[s] (right column).}\label{fig:ex05}
\end{figure}

\cred{
To finalise this section we include a problem pertaining to the application of the interaction between aqueous humour in the anterior chamber and the trabecular meshwork.} Therein, one of the main driving questions is whether one can  observe deformation of the porous skeleton (and in particular of the interface) {that could drive} a rise in intra-ocular pressure. 
For this test we use the axisymmetric formulation \eqref{eq:axisym}, {we include the convective term in \eqref{eq:momentumA} and we discard gravity}. A large amount of data is available to {specialise the geometry and the mechanical properties (both fluidic and elastic) of the eye to different animal species}  \cite{cannizzo17,johnstone04,fitt06,martinez19}.  In our case (and 
consistent with a uniform temperature of $37^\circ$ and a characteristic length of $6.8\cdot10^{-3}\,[\text{m}]$) 
{we impose}  
\begin{gather*}
\rho_f = 998.7\,[\text{Kg}\cdot\text{m}^{-3}], \quad 
\mu_f = 7.5\cdot 10^{-4}\, [\text{Pa}\cdot \text{s}], \quad p_0 = 0\,[\text{Pa}],\quad 
E = 2700\,[\text{Pa}], \quad \nu = 0.47,  \\ 
\lambda 
= 14388\,[\text{Pa}], \quad \mu_s 
= 918.36\,[\text{Pa}], \quad  \alpha= 1, \quad 
\rho_s = 1102\,[\text{Kg}\cdot\text{m}^{-3}], \quad C_0=0, \quad 
\gamma = 0.1. 
\end{gather*} 
%
%

As discussed in \cite{crowder13}, an increase in the Beavers-Joseph-Saffmann friction parameter $\gamma$ leads to higher pressure differences between the trabecular meshwork and the anterior chamber. Note that even if {the choice of} $\alpha = 1$, $C_0=0$ indicate{s} that both constituents (fluid and solid) are assumed {intrinsically} incompressible, a Poisson ratio smaller than 0.5 implies compressibility of the poroelastic medium ({due to the possible rearrangement of the porosity field, i.e., the fluid escaping the medium}).

With reference to Figure \ref{fig:sketch}, the length of the interface $\Sigma$ between the trabecular meshwork and the anterior chamber is $5.7\cdot10^{-4}\,[\text{m}]$, and the length of the separation between the trabecular meshwork and angular aqueous plexus ($\Gamma^{\text{out}}$) is $3.4\cdot10^{-4}\,[\text{m}]$. 
A parabolic profile for inlet velocity with a pulsating magnitude $u_{\text{in}} = 4.89\cdot10^{-7}\sin^2(\pi t)\,[\text{m}\cdot\text{s}^{-1}]$ (that has approximately the same frequency as the heartbeat) is imposed on $\Gamma^{\mathrm{in}}$ (the magnitude is obtained from the ratio $\frac{C \mu_s \kappa}{L \mu_f}$ with $C =0.2$ integrated through the thickness and the condition is imposed through a Nitsche approach with penalty parameter equal to 1 \cred{since the relevant sub-boundary is not aligned with the axes}), no-slip conditions are prescribed essentially on the walls, and a slip condition is considered for the fluid velocity on the symmetry axis (also imposed essentially). On the outlet $\Gamma^{\text{out}}$ we impose zero fluid pressure $p_P = p_0$. We simulate the interfacial flow until $t = 5$\,[s] and use a time step of $\Delta t = 0.1$\,[s]. \cred{In this case, Taylor-Hood elements are used for the pairs [velocity, fluid pressure] and [displacement, total pressure], alongside continuous and piecewise quadratic elements for Biot fluid pressure}. Figure~\ref{fig:ex05} depicts the numerical solutions at two time instants, showing the distribution on the meridional axisymmetric domain of displacement, pressures and velocity (including the Darcy velocity in the trabecular meshwork) and illustrating the modification of flow patterns throughout the last cycle of the computation, {and using (for this first run of simulations), a constant permeability  $\kappa_0 = 5.0 \cdot 10^{-12}\,[\text{m}^2]$.  For sake of visualisation, we have also plotted the fluid pressure on both sub-domains rotationally extruded to a cut of the 3D domain.

\begin{figure}[!t]
\begin{center}
\includegraphics[height=0.25\textwidth]{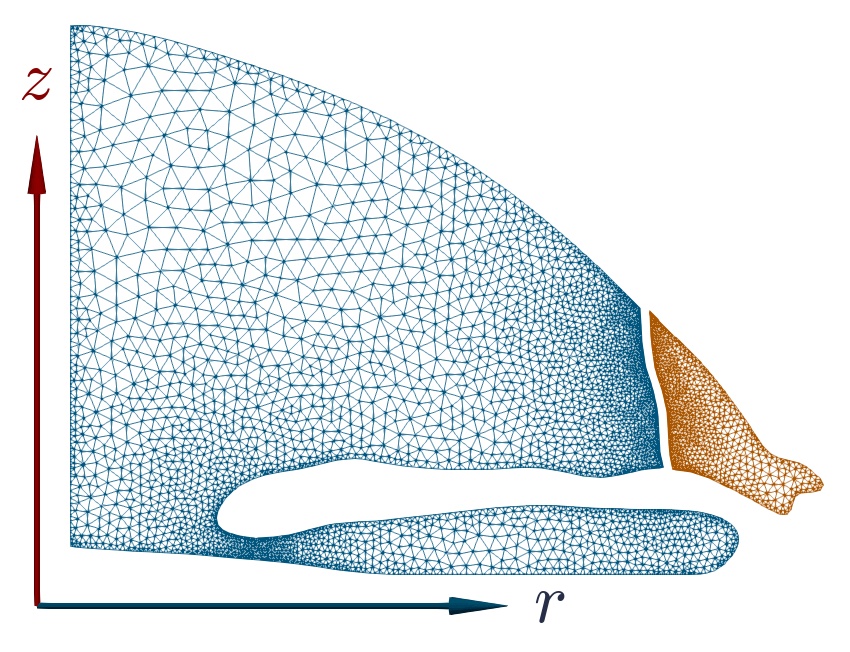} \\
\includegraphics[width=0.15\textwidth]{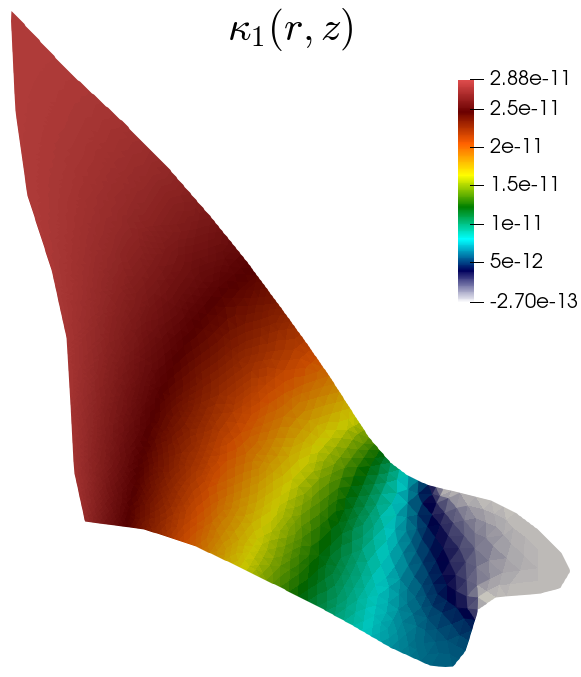}\qquad \qquad 
\includegraphics[width=0.15\textwidth]{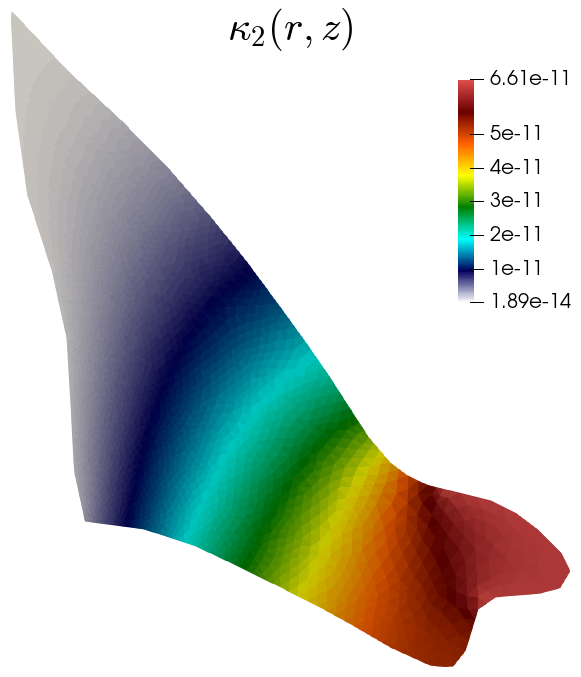}\qquad\qquad 
\includegraphics[width=0.15\textwidth]{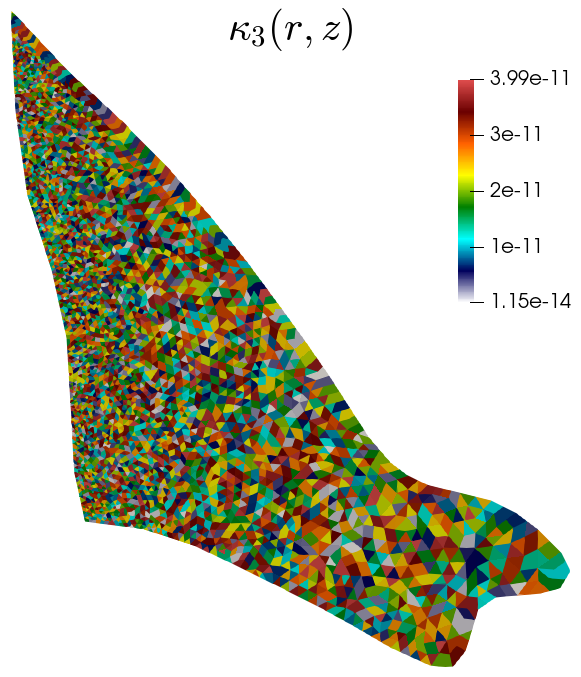}\qquad\qquad 
\includegraphics[width=0.15\textwidth]{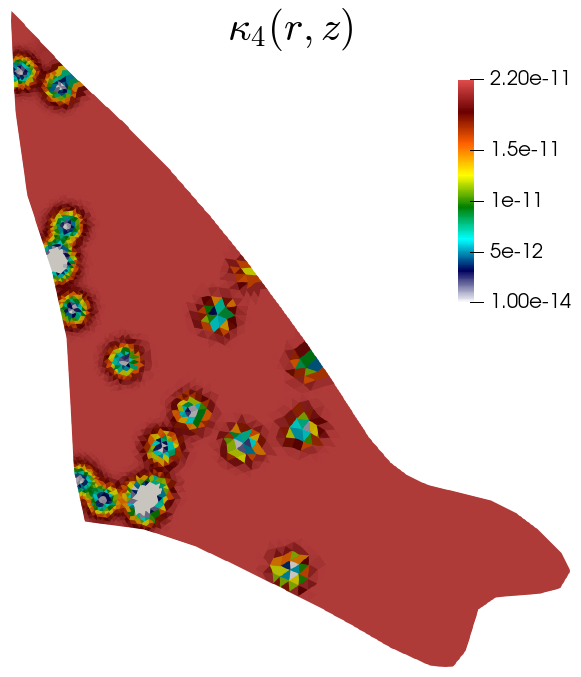}
\\
\includegraphics[width=0.24\textwidth]{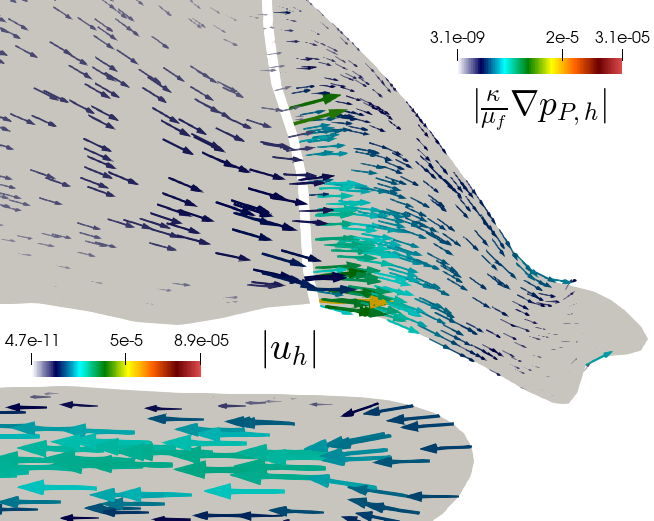}
\includegraphics[width=0.24\textwidth]{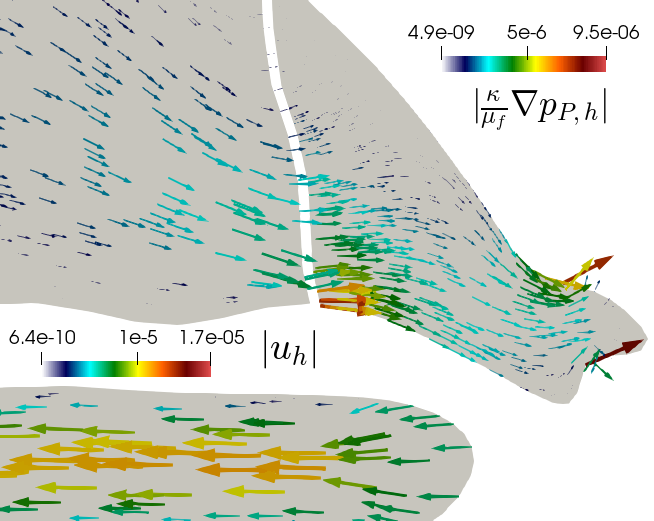}
\includegraphics[width=0.24\textwidth]{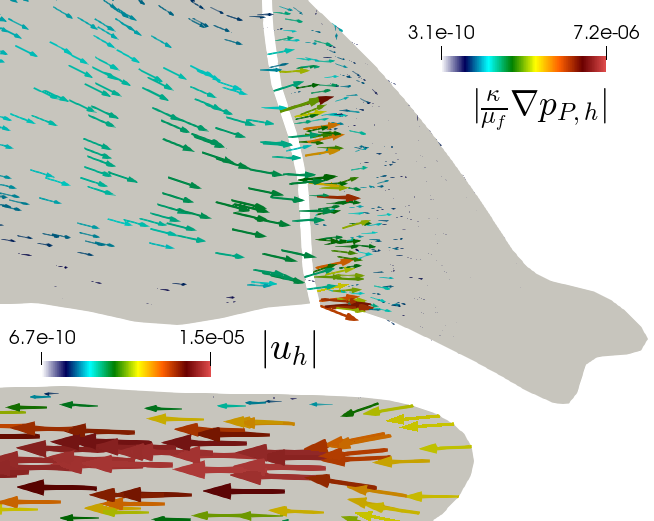}
\includegraphics[width=0.24\textwidth]{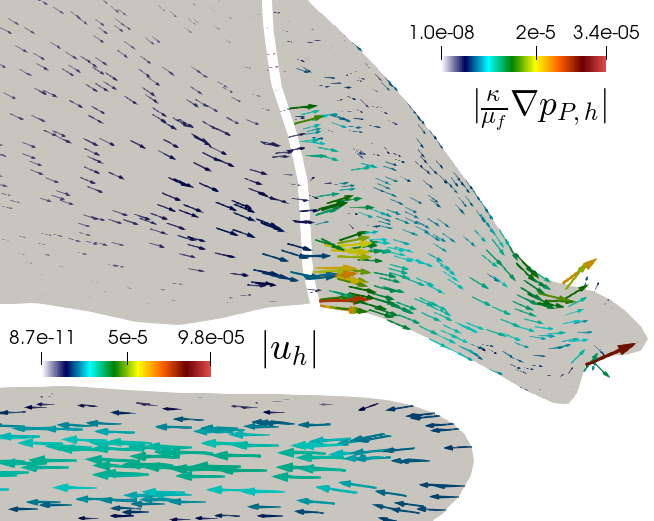}
\\[0.5ex]
\includegraphics[width=0.24\textwidth]{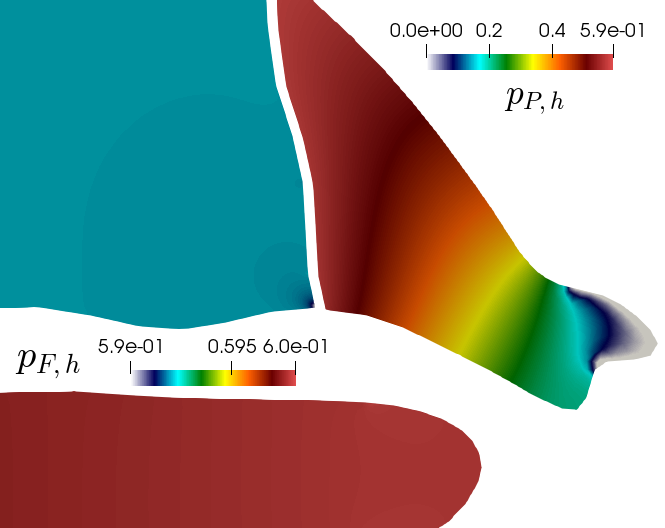}
\includegraphics[width=0.24\textwidth]{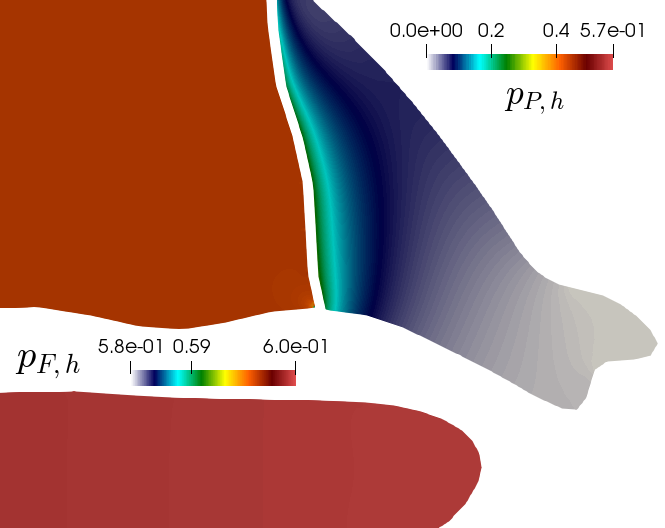}
\includegraphics[width=0.24\textwidth]{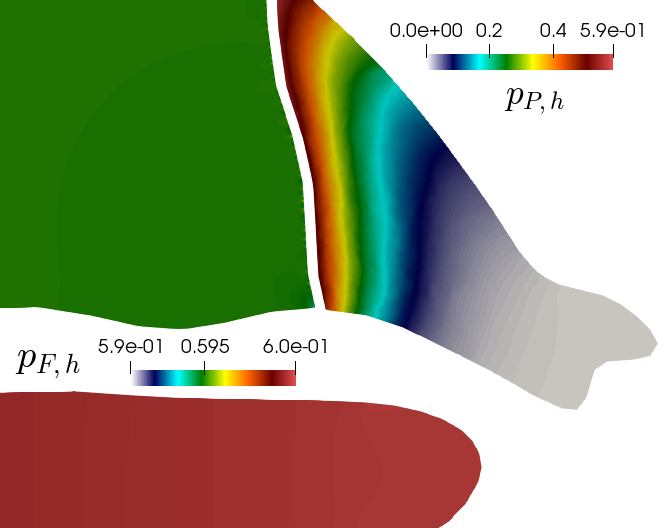}
\includegraphics[width=0.24\textwidth]{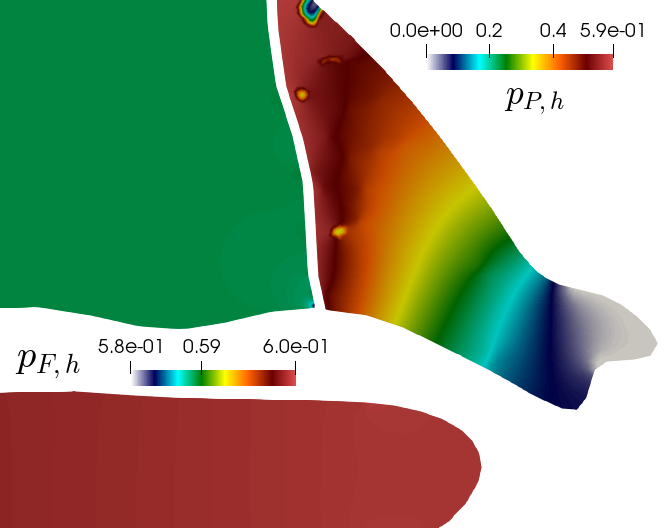}
\\[0.5ex]
\includegraphics[width=0.24\textwidth]{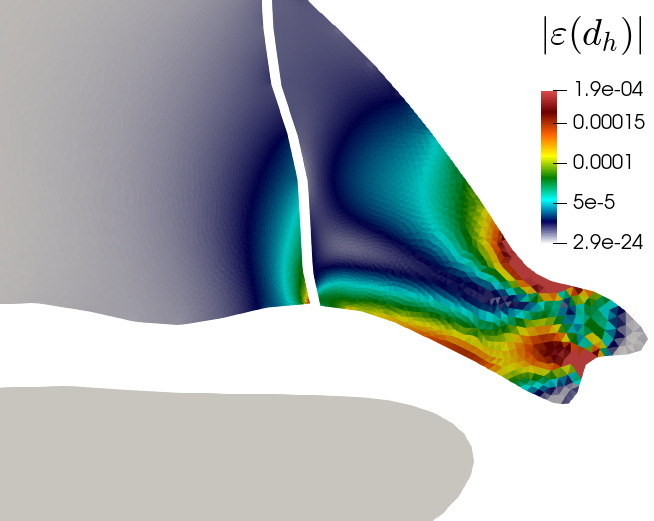}
\includegraphics[width=0.24\textwidth]{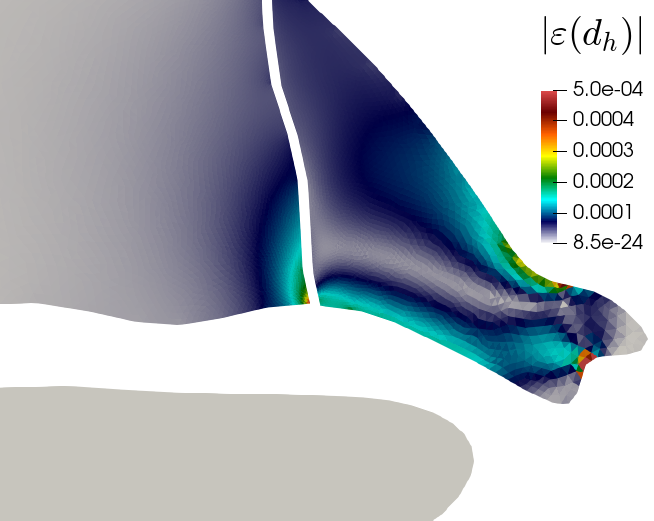}
\includegraphics[width=0.24\textwidth]{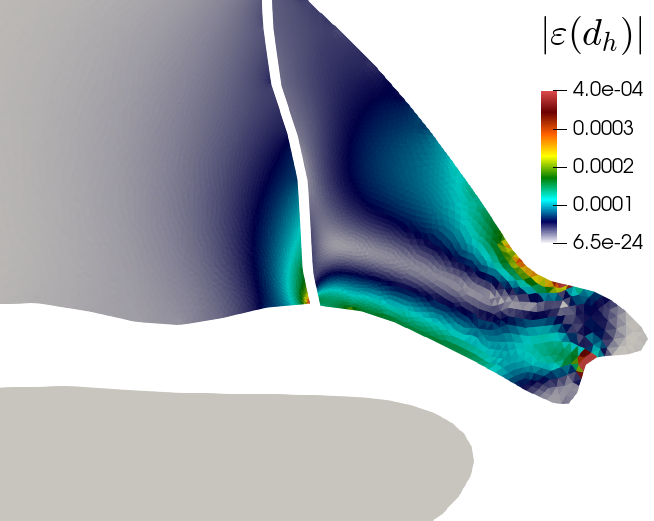}
\includegraphics[width=0.24\textwidth]{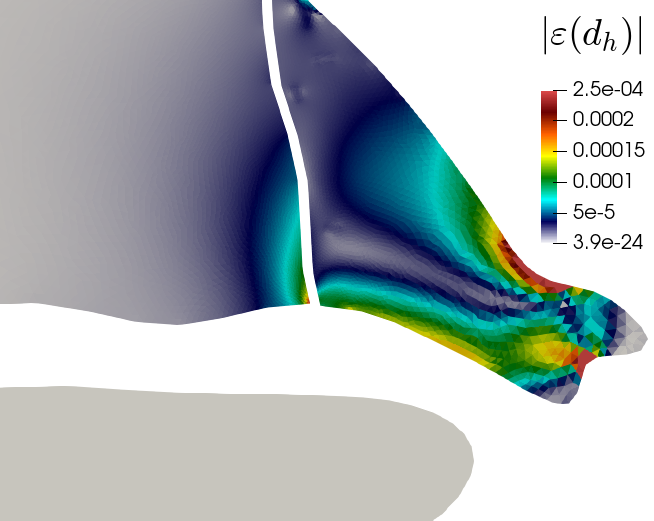}
\end{center}

\vspace{-3mm}
\caption{Example \cred{4}. Axisymmetric interfacial flow in the eye. 
\cred{Top: Coarse unstructured mesh indicating the axis of symmetry. Second row:  four permeability distributions  (in $[\text{m}^2]$) projected on piecewise constants over $\Omega_P$. Linearly decreasing (left), linearly increasing (centre-left), with a log-random distribution (centre-right), and with randomly distributed spots of smaller permeability (right). In all cases the average permeability is $\bar{\kappa}=2.0\cdot10^{-11}$.} Third-fifth rows:  
comparisons of resulting  profiles (velocity vectors, fluid pressure, and post-processed strain) at \cred{$t = 6$\,[s]}.} \label{fig:ex05-compar}
\end{figure}

In addition, we compare the behaviour produced by four heterogeneous permeability profiles 
\eqref{eq:generic-kappa}. The first case has a gradient going from \cred{$\kappa_1^{\max}=2.88\cdot 10^{-11}\,[\text{m}^2]$ on the interface, linearly down to $\kappa^{\min}=10^{-14}\,[\text{m}^2]$ on the outlet. A second synthetic permeability profile will decrease from the  value $\kappa_2^{\max}=6.55\cdot10^{-11}\,[\text{m}^2]$ on the interface down to  $\kappa^{\min}$ on the outlet. These two  profiles are generated by solving a Laplace problem with mixed boundary conditions, setting $\kappa_{i}^{\max}$ and $\kappa^{\min}$ essentially, and no-flux naturally on the remainder of $\Gamma_P$.  A third permeability distribution is 
generated  by a uniform random distribution bounded by $\kappa^{\min}$ and $\kappa_3^{\max}=3.99\cdot10^{-11}\,[\text{m}^2]$, and a fourth case is constructed by placing random points in $\Omega_P$ having permeability $\kappa^{\min}$, and $\kappa_4^{\max}=2.28\cdot10^{-11}\,[\text{m}^2]$ elsewhere (see the top panels of Figure~\ref{fig:ex05-compar}). The maximum values $\kappa_{i}^{\max}$  were tuned so that the average permeability $\bar{\kappa}_i= \frac{1}{|\Omega_P|}\int_{\Omega_P}\kappa_i(r,z) \mathrm{d}r\,\mathrm{d}z$ is equal to $2.0 \cdot 10^{-11}$ in all four cases.}

\cred{The effect of  spatial variations in  permeability are evaluated by imposing an inlet velocity profile (that is, equivalently, controlling the flow) with $\bu_{\mathrm{in}}\cdot \nn   = - 0.1 t\,[\text{m}\cdot\text{s}^{-1}]$  while prescribing zero fluid Biot pressure on $\Gamma^{\text{out}}$. From the third row of Figure~\ref{fig:ex05-compar} we see that the pressure difference (intra-ocular fluid pressure minus the pressure at the angular aqueous plexus) remains roughly the same in all the cases consistently with the fact that $\bar{\kappa}$ is kept constant. Nonetheless we see that different permeability profiles give rise to slightly different spatial distributions for velocity, for the Biot fluid pressure, and for the strain (post-processed from the Biot  displacement and from the harmonic extension Stokes displacement). Compared with the constant permeability case, that is the underlying field in the fourth permeability profile (right column), where the pressure decays linearly toward the outlet, in the decreasing permeability case (left column), the pressure gradient looks larger close to the inlet while the strain are, in general everywhere, lower; for the increasing permeability case (centre-left column) the pressure gradient is larger at the inlet and the strain are, in general everywhere, larger. Taking into consideration that here we investigate a much complex geometry, these speculations are consistent with the description provided in \cite{taffetani20} and obtained in the limit of a slender geometry. The description provided here allows to speculate that a rearrangement of the microstructure could not drive a macroscopic change in the mechanics of the system within a short time horizon, since the flow and the pressure drop remains of the same order of magnitude; however, higher pressure and strain gradients could be the cause of a local remodelling of the tissue that, in the long term and in the nonlinear regime, could lead to a collapse of part of the pectinate ligaments.}

\section{Concluding remarks}\label{sec:concl}
We have introduced a new formulation for the coupling of Biot's poroelasticity system using total pressure and free flow described by the Stokes and Navier-Stokes equations, and which does not require Lagrange multipliers to set up the interface conditions between the two subdomains. The well-posedness of the continuous problem has been proved, we have provided a rigorous stability analysis. A mixed finite element method is defined for the proposed formulation together with a corresponding reduction to the axisymmetric case and the proposed schemes are robust with respect to the first Lam\'e parameter. We have conducted the numerical validation of   spatio-temporal accuracy and \cred{have also performed some tests of applicative relevance, studying the behaviour of poromechanical filtration in subsurface hydraulic fracture with challenging heterogeneous material parameters, and simulating interfacial flow through the trabecular meshwork of the eye.} 

Different formulations we will study next include conservative discretisations on both poroelasticity and free flow \cite{boon20,hong21,kumar20}. 
{We are currently working towards the design of monolithic block preconditioners being robust with respect to all material parameters, and whose construction hinges on   adequate weighted Sobolev spaces and fractional interfacial norms.} On the other hand,  a crucial model extension corresponds to the regime of large deformations and the incorporation of   remodelling mechanisms that would better explain the progressive consolidation of the interface and the shrinkage of the trabecular meshwork and associated ciliary cleft collapse seen in canines with glaucoma (using formalisms sharing similarities with the study of lamina cribrosa thickening \cite{grytz12}). \cred{There, it is also necessary to consider nonlinear variations on permeability depending on porosity, which in turn undergoes changes due to microstructural rearrangement \cite{macminn2016}.}


\section*{Acknowledgments} This work has been partially supported by the Monash Mathematics Research Fund S05802-3951284; by the Ministry of Science and Higher Education of the Russian Federation within the framework of state support for the creation and development of World-Class Research Centers {``Digital biodesign and personalised healthcare'' No. 075-15-2020-926}; 
by Gruppo Nazio\-nale di Fisica Matematica (GNFM) of the Istituto Nazio\-nale di Alta Matematica (INdAM); by a grant from the American College of Veterinary Ophthalmologists Vision  for Animals Foundation VAFGL2017; and by NSF grant DMS 1818775.  {In addition, the authors gratefully acknowledge the many fruitful discussions with  Wietse Boon, 
Elfriede Friedmann, Miroslav Kuchta, Kent-Andr\'e Mardal,  
and Sarah L. Waters, regarding models and suitable discretisations for interfacial flow couplings.} 


\bibliographystyle{siam}
\bibliography{rtwy-refs}
\end{document}